\documentclass[twoside]{article}
\usepackage[letterpaper, left=3.5cm, right=3.5cm, top=4cm, bottom=2.5cm]{geometry}
\usepackage{graphicx}
\usepackage{amsmath,amssymb,amsthm}
\usepackage{authblk}
\usepackage[utf8]{inputenc}
\usepackage{microtype}
\usepackage{mathrsfs} 
\usepackage{enumitem}
\usepackage{calc}
\usepackage{esint}
\usepackage{fancyhdr}
\usepackage{xcolor}
\usepackage[labelfont = bf]{caption}
\usepackage{doi}
\usepackage{subfigure}
\usepackage[numbers]{natbib}
\usepackage{float}
\usepackage{stmaryrd}
\usepackage{mathtools}
\usepackage{booktabs}
\usepackage{colortbl}
\usepackage{tabulary}
\usepackage{diagbox}
\usepackage{caption}
\usepackage{scalerel}
\usepackage{paper_diening}

\DeclareMathAlphabet{\mathbbold}{U}{bbold}{m}{n}

\definecolor{mordantred19}{rgb}{0.68, 0.05, 0.0}
\definecolor{pakistangreen}{rgb}{0.0, 0.4, 0.0}

\usepackage{physics} 
\usepackage{tikz}
\usepackage{mathdots}
\usepackage{yhmath}
\usepackage{cancel}
\usepackage{color}
\usepackage{siunitx}
\usepackage{array}
\usepackage{multirow}
\usepackage{amssymb}
\usepackage{gensymb}
\usepackage{tabularx}
\usepackage{extarrows}
\usepackage{booktabs}
\usetikzlibrary{fadings}
\usetikzlibrary{patterns}
\usetikzlibrary{shadows.blur}
\usetikzlibrary{shapes}

\newtheorem{theorem}{Theorem}[section]
\numberwithin{equation}{section}
\newtheorem{proposition}[theorem]{Proposition}
\newtheorem{definition}[theorem]{Definition}
\newtheorem{corollary}[theorem]{Corollary}
\newtheorem{remark}[theorem]{Remark}
\newtheorem{lemma}[theorem]{Lemma}

\newtheorem{assumption}[theorem]{Assumption}

\usepackage{titlesec}
\titleformat{\section}{\normalfont\scshape\centering}{\thesection.}{0.5em}{}
\titleformat*{\subsection}{\itshape}
\titleformat*{\subsubsection}{\itshape}

\providecommand{\keywords}[1]
{
	{\small\emph{Keywords:} #1}
}

\providecommand{\MSC}[1]
{
	{\small\emph{AMS MSC (2020):~~} #1}
}

\definecolor{denim}{rgb}{0.08, 0.38, 0.74}
\definecolor{byzantium}{rgb}{0.44, 0.16, 0.39} 
\definecolor{shamrockgreen}{rgb}{0.0, 0.62, 0.38} 

\usepackage{hyperref}
\hypersetup{
	colorlinks=true,
	linkcolor=denim,
	citecolor = shamrockgreen,
	filecolor=magenta, 
	urlcolor=byzantium,
} 

\providecommand{\jumptmp}[2]{#1\llbracket{#2}#1\rrbracket}

\newcommand\blfootnote[1]{%
	
	\begingroup
	\renewcommand\thefootnote{}\footnote{#1}%
	\addtocounter{footnote}{-1}%
	\endgroup
}

\begin{document}
	\setlength{\abovedisplayskip}{5pt}
	\setlength{\belowdisplayskip}{5pt}
	\setlength{\abovedisplayshortskip}{5pt}
	\setlength{\belowdisplayshortskip}{5pt}

	\title{\vspace{-7.5mm}\textit{A Priori} Error Analysis for the $p$-Stokes Equations\\ with Slip Boundary Conditions:\\ A Discrete Leray Projection Framework} 
	\author[1]{Alex Kaltenbach\thanks{Email: \url{kaltenbach@math.tu-berlin.de}}}
	\author[2]{Jörn Wichmann\thanks{Email: \url{joern.wichmann@math.uni-bielefeld.de}}}
	\date{\today}
	\affil[1]{\small{Institute of Mathematics, Technical University of Berlin, Stra\ss e des 17.\ Juni 135, 10623 Berlin, GERMANY}}
    \affil[2]{\small{School of Mathematics, Monash University, 9 Rainforest Walk, Victoria 3800, AUSTRALIA}}
	\maketitle

	\pagestyle{fancy}
	\fancyhf{}
	\fancyheadoffset{0cm}
	\addtolength{\headheight}{-0.25cm}
	\renewcommand{\headrulewidth}{0pt} 
	\renewcommand{\footrulewidth}{0pt}
	\fancyhead[CO]{\textsc{Error Analysis for $p$-Stokes: A Leray Projection Framework}}
	\fancyhead[CE]{\textsc{A. Kaltenbach and Jörn Wichmann}}
	\fancyhead[R]{\thepage}
	\fancyfoot[R]{}
	
	\begin{abstract}
		We present an \emph{a priori} error analysis for the kinematic pressure in a fully-discrete finite-differences/-elements discretization of the unsteady $p$-Stokes equations, modelling non-Newtonian fluids. This system is subject to both impermeability and perfect Navier slip boundary conditions, which are incorporated either weakly via Lagrange multipliers or strongly in the discrete velocity space. A central aspect of the \emph{a priori} error  analysis is the discrete Leray projection, constructed to quantitatively approximate its continuous counterpart. The discrete Leray projection enables a Helmholtz-type decomposition at the discrete level and plays a key role in deriving error decay rates for the kinematic pressure. We derive (in some cases optimal) error decay rates for both the velocity vector field and kinematic pressure, with the error for the kinematic pressure measured in an \textit{ad hoc} norm informed by the projection framework. The \emph{a priori} error analysis remains robust even under reduced regularity of the velocity vector field and the kinematic pressure, and illustrates how the interplay of boundary conditions and projection stability governs the accuracy of pressure approximations.
	\end{abstract}
	
	
	\MSC{65M60; 76A05; 35Q35; 76D07; 65M15; 35B45.}
	
	\section{Introduction}\thispagestyle{empty}\enlargethispage{15mm}

    \hspace{5mm}In recent years, many algorithms have been proposed for the approximation of power-law or, more generally, non-Newtonian fluids, including the finite element method (\textit{cf}.~\cite{MR4048427,MR1069652,Gatica2004,MR1301740,Carelli2010,BBDR12,MR3035482,MR3020906,JessbergerKaltenbach2024}), virtual element method (\textit{cf}.~\cite{Antonietti2024}), hybrid high-order method (\textit{cf}.~\cite{MR4319602}), gradient discretization method (\textit{cf}.~\cite{2024arXiv241214316D}),~or~discontinuous Galerkin method (\textit{cf.}~\cite{KR2023-2,KR2023-3}) in combination with time-stepping schemes (\textit{cf}.~\cite{MR2511695,Le2024}) for the spatial and temporal discretizations, respectively. These algorithms construct quantitative predictions of the fluid's velocity~vector~field~and~kinematic pressure, which can, in turn, be further used to derive suitable actions~in~applications.\vspace{-4mm}\blfootnote{\textbf{Funding:} AK was supported by the School of Mathematics at Monash University through the Robert Bartnik Visiting Fellowship. JW was supported by the Australian Government through the Australian
Research Council’s Discovery Projects funding scheme (grant number DP220100937).}

In order to guarantee the accuracy of these predictions, mathematical theory is needed to control the approximation error. However, the rigorous mathematical investigation~of~these~algorithms is still incomplete: while the \emph{a priori} error analysis of the fluid's velocity vector field~is well-understood, the \emph{a priori} error analysis of the fluid's kinematic pressure is mostly unexplored for unsteady flows.\newpage

This is due to two reasons:
\begin{itemize}[noitemsep,topsep=2pt,leftmargin=!,labelwidth=\widthof{2.},font=\itshape]
    \item[1.] In contrast to the velocity vector field, for which the \textit{natural regularity} has been found in  \cite{LiuBar93rem}, the natural regularity of the kinematic pressure is still unknown. Thus, so far, there does not exist a canonical way to measure the pressure approximation error; instead,~it~is~chosen~\emph{ad~hoc}. Moreover, the fluid's acceleration vector field (which affects the dynamics of unsteady fluids only) typically lacks regularity. This lack of regularity is inherited by the kinematic pressure, limiting any \textit{a~priori} error analysis;
    \item[2.] On the continuous level, the Helmholtz decomposition and the corresponding Leray projection play a pivotal role for the factorization of the evolution equations for the  velocity vector field and the kinematic pressure. The Leray projection~\mbox{crucially} depends on the incompressibility condition and boundary~\mbox{conditions}. 
    On the discrete level both conditions might~not~be~satisfied \textit{exactly}, but 
    \textit{approximately} as typically flexibility in the construction of discrete spaces is needed. The violation~of the constraints, however, results in a conflict~of~discrete~and~continuous projections.
\end{itemize}

In this paper, we derive \emph{a priori} error estimates for the velocity vector field~and, most importantly, for the kinematic pressure of a fully-discrete finite-differences/-elements discretization of the \emph{unsteady $p$-Stokes equations}, a model for power-law~fluids, supplemented with \textit{impermeability} and \textit{perfect Navier slip} boundary conditions. 

\subsection{Boundary conditions and their discretizations} 

\hspace{5mm}The impermeability and perfect Navier slip boundary conditions prescribe the behaviour of the fluid in the normal and tangential directions at the (topological) boundary $\Gamma\coloneqq\partial\Omega$ of a computational domain $\Omega$, respectively. They correspond to: \emph{`Once fluid particles have touched the boundary of the domain, they stick to it but they can slip on its surface'}, and \emph{`This surface motion of fluid particles is independent of the viscous stress tensor'}, respectively. Slip type boundary conditions arise in various applications, such as iron melts leaving the furnace~and~polymer~melts; see, \textit{e.g.},~\cite{Hervet2003} and the references therein.

For the discretization of these boundary conditions, three methods~are~canonical: 
\begin{itemize}[noitemsep,topsep=2pt,leftmargin=!,labelwidth=\widthof{(a)},font=\itshape]
    \item[(a)] \emph{Nitsche's method~(\textit{cf}.\ \cite{Nitsche1971}):} the boundary conditions are enforced approximately~by~means~of~a penalty term in the variational formulation of the problem;
    \item[(b)] \emph{Weak imposition~(\textit{cf}.\ \cite{Verfeurth1986,Verfeurth1991}):} 
     the boundary conditions are enforced approximately~by~means of an augmented saddle-point formulation of the problem;
    \item[(c)] \emph{Strong imposition:} 
    the boundary conditions are enforced exactly by  incorporating the latter in the discrete velocity space.
\end{itemize}
More details on Nitsche's method in the context of non-Newtonian fluids can be found,~\textit{e.g.},~in~\cite{GazcaGmeinederMaringovaTscherpel2025}. However, we will restrict ourselves to the second and third~method. The former has been studied, \textit{e.g.}, in~\cite{Verfeurth1986,Verfeurth1991} for the steady Navier--Stokes equations. It enforces the boundary conditions using a Lagrange multiplier, which,~in~general,~imposes the boundary conditions not~strongly~(\textit{i.e.},~\mbox{exactly}) but weakly (\textit{i.e.},~\mbox{approximately}). While for particular choices of discrete velocity spaces the boundary conditions are matched exactly, for others the mismatch~must~be~taken~into~account. 

\subsection{Leray projection and its approximation} 

\hspace{5mm}The Leray projection is an operator that realizes the Helmholtz decomposition of a vector field into a solenoidal vector field and a gradient of a scalar function. Identifying the action of this operator on more regular vector fields with prescribed boundary conditions is a non-trivial task. Slip boundary conditions provide a special case, in which it is possible to~identify~the~action. We use this identification to construct a discrete Leray projection that converges to the continuous one   
with linear error decay rate. This quantitative convergence 
enables to derive \emph{a priori} error estimates for the velocity vector field and the kinematic pressure.

\subsection{Error analysis} 
\hspace{5mm}Our derivation of convergence rates for the approximation errors \hspace{-0.15mm}rests \hspace{-0.15mm}on \hspace{-0.15mm}two \hspace{-0.15mm}steps: \hspace{-0.15mm}first, \hspace{-0.15mm}we \hspace{-0.15mm}show \hspace{-0.15mm}that  \hspace{-0.15mm}the \hspace{-0.15mm}velocity \hspace{-0.15mm}and \hspace{-0.15mm}pressure~\hspace{-0.15mm}\mbox{approximations}~\hspace{-0.15mm}are (up to solution-dependent terms) \textit{best-approximations} of velocity and pressure~in~the~natural distance and an \textit{ad hoc} distance, respectively (\textit{cf}.\ Theorem~\ref{thm:main.velocity} and Theorem~\ref{thm:main}, respectively); and, secondly, using regularity theory for non-Newtonian fluids (\textit{cf}.~\cite{DieningRuzicka2005,BerselliRuzicka2021}), we estimate the best-approximation and solution-dependent terms with respect to the discretization parameters, eventually, obtaining \textit{optimal} error decay rates for the velocity and pressure approximations (\textit{cf}.\ Corollary \ref{cor:main.velocity}~and~Corollary~\ref{cor:pressure_rates},~respectively). 

A major advantage of this decomposition into two steps is a \textit{unified theory} even for irregular solutions: the first step is independent of the regularity of solutions and it purely depends on the model structure; and, conversely, the second step is independent of the problem and it only depends on the regularity of solutions. 
The regularity then determines the error decay rates (\textit{i.e.}, less regular solutions lead to slower convergence). 

\subsection{Structure of paper} 
\hspace{5mm}In Section~\ref{sec:prelim}, we state 
the basic notation, model assumptions, function spaces, and notion of solution,  
and derive an $L^{p'}$-integrability result for the fluid's acceleration~vector~field.
In~Section~\ref{sec:disc-problem}, we introduce the \mbox{discretization}, including the assumptions on the finite element spaces, the time discretization, and the discrete weak formulation.
The continuous and discrete Leray projections~are~introduced and discussed in Section~\ref{sec:Leray}. In Section~\ref{sec:error-estimate},~we~prove~the~main~\mbox{results}~of~this~paper, \textit{i.e.}, (quasi-)best-approximation results as well as error decay rates for the velocity vector field, 
kinematic pressure, and acceleration vector field.
In Section \ref{sec:experiments}, we complement the theoretical findings via numerical experiments.

\section{Preliminaries} \label{sec:prelim}
    \subsection{Basic notation} 
    
    \hspace{5mm}Throughout the entire paper, if not otherwise specified, let $\Omega\subseteq \mathbb{R}^d$, $d\ge 2$, be a bounded domain with polyhedral Lipschitz continuous~(topological) boundary $\Gamma\coloneqq\partial\Omega$.

    For a (Lebesgue) measurable set $\omega\subseteq \mathbb{R}^n$, $n\in \mathbb{N}$, we employ the following  notation:
    for (Lebesgue) measurable functions, vector or tensor fields ${v,w\colon \hspace{-0.15em}\omega\hspace{-0.15em}\to\hspace{-0.15em} \mathbb{R}^{\ell}}$,~${\ell\hspace{-0.15em}\in\hspace{-0.15em}\{1,d,d\hspace{-0.175em}\times\hspace{-0.175em} d\}}$, we write $(v,w)_{\omega}\coloneqq \int_{\omega}{v\odot w\,\mathrm{d}x}$, 
	whenever the integral is well-defined, where $\odot\colon \mathbb{R}^{\ell}\times \mathbb{R}^{\ell} $ $\to \mathbb{R}$ either denotes scalar multiplication,  the Euclidean or the Frobenius~inner~product. If ${\vert \omega\vert\coloneqq \int_{\omega}{1\,\mathrm{d}x}\in  (0,+\infty)}$, the average~of~an~\mbox{integrable} function, vector or tensor~field $v\colon\hspace{-0.1em} \omega\hspace{-0.1em}\to\hspace{-0.1em} \mathbb{R}^{\ell}$, $\ell\hspace{-0.1em}\in\hspace{-0.1em}\{1,d,d\times d\}$,~is~defined~by $${\langle v\rangle_\omega\coloneqq \frac{1}{\vert \omega\vert}\int_{\omega}{v\,\mathrm{d}x}}\,.$$ For $p\in [1,+\infty]$, we employ the notation 
    \begin{align*}
        \|\cdot\|_{p,\omega}\coloneqq\begin{cases}
             (\int_\omega{\vert \cdot\vert^p\,\mathrm{d}x})^{\smash{\frac{1}{p}}}&\text{ if }p\in [1,+\infty)\,,\\
             \textup{ess\,sup}_{x\in \omega}{\vert (\cdot)(x)\vert}&\text{ else}\,.
        \end{cases}
    \end{align*}
    Apart from that, we
    employ the same notation~if~$\omega$~is~replaced~by a (relatively)~open~subset~$\gamma\subseteq \Gamma$, in which case the Lebesgue measure $\mathrm{d}x$ is replaced by the~surface~\mbox{measure}~$\mathrm{d}s$.

    \subsection{Mathematical model}
\hspace{5mm}We are interested in the derivation of  \emph{a priori} error estimates for  
a fully-discrete finite-differences/-elements discretization of the unsteady $p$-Stokes equations supplemented with suitable boundary conditions.\newpage

\subsubsection{Governing equations}\vspace{-0.5mm}

\hspace{5mm} The governing equations of the unsteady $p$-Stokes equations, in a bounded time-space cylinder $\Omega_T \coloneqq  I\times \Omega$, where~$I\coloneqq (0,T)$, $T\in  (0,+\infty)$, 
for a given \emph{external force} $\bff \colon \Omega_T\to \mathbb{R}^d$ and an \emph{initial velocity vector field} $\bfv_0\colon \Omega\to \mathbb{R}^d$, 
seek  a \emph{velocity vector field} $\bfv\colon \Omega_T\to \setR^d$ and a \emph{kinematic pressure} $q\colon \Omega_T\to \setR$ such that
\begin{align}\label{eq:p-SE}
    \begin{aligned}
        \partial_t \bfv-\mathrm{div}(\bfS(\bfvarepsilon(\bfv))-q\mathbb{I}_{d\times d})&=\bff&&\quad \text{ in }\Omega_T\,,\\
        \mathrm{div}\,\bfv&=0&&\quad \text{ in }\Omega_T\,,\\ 
        \bfv(0)&=\bfv_0&&\quad \text{ in }\Omega\,.
    \end{aligned}\tag{$p$-SE}
\end{align} 
In the system \eqref{eq:p-SE}, the \emph{extra-stress tensor} $\bfS(\bfvarepsilon(\bfv))\colon \Omega_T\to\mathbb{R}^{d\times d}_{\textup{sym}}$\footnote{$\mathbb{R}^{d\times d}_{\textup{sym}}\coloneqq \{\bfA\in \mathbb{R}^{d\times d}\mid \bfA^\top =\bfA\}$.} (see  \eqref{def:A},~for~a~precise definition) depends on the \emph{strain-rate tensor}  $\smash{\bfvarepsilon(\bfv)\coloneqq \frac{1}{2}(\nabla\bfv+\nabla\bfv^\top)\colon \Omega_T\to \mathbb{R}^{d\times d}_{\textup{sym}}}$.

\subsubsection{Boundary conditions}\vspace{-0.5mm}

\hspace{5mm}If $\bfn\colon \Gamma\to \mathbb{S}^{d-1}$ denotes the outward unit normal vector field to $\Omega$, abbreviating $\Gamma_T\coloneqq I\times \Gamma$, the unsteady 
$p$-Stokes equations \eqref{eq:p-SE} are supplemented with the following boundary conditions:
\begin{itemize}[noitemsep,topsep=2pt,leftmargin=!,labelwidth=\widthof{$\bullet$}]
    \item[$\bullet$] \emph{Impermeability condition:} \hspace{-0.5mm}Fluid particles cannot pass through the (topological) boundary $\Gamma$; they \textit{`stick'} in the normal direction and can \textit{`slip'} in the tangential \mbox{direction}, \textit{i.e.},\vspace{-0.5mm}
    \begin{align}\tag{BCI}
         \bfv\cdot \bfn=0\quad \text{ on }\Gamma_T\,;\label{eq:bc.1}
    \end{align} 

    \item[$\bullet$] \emph{Perfect Navier slip condition:} Fluid particles at the (topological) boundary $\Gamma$ \textit{`slip'} in~the~tangential direction without friction ($\widehat{=}$ independent of the viscous~stress),~\textit{i.e.},\footnote{For a vector field $\mathbf{a}\colon \Gamma\to \mathbb{R}^d$,~the~\textit{tangential component} is defined by $\mathbf{a}_{\bftau}\coloneqq \mathbf{a}-(\mathbf{a}\cdot\bfn)\bfn\colon \Gamma\to \mathbb{R}^d$.}\vspace{-0.5mm}
    \begin{align}\tag{BCII}
(\bfS(\bfvarepsilon(\bfv))\bfn)_{\bftau}=\bfzero_d\quad \text{ on }\Gamma_T\,.\label{eq:bc.2}
    \end{align} 
\end{itemize}

    \subsection{Extra-stress tensor}\vspace{-0.5mm}\enlargethispage{11mm}
    
    \hspace{5mm}For the extra-stress tensor $\bfS\colon  \mathbb{R}^{d\times d}\to \mathbb{R}^{d\times d}_{\textup{sym}}$~in~the~unsteady $p$-Stokes equations \eqref{eq:p-SE},~we~assume that there exist constants  $\nu_0>0$, $\delta\ge  0$, and $p\in (1,+\infty)$ such that for every $\bfA\in \mathbb{R}^{d\times d}$, we~have that\vspace{-0.5mm}
	\begin{align}\label{def:A}
		\bfS(\bfA)\coloneqq \nu_0\,(\delta+\vert \bfA^{\textup{sym}}\vert )^{p-2}\bfA^{\textup{sym}}\,.
	\end{align} 
	For the same constants $\nu_0>0$, $\delta\ge  0$, and $p\in (1,+\infty)$ as in the definition \eqref{def:A}, we introduce the \textit{special $N$-function}~$\varphi\colon \mathbb{R}_{\ge 0}\to \mathbb{R}_{\ge 0}$, defined by\vspace{-0.5mm}
    \begin{align} 
		\label{eq:def_phi} 
		\varphi(0)\coloneqq 0\qquad\text{ and }\qquad
		\varphi'(t) \coloneqq (\delta +t)^{p-2} t\quad\text{ for all }t\ge 0\,.
	\end{align} 
	The \textit{shifted special $N$-function} $\varphi_a\colon  \mathbb{R}_{\ge 0}\to \mathbb{R}_{\ge 0}$, for every \emph{shift} $a\ge 0$,~is~defined~by
    \begin{align}
		\label{eq:phi_shifted}
		\varphi_a(0)\coloneqq 0\qquad\text{ and }\qquad
		\varphi'_a(t)\coloneqq \varphi'(a+t)\smash{\frac {t}{a+t}}\quad\text{ for all }t\ge 0\,,
	\end{align} 
    and its \emph{(Fenchel) conjugate} $(\varphi_a)^*\colon \mathbb{R}_{\ge 0}\to  \mathbb{R}_{\ge 0}$, for every shift $a\ge 0$, is defined by
    \begin{align}
    \label{eq:phi_shifted_conjugate}
        (\varphi_a)^*(s)\coloneqq \sup_{t\ge 0}{\big\{st-\varphi_a(t)\big\}}\,.
    \end{align}
	
	Next, motivated by the definition \eqref{def:A} of the extra-stress tensor $\smash{\bfS\colon\mathbb{R}^{d\times d}\to \mathbb{R}^{d\times d}_{\textup{sym}}}$, we introduce the  mapping $\bfF\colon \mathbb{R}^{d\times d}\to \mathbb{R}^{d\times d}_{\textup{sym}}$, for every $\bfA\in \mathbb{R}^{d\times d}$ defined~by\vspace{-0.5mm}
	\begin{align}
		\begin{aligned}
			\bfF(\bfA)\coloneqq (\delta+\vert \bfA^{\textup{sym}}\vert)^{\smash{\frac{p-2}{2}}}\bfA^{\textup{sym}}\,,\label{eq:def_F}
		\end{aligned}
	\end{align}
    which is related to 
	$\smash{\bfS\colon \hspace{-0.15em} \mathbb{R}^{d\times d} \hspace{-0.15em}
	 \hspace{-0.15em}\to  \hspace{-0.15em}\mathbb{R}^{d\times d}_{\textup{sym}}}$ and
	$\varphi_a,(\varphi_a)^*\colon  \hspace{-0.15em}\mathbb{R}_{\ge
		0} \hspace{-0.15em}\to \hspace{-0.15em} \mathbb{R}_{\ge
		0}$,~${a \hspace{-0.15em}\ge \hspace{-0.15em} 0}$,~via~the~\mbox{following}~\mbox{equivalences}. 
	
	\begin{lemma}\label{lem:hammer}
		For every  $\bfA, \bfB \in \mathbb{R}^{d\times d}$, we have that
		\begin{align}
            \begin{aligned} 
			(\bfS(\bfA) - \bfS(\bfB))
			\cdot(\bfA-\bfB ) &\sim \smash{\vert \bfF(\bfA) - \bfF(\bfB)\vert^2}
			\\&\sim \varphi_{\vert \bfA^{\textup{sym}} \vert }(\vert \bfA^{\textup{sym}} - \bfB^{\textup{sym}} \vert )
			\\&\sim (\varphi_{\vert \bfA^{\textup{sym}} \vert })^*(\vert \bfS(\bfA)-\bfS(\bfB)\vert )\,. 
            \end{aligned}\label{eq:hammera}
		\end{align}
	\end{lemma} 
	
	\begin{proof}
       See \cite[Lem.~6.16]{NAFSA07}.
	\end{proof}
	
    We frequently use the following $\varepsilon$-Young inequality and  $\varepsilon$-Young type result on a change of shift in the shifted special $N$-function \eqref{eq:phi_shifted} and its (Fenchel)~conjugate~\eqref{eq:phi_shifted_conjugate}.

    \begin{lemma}[$\varepsilon$-Young inequality] \label{lem:young}
For each $\varepsilon>0$, there exists a constant $c_\varepsilon \geq 1$, depending on $\varepsilon>0$, $p\in (1,+\infty)$, and $\delta\ge 0$, such that for every $t,s,a\ge 0$,~we~have~that 
\begin{align}\label{ineq:young.1}
    st &\leq c_\varepsilon (\varphi_a)^*( s)+\varepsilon \varphi_a( t )  \,.
\end{align} 
\end{lemma}

\begin{proof}
    See \cite[p.\ 107]{NAFSA07}.
\end{proof}
	
	\begin{lemma}[shift-change]\label{lem:shift-change}
		For each $\varepsilon>0$, there exists a constant $c_\varepsilon \geq 1$, depending on $\varepsilon>0$, $p\in (1,+\infty)$, and $\delta\ge 0$, such that for every $\bfA, \bfB \in \mathbb{R}^{d\times d}$~and~${r\ge 0}$, we have that
		\begin{align}
			\varphi_{\vert \bfA^{\textup{sym}} \vert}(r)&\leq c_\varepsilon\, \varphi_{\vert \bfB^{\textup{sym}} \vert }(r)
			+\varepsilon\, \vert \bfF(\bfA) - \bfF(\bfB)\vert^2\,,\label{lem:shift-change.1}
			\\
			(\varphi_{\vert \bfA^{\textup{sym}} \vert})^*(r)&\leq c_\varepsilon\, (\varphi_{\vert \bfB^{\textup{sym}} \vert })^*(r)
			+\varepsilon\, \vert \bfF(\bfA) - \bfF(\bfB)\vert^2\,.\label{lem:shift-change.3}
		\end{align}
	\end{lemma}

    \begin{proof}
        See \cite[Lems.\ 5.15, 5.18]{NAFSA07}.
    \end{proof} 

    Moreover, for a (Lebesgue) measurable set $\omega\subseteq \mathbb{R}^n$, $n\in \mathbb{N}$,  a non-negative~(possible variable) shift  $a\hspace{-0.1em}\in\hspace{-0.1em} L^p(\omega)$, and  function, vector or tensor field $v\colon \omega\hspace{-0.1em}\to\hspace{-0.1em} \mathbb{R}^\ell$,~${\ell\hspace{-0.1em}\in\hspace{-0.1em} \{1,d,d\hspace{-0.1em}\times\hspace{-0.1em} d\}}$, we introduce the \emph{modulars (with respect to $\varphi_a$ and $(\varphi_a)^*$, respectively)}
    \begin{align*}
        \rho_{\varphi_a,\omega}(v)\coloneqq \int_{\omega}{\varphi_a(\vert v\vert)\,\mathrm{d}x}\,,\qquad
        \rho_{(\varphi_a)^*,\omega}(v)\coloneqq \int_{\omega}{(\varphi_a)^*(\vert v\vert)\,\mathrm{d}x}\,,
    \end{align*}
    whenever the respective right-hand side integral is well-defined. We
    employ the same notation~if~$\omega$ is replaced by a (relatively) open set $\gamma\subseteq \Gamma$, in which case the Lebesgue measure $\mathrm{d}x$ is replaced by the surface measure $\mathrm{d}s$.

   \subsection{Function spaces} \hspace{5mm}For an arbitrary integrability index $r\in (1,+\infty)$, denoting by $L^r(\Omega)$ the Lebesgue space of $r$-integrable scalar functions,  we employ the following abbreviated notations for the vector- and  tensor-valued counterparts:
    \begin{align*} 
        \bfL^r(\Omega)&\coloneqq (L^r(\Omega))^d\,,\qquad\mathbb{L}^r(\Omega)\coloneqq (L^r(\Omega))^{d\times d}\,.
	\end{align*}
    In addition, denoting by $W^{1,r}(\Omega)$ the Sobolev space of $r$-integrable scalar functions with $r$-inte\-grable weak gradients,  we employ the following abbreviated notations~for~the~(\mbox{normal-trace-free}) vector- and tensor-valued counterparts:
    \begin{align*} 
        \bfW^{1,r}(\Omega)&\coloneqq (W^{1,r}(\Omega))^d\,,\qquad \mathbb{W}^{1,r}(\Omega)\coloneqq (W^{1,r}(\Omega))^{d\times d}\,,\\\bfW_{\bfn}^{1,r}(\Omega) &\coloneqq \big\{\mathbf{u}\in  \bfW^{1,r}(\Omega)\mid \mathbf{u}\cdot \mathbf{n}=0\text{ in }L^r(\Gamma)\big\}\,.
	\end{align*} 
    Moreover, we need the following function spaces:
    \begin{align*} 
        \bfW^{r}(\textup{div};\Omega)&\coloneqq 
        \big\{\bfu\in \bfL^r(\Omega)\mid \mathrm{div}\,\bfu\in L^r(\Omega)\big\}\,,\\
        \bfW^{r}_0(\textup{div};\Omega)&\coloneqq 
        \big\{\bfu\in \bfW^{r}(\textup{div};\Omega)\mid \bfu\cdot \bfn=0\text{ in }(W^{1-\smash{\frac{1}{r}},r}(\Gamma))^*\big\}\,,\\
        \bfW^{r}_0(\textup{div}^0;\Omega)&\coloneqq 
        \big\{\bfu\in \bfW^{r}_0(\textup{div};\Omega)\mid \mathrm{div}\,\bfu=0\text{ in }L^r(\Omega)\big\}\,,
	\end{align*} 
    where we used 
    that 
 the  normal trace operator $(\bfu\mapsto \bfu \cdot \bfn)\colon \bfW^{r}(\textup{div};\Omega)\to (W^{1-\smash{\frac{1}{r}},r}(\Gamma))^*$, defined by $\langle \bfu \cdot \bfn,\eta \rangle_{W^{1-\smash{\frac{1}{r}},r}(\Gamma)}\coloneqq (\eta,\mathrm{div}\,\bfu)_{\Omega}+(\nabla \eta,\bfu)_{\Omega}$
 for all $\bfu\in \bfW^{r}(\textup{div};\Omega)$ and $\eta\in W^{1,r}(\Omega)$, 
 is well-defined (\textit{cf}.\ \cite[Sec.\ 4.3]{EG21}). 
    In this context, 
   for every $\mu\in (W^{1-\smash{\frac{1}{r}},r}(\Gamma))^*$ and $\eta\in W^{1,r}(\Omega)$,~we abbreviate $\langle \mu,\eta\rangle_{\Gamma}\coloneqq  \langle \mu,\eta \rangle_{\smash{W^{1-\smash{\frac{1}{r}},r}(\Gamma)}}$, $\|\mu\|_{\smash{-\frac{1}{r'},r',\Gamma}}\coloneqq \|\mu\|_{\smash{(W^{1-\smash{\frac{1}{r}},r}(\Gamma))^*}}$,~and~$\|\eta\|_{\smash{1-\smash{\frac{1}{r}},r,\Gamma}}\coloneqq \|\eta\|_{\smash{W^{1-\smash{\frac{1}{r}},r}(\Gamma)}}$. 
    
    \subsection{Continuous problem} 

    \hspace{5mm}In this subsection, we introduce and discuss~the~weak formulation of \eqref{eq:p-SE}--\eqref{eq:bc.2}.~To~this~end, for a fixed power-law~index~$p\in (1,+\infty)$, 
    let us first introduce the following abbreviated notations: 
    \begin{align*}
        \widehat{\bfV}\coloneqq \bfW^{1,p}(\Omega)\,,\qquad  \widehat{Q}\coloneqq L^{\smash{p'}}(\Omega)\,,\qquad \widehat{Z}\coloneqq \smash{(W^{1-\smash{\frac{1}{p}},p}(\Gamma))^*}\,.
    \end{align*}
    In addition, we introduce abbreviated notations
  for the following linear subspaces:
     \begin{align*} 
        \bfV&\coloneqq \bfW^{1,p}_{\bfn}(\Omega)\,,\\\bfV_{\!\textup{div}}&\coloneqq \bfW^{1,p}_{\bfn}(\Omega)\cap \bfW^{p}_0(\textup{div}^0;\Omega)\,,\\
         \bfH&\coloneqq\bfW^2_0(\textup{div}^0;\Omega)\,,\\
     Q&\coloneqq \big\{\eta\in L^{\smash{p'}}(\Omega)\mid (\eta,1)_{\Omega}=0\big\}\,. 
	\end{align*}
    
    \begin{definition}[weak formulation]\label{def:weak_form}
        Let $\bff\in   L^{p'}(I;\bfL^{p'}(\Omega))$ and $\bfv_0 \in \bfH$.~Then,~a triple  $(\bfv,q,\lambda)\in   (L^p(I;\widehat{\bfV})\cap C^0(\overline{I};\bfL^2(\Omega))) \times  L^{p'}(I;Q)\times L^{p'}(I;\widehat{Z})$ is called  \emph{weak~\mbox{solution}~of~\mbox{\eqref{eq:p-SE}--\eqref{eq:bc.2}}} if 
      $\bfv(0)=\bfv_0$ in $\bfL^2(\Omega)$ and for every $(\bfxi,\eta,\mu)\in (L^p(I;\widehat{\bfV})\cap W^{1,1}(I;\bfL^2(\Omega)))\times L^{p'}(I;\widehat{Q})\times L^{p'}(I;\widehat{Z}) $ with $\bfxi(0)=\bfxi(T)=\bfzero_d$ a.e.\ in $\Omega$, there holds
	\begin{align*}
        -(\bfv,\partial_t\bfxi)_{\Omega_T}+(\bfS(\bfvarepsilon(\bfv))-q\mathbb{I}_{d\times d},\bfvarepsilon(\bfxi))_{\Omega_T}+\langle \lambda,\bfxi\cdot\bfn\rangle_{\Gamma_T} &=(\bff,\bfxi)_{\Omega_T}\,,\\
        (\eta,\mathrm{div}\,\bfv)_{\Omega_T}&=0\,,\\
        \langle \mu,\bfv\cdot \bfn\rangle_{\Gamma_T}&=0\,.
	\end{align*} 
    \end{definition}
    
    Since the weak formulation (in the sense of Definition \ref{def:weak_form}) is a saddle-point problem with a monotone system of equations for the velocity vector field and linear constraints for the kinematic pressure and normal stress component, a version of the Babuska--Lax--Milgram theorem ensures that its well-posedness is equivalent~to~the~following inf-sup stability result, corresponding to the special case $r=p$.

     \begin{lemma}[inf-sup stability for $(\widehat{\bfV},Q,\widehat{Z})$]\label{lem:inf-sup_continuous}
         Let $r\in (1,+\infty)$ and the (steady) \emph{Neumann--Laplace problem} be \emph{$W^{2,r}$-regular}, \textit{i.e.}, for every $f\in L^{r}(\Omega)$ and $g\in W^{1-\frac{1}{r},r}(\Gamma)$ with $(f,1)_{\Omega}=(g,1)_{\Gamma}$, there exists a unique function $u\in W^{2,r}(\Omega)\cap L^r_0(\Omega)$ such that 
         \begin{subequations}\label{eq:laplace-neumann}
         \begin{alignat}{2}\label{eq:laplace-neumann.1}
             -\Delta u &= f&&\quad \text{ in }\Omega\,,\\
             \nabla u\cdot \bfn &= g&&\quad \text{ in }\Gamma\,, \label{eq:laplace-neumann.2}
         \end{alignat}
         \end{subequations}
         and 
         \begin{align}\label{eq:laplace-neumann-stability}
             \|\nabla^2 u\|_{r,\Omega}\lesssim \|f\|_{r,\Omega}+\|g\|_{1-\frac{1}{r},r,\Gamma}\,, 
         \end{align}
        where the implicit constant in $\lesssim$ depends only on $r$ and $\Omega$.
         Then, for every $(\eta,\mu)\in L^{r'}_0(\Omega)\times W^{-\frac{1}{r'},r'}(\Gamma)$, there holds
         \begin{align*}
             \|\eta\|_{r',\Omega}+\|\mu\|_{-\frac{1}{r'},r',\Gamma}\lesssim \sup_{\bfxi\in \bfW^{1,r}(\Omega)\setminus \{\bfzero_d\}}{\left\{\frac{(\eta,\mathrm{div}\bfxi)_{\Omega}-\langle \mu,\bfxi\cdot \bfn\rangle_{\Gamma}}{\|\bfxi\|_{r,\Omega}+\|\nabla \bfxi\|_{r,\Omega}}\right\}}\,.
         \end{align*} 
     \end{lemma} 

     \begin{remark}\label{rem:laplace-neumann}
         The Neumann--Laplace problem \eqref{eq:laplace-neumann} is $W^{2,r}$-regular if either of the following sufficient cases is satisfied:
         \begin{itemize}[noitemsep,topsep=2pt,leftmargin=!,labelwidth=\widthof{(Case 3)}]
    \item[(Case 1)]\hypertarget{Case 1}{} $\partial \Omega$ is smooth and $g=0$ (\textit{cf}.\ \cite{Agmon1964});
    \item[(Case 2)]\hypertarget{Case 2}{} $d=2$ and $\Omega$ is convex and polygonal (\textit{cf}.\ \cite[Chap.\ 4,  Thm.\ 4.3.2.4]{Grisvard2011});
    \item[(Case 3)]\hypertarget{Case 3}{} $d \geq 3$, $\Omega$ is convex,  $r \in (1,2]$, and $g=0$ (\textit{cf}.\ \cite[Thm.\ 3.1]{Adolfsson1994}).
\end{itemize}
     \end{remark}

     \begin{proof}[Proof (of Lemma \ref{lem:inf-sup_continuous})]
         We follow along the lines of the proof of \cite[Lem.\ 3.1]{Verfeurth1986}, where the case $r=2$ is considered, up to obvious adjustments.
     \end{proof}

    \begin{remark}[Equivalent formulations]\label{rem:equiv_form}
    If, in addition, $\bfv\in W^{1,1}(I;\bfH)$, 
    then the triple 
     $(\bfv,q,\lambda )\in  (L^p(I;\widehat{\bfV})\cap C^0(\overline{I};\bfL^2(\Omega)))\times L^{p'}(I;Q)\times L^{p'}(I;\widehat{Z})$ is a weak solution~of \eqref{eq:p-SE}--\eqref{eq:bc.2}
    (in the sense of Definition \ref{def:weak_form})     
     if and only if $\bfv(0)=\bfv_0$ in $\bfL^2(\Omega)$ and 
    for every $(\bfxi,\eta,\mu)\in \widehat{\bfV}\times \widehat{Q}\times \widehat{Z}$ and a.e.\ $t\in  I$, there holds 
        \begin{align*}
            (\partial_t\bfv(t),\bfxi)_{\Omega}+(\bfS(\bfvarepsilon(\bfv)(t))-q(t)\mathbb{I}_{d\times d},\bfvarepsilon(\bfxi))_{\Omega}+\langle \lambda(t), \bfxi\cdot \bfn\rangle_{\Gamma} &= (\bff(t),\bfxi)_{\Omega}\,,\\
            (\eta,\mathrm{div}\,\bfv(t))_{\Omega}&=0\,,\\
            \langle \mu,\bfv(t)\cdot\bfn\rangle_{\Gamma}&=0\,.
        \end{align*}   
    \end{remark}

    The following result yields sufficient conditions on the data that guarantee higher temporal regularity of the velocity vector field.

    \begin{proposition}\label{prop:regularity1}
        Let $p > \frac{2d}{d+2}$ and assume that $\bfv_0\hspace{-0.1em}\in\hspace{-0.1em} \bfV$ with ${\mathrm{div}\,\bfS(\bfvarepsilon(\bfv_0))\hspace{-0.1em}\in \hspace{-0.1em}\bfL^2(\Omega)}$ and $\bff\in \smash{L^{p'}(I;\bfL^{p'}(\Omega))}\cap W^{1,2}(I;\bfL^2(\Omega))$.  
        Then, there exists a unique~weak~solution $(\bfv,q,\lambda)\in (L^p(I;\widehat{\bfV})\cap L^\infty(I;\bfL^2(\Omega)))\times L^{p'}(I;Q)\times L^{p'}(I;\widehat{Z})$ of \eqref{eq:p-SE}--\eqref{eq:bc.2} (in the sense of Definition \ref{def:weak_form})~such~that
        \begin{align*}
            \partial_t \bfv&\in L^\infty(I;\bfH)\,,\\
            \bfF(\bfvarepsilon(\bfv))&\in W^{1,2}(I;\mathbb{L}^2(\Omega))\,.
        \end{align*} 
    \end{proposition}

    \begin{proof}
        We follow along the lines of the proof \cite[Prop. 2.12]{BerselliRuzicka2021}, where 
        no-slip boundary conditions 
        are considered, up to obvious adjustments.
    \end{proof}

    From Proposition \ref{prop:regularity1}, in turn, we infer the following $L^{p'}(I;\bfL^{p'}(\Omega))$-integrability result for the fluid's acceleration vector field in the shear-thinning case~(\textit{i.e.}, $p\leq 2$).

    \begin{proposition}\label{prop:regularity2}
       Let the assumptions of Proposition \ref{prop:regularity1} be satisfied. Moreover, assume that $p\in (1,2]$ and $\bfF(\bfvarepsilon(\bfv))\in L^2(I;\mathbb{W}^{1,2}(\Omega))$. Then, we have that
        \begin{align}\label{prop:regularity2.0}
            \partial_t \bfv\in L^{p'}(I;\bfL^{p'}(\Omega))\quad\text{ if }p\ge \tfrac{-1+4d+\sqrt{9-4d+4d^2}}{3d+2}\,;
        \end{align}
        that is, $p\ge \frac{1}{8}(7+\sqrt{17})\approx 1.39$ if $d=2$ and $p\ge \frac{1}{11}(11+\sqrt{33}) \approx 1.52$ if $d=3$.
    \end{proposition}

    \begin{proof}
    By Proposition \ref{prop:regularity1}, we have that $\bfF(\bfvarepsilon(\bfv))\in W^{1,2}(I;\mathbb{L}^2(\Omega))$. This together with $\bfF(\bfvarepsilon(\bfv))\in L^2(I;\mathbb{W}^{1,2}(\Omega))$, by  real interpolation~(\textit{cf}.~\mbox{\cite[Thm.\ 33]{DieningRuzicka2005}}),~yields that $\bfF(\bfvarepsilon(\bfv))\in C^0(\overline{I};\mathbb{W}^{\frac{1}{2},2}(\Omega))$,
    which, by the fractional Sobolev embedding theorem,~gives
    \begin{align}\label{eq:regularity2.1}
        \bfF(\bfvarepsilon(\bfv))\in \smash{C^0(\overline{I};\mathbb{L}^{\smash{\frac{2d}{d-1}}}(\Omega))}\quad\Leftrightarrow\quad \bfv\in \smash{C^0(\overline{I};\mathbf{W}^{1,\smash{\frac{pd}{d-1}}}(\Omega))}\,.
    \end{align} 
    From \eqref{eq:regularity2.1} together with  $\bfF(\bfvarepsilon(\bfv))\in  \smash{L^2(I;\mathbb{W}^{1,2}(\Omega))}$, resorting to \cite[Lem.\ 4.5]{BerselliDieningRuzicka2010} (with~${s=\frac{pd}{d-1}}$), we obtain
    \begin{align}\label{eq:regularity2.2}
        \partial_t\nabla\bfv \in L^2(I;\mathbb{L}^{\kappa}(\Omega))\,,\quad\kappa \coloneqq \tfrac{2dp}{p+2d-2}\,.
    \end{align}
    On the other hand, by Propostion \ref{prop:regularity1}, we have that
    $\partial_t\bfv\in L^\infty(I;\bfH)$.~This~together~with~\eqref{eq:regularity2.2}, 
    by real interpolation (\textit{cf}.\ \cite[Thm. 33]{DieningRuzicka2005}), yields that $\partial_t\bfv\in L^{p'}(I;\mathbf{W}^{\theta,\tilde{q}}(\Omega))$, 
    where~$\smash{\frac{1}{p'}\hspace{-0.1em}=\hspace{-0.1em}\frac{\theta}{2}\hspace{-0.1em}+\hspace{-0.1em}\frac{1-\theta}{\infty}}$ and $\smash{\frac{1}{\tilde{q}}\hspace{-0.1em}=\hspace{-0.1em}\frac{\theta}{\kappa}\hspace{-0.1em}+\hspace{-0.1em}\frac{1-\theta}{2}}$, and $\mathbf{W}^{\theta,\tilde{q}}(\Omega)\coloneqq (W^{\theta,\tilde{q}}(\Omega))^d$. As a consequence, by the fractional Sobolev embedding theorem, we have that $\partial_t\bfv\hspace{-0.1em}\in\hspace{-0.1em} L^{p'}(I;\bfL^{\smash{\frac{d\tilde{q}}{d-\theta\tilde{q}}}}(\Omega))$.
    Since $p'\hspace{-0.1em}\leq\hspace{-0.1em}\smash{\frac{d\tilde{q}}{d-\theta\tilde{q}}}$ if and only~if~${p\hspace{-0.1em}\ge\hspace{-0.1em} \frac{-1+4d+\sqrt{9-4d+4d^2}}{3d+2}}$, we conclude  that \eqref{prop:regularity2.0} applies.
\end{proof}
 
    \section{Discrete problem} \label{sec:disc-problem}

 \subsection{Finite element spaces and projection operators}
 
    \hspace{5mm}We denote by $\{\mathcal{T}_h\}_{h>0}$ a family of \emph{quasi-uniform} triangulations of $\Omega$ (\textit{cf}.\ \cite[Def. 22.21]{EG21}) consisting~of $d$-dimensional simplices $T\in \mathcal{T}_h$, where $h>0$ refers to the \textit{maximal~\mbox{mesh-size}}.~The~set of boundary sides~is~defined~by $\mathcal{S}_h^{\Gamma}\coloneqq \{ K\cap K'\mid K,K'\in \mathcal{T}_h\,,\textup{dim}_{\mathcal{H}}( K\cap  K')=d-1\}$\footnote{Here, $\textup{dim}_{\mathcal{H}}(\cdot)$ refers to the \emph{Hausdorff dimension}.}.\newpage
    
	Given
	$\ell \in \mathbb N_0$, we denote by $\mathbb{P}^\ell(\mathcal{T}_h)$ (or $\mathbb{P}^\ell(\mathcal{S}_h^{\Gamma})$) the space of   scalar functions~that~are polynomials of degree at most $\ell$ on each simplex $K\in  \mathcal{T}_h$~(or~facet~$S\in \mathcal{S}_h^{\Gamma}$), and set ${\mathbb{P}^\ell_c(\mathcal{T}_h) \coloneqq  \mathbb{P}^\ell(\mathcal{T}_h)\cap C^0(\overline{\Omega})}$. 
    Then, given $\ell_{\bfv}\in\mathbb{N}$ and $\ell_{q},\ell_{\lambda} \in \mathbb{N}_0$, we denote by
	\begin{align}\label{eq:fem_choice} 
			\smash{\widehat{\bfV}_h\subseteq (\mathbb{P}^{\ell_\bfv}_c(\mathcal{T}_h))^d\,,\qquad
			\widehat{Q}_h\subseteq \mathbb{P}^{\ell_q}(\mathcal{T}_h)\,,\qquad
            \widehat{Z}_h\subseteq \mathbb{P}^{\ell_\lambda}(\mathcal{S}_h^{\Gamma})\,,} 
	\end{align}
    finite element spaces such that for the linear subspace
    \begin{align*}
        \begin{aligned}
            \bfV_h&\coloneqq 
            \begin{cases}
            \smash{\widehat{\bfV}_h}\cap \bfV&\text{ if \eqref{eq:bc.1} is  strongly imposed}\,,\\
                \big\{ \bfxi_h \in \smash{\widehat{\bfV}_h}\mid \forall\mu_h\in \smash{\widehat{Z}_h}\colon (\mu_h,\bfxi_h\cdot \bfn)_{\Gamma}=0\big\}&\text{ if \eqref{eq:bc.1} is weakly imposed}\,,
            \end{cases}\\[-0.5mm]
            \bfV_{h,\textup{div}}&\coloneqq \big\{ \bfxi_h \in \bfV_h\mid \forall\eta_h\in \widehat{Q}_h\colon (\eta_h,\mathrm{div}\,\bfxi_h)_{\Omega}=0 \big\}\,,\\[-0.5mm]
            Q_h&\coloneqq \big\{\eta_h\in \widehat{Q}_h\mid (\eta_h,1)_{\Omega}=0\big\}\,,
        \end{aligned}
    \end{align*}
   where we always set $\widehat{Z}_h\coloneqq\{0\}$ in the case that the impermeability condition \eqref{eq:bc.1} is strongly imposed,  the following set of assumptions is satisfied: 

    The first assumption ensures the coercivity of the extra-stress tensor (\textit{cf}.\ \eqref{def:A}).

    \begin{assumption}[Korn's \hspace{-0.15mm}inequality]\label{ass:korn}
       \hspace{-0.35em}We \hspace{-0.15mm}assume \hspace{-0.15mm}that 
        \hspace{-0.15mm}for \hspace{-0.15mm}every \hspace{-0.15mm}$\bfxi_h\hspace{-0.175em}\in \hspace{-0.175em}\bfV_h$,~\hspace{-0.15mm}there~\hspace{-0.15mm}holds
        \begin{align*}
           \|\bfxi_h\|_{p,\Omega}+ \|\nabla \bfxi_h\|_{p,\Omega}\lesssim 
            \|\bfvarepsilon(\bfxi_h)\|_{p,\Omega}\,.
        \end{align*}
    \end{assumption}

    \begin{remark} 
    \label{rem:korn}
        Assumption \ref{ass:korn} is satisfied if either of the following cases is satisfied:
        \begin{itemize}[noitemsep,topsep=2pt,leftmargin=!,labelwidth=\widthof{(iii)}]
            \item[(i)] \hypertarget{ass:korn.i}{}$\bfV_h\subseteq (\mathbb{P}^k_c(\mathcal{T}_h))^d/\mathcal{R}(\Omega)$, where $\mathcal{R}(\Omega)
            \coloneqq 
            \big\{ \bfA (\cdot)+\bfb\colon \Omega\to \mathbb{R}^d\mid \bfA\in \mathbb{R}^{d\times d}\text{ with }\bfA^\top=-\bfA\,,$ $\bfb\in \mathbb{R}^d\big\}$   
            is the \textit{space of rigid deformations} (\textit{cf}.\ \cite{Necas66});
            \item[(ii)] \hypertarget{ass:korn.ii}{}$\bfV_h\coloneqq \widehat{\bfV}_h\cap\bfV$ 
            (\textit{cf}.\ \cite[Thm.\ 3.2]{GazcaGmeinederMaringovaTscherpel2025});
            \item[(iii)]  \hypertarget{ass:korn.iii}{}$\mathbb{P}^1(\mathcal{S}_h^{\Gamma})\subseteq \widehat{Z}_h$ (\textit{cf}.\ Lemma \ref{lem:normal_trace_estimate}\eqref{lem:normal_trace_estimate.2}).
        \end{itemize}
    \end{remark}
 
    If $\mathbb{P}^1(\mathcal{S}_h^{\Gamma})\subseteq \widehat{Z}_h$, then we have a Korn type inequality for the space $\bfV_h$. 

       \begin{lemma}\label{lem:normal_trace_estimate}
    If $\mathbb{P}^1(\mathcal{S}_h^{\Gamma})\subseteq \widehat{Z}_h$, then  for every $\bfxi_h\in \bfV_h$ and $S\in \mathcal{S}_h^{\Gamma}$, we have that 
    \begin{align}\label{lem:normal_trace_estimate.1}
        \smash{h\,\rho_{\varphi_a,S}(h^{-1}\bfxi_h\cdot \bfn)\lesssim 
        \rho_{\varphi_a,\omega_S}(\bfvarepsilon(\bfxi_h))\,,}
    \end{align}
    where $\omega_S\in \mathcal{T}_h$ with $S\subseteq\partial\omega_S$  
    and~the~implicit~\mbox{constant} in $\lesssim$ depends on $p$, $\delta$, $\Omega$, and the choice of finite~element~spaces~\eqref{eq:fem_choice}. In particular, for every $\bfxi_h\in \bfV_h$, we~have~that
    \begin{align}\label{lem:normal_trace_estimate.2}
        \|\bfxi_h\|_{p,\Omega}+\|\nabla\bfxi_h\|_{p,\Omega}\lesssim \|\bfvarepsilon(\bfxi_h)\|_{p,\Omega}\,.
    \end{align}
\end{lemma}

\begin{proof}
    \emph{ad \eqref{lem:normal_trace_estimate.1}.} Due to $(\mu_h,\bfxi_h\cdot \bfn)_{\Gamma}=0$ for all $\mu_h\in  \mathbb{P}^1(\mathcal{S}_h^{\Gamma})$, for every $S\in \mathcal{S}_h^{\Gamma}$, we have that $\pi_h^{\smash{1,S}}(\bfxi_h\cdot \bfn)=0$  on $S$,  
    where $\pi_h^{\smash{1,S}}\colon L^1(S)\to \mathbb{P}^1(S)$ is the~\mbox{$L^2$-projection}.~Hence, resorting \hspace{-0.15mm}to \hspace{-0.15mm}an \hspace{-0.15mm}inverse \hspace{-0.15mm}estimate \hspace{-0.15mm}(\textit{cf}.\ \hspace{-0.15mm}\cite[Lem.\ \hspace{-0.15mm}12.1]{EG21}),~\hspace{-0.15mm}the~\hspace{-0.15mm}\mbox{approximation}~\hspace{-0.15mm}\mbox{properties}~\hspace{-0.15mm}of $\pi_h^{1,S}$  \hspace{-0.15mm}(\textit{cf}.\ \cite[Thm.\ 18.16]{EG21}), and a discrete trace inequality (\textit{cf}.\ \cite[Lem.\ 12.8]{EG21}),~we~\mbox{obtain}
    \begin{align}\label{lem:normal_trace_estimate.3}
        \begin{aligned} \|\bfxi_h\cdot \bfn\|_{\infty,S} 
        &\lesssim \smash{\vert S\vert^{-1}\|\bfxi_h\cdot \bfn-\pi_h^{\smash{1,S}}(\bfxi_h\cdot \bfn)\|_{1,S}} 
        \\&
        \lesssim \smash{h^2\,\vert \omega_S\vert^{-1}\|\nabla^2  \bfxi_h\|_{1,\omega_S}}
        \\&\lesssim \smash{h\, \vert \omega_S\vert^{-1}\| \bfvarepsilon(\bfxi_h)\|_{1,\omega_S}}\,,
         \end{aligned}
    \end{align}
    where we used in the last step that $\smash{\frac{\smash{\partial^2}}{\partial x_k\partial x_\ell}=\frac{\partial\bfvarepsilon_{j\ell}}{\partial x_k}+\frac{\partial\bfvarepsilon_{jk}}{\partial x_\ell}-\frac{\partial\bfvarepsilon_{k\ell}}{\partial x_j}}$~for~all~${j,k,\ell\in \{1,\ldots,d\}}$ and an inverse estimate (\textit{cf}.\ \cite[Lem.\ 12.1]{EG21}). Finally,
    by Jensen's inequality,~from~\eqref{lem:normal_trace_estimate.3}, we conclude that the claimed trace inequality \eqref{lem:normal_trace_estimate.1} applies.

    \emph{ad \eqref{lem:normal_trace_estimate.2}.} On the one hand, by \eqref{lem:normal_trace_estimate.1} (in the case $\delta=0$), we have that
    \begin{align}\label{lem:normal_trace_estimate.4}
        \smash{\|\bfxi_h\cdot \bfn\|_{p,\Gamma}\lesssim h^{\smash{\frac{1}{p'}}}\|\bfvarepsilon(\bfxi_h)\|_{p,\Omega}\,.}
    \end{align} 
    On the other hand, by \cite[Thm.\ 3.2]{GazcaGmeinederMaringovaTscherpel2025}, we have that
    \begin{align}\label{lem:normal_trace_estimate.5}
        \|\bfxi_h\|_{p,\Omega}+\|\nabla\bfxi_h\|_{p,\Omega}\lesssim \|\bfvarepsilon(\bfxi_h)\|_{p,\Omega}+\|\bfxi_h\cdot \bfn\|_{p,\Gamma}\,.
    \end{align}
    Eventually, if we combine \eqref{lem:normal_trace_estimate.4} and \eqref{lem:normal_trace_estimate.5}, we arrive at the claimed estimate \eqref{lem:normal_trace_estimate.2}.~
\end{proof}

    If only $\mathbb{P}^0(\mathcal{S}_h^{\Gamma})\subseteq \widehat{Z}_h$, we have at least a Poincar\'e type inequality for the space $\bfV_h$.

      \begin{lemma}\label{lem:poincare}
    If $\mathbb{P}^0(\mathcal{S}_h^{\Gamma})\subseteq \widehat{Z}_h$, then  for every $\bfxi_h\in \bfV_h$ and $S\in \mathcal{S}_h^{\Gamma}$, we have that
    \begin{align}\label{lem:poincare.1}
        \smash{h\,\rho_{\varphi_a,S}(h^{-1}\bfxi_h\cdot \bfn)\lesssim 
        \rho_{\varphi_a,\omega_S}(\nabla\bfxi_h)\,,}
    \end{align}
    where the implicit constant  in $\lesssim$ depends on $p$, $\delta$, $\Omega$, and the choice of finite element spaces \eqref{eq:fem_choice}. In particular, for every $\bfxi_h\in \bfV_h$, we have that
    \begin{align}\label{lem:poincare.2}
        \|\bfxi_h\|_{p,\Omega}\lesssim \|\nabla\bfxi_h\|_{p,\Omega}\,.
    \end{align}
\end{lemma}

    \begin{proof}
        We argue similarly to the proof of Lemma \ref{lem:normal_trace_estimate} up to minor adjustments,~\textit{e.g.}, replacing $\pi_h^{1,S}\colon \ L^1(S)\to \mathbb{P}^1(S)$ by the $L^2$-projection ${\pi_h^{0,S}\colon \hspace{-0.175em}L^1(S)\to\mathbb{P}^0(S)}$~for~all~${S\in  \mathcal{S}_h^{\Gamma}}$.
    \end{proof}
    
    The next two assumptions ensure the approximability of $(\bfV,Q)$~by~$\{(\bfV_h,Q_h)\}_{h>0}$.
	
	\begin{assumption}[Projection operator $\Pi_h^Q$]
		\label{ass:PiQ}
		We assume that $\setR
		\subseteq \widehat{Q}_h$ and that there exists a linear projection operator
		$\Pi_h^Q\colon L^1(\Omega) \to  \widehat{Q}_h$, which is \emph{locally $L^1$-stable}, \textit{i.e.}, for every $\eta\in L^1(\Omega)$ and $K\in \mathcal{T}_h$,~there holds
		\begin{align}
			\label{eq:PiQstab}
			\|\Pi_h^Q \eta\|_{1,K} \lesssim  \|\eta\|_{1,\omega_K}\,,
		\end{align}
        where $\omega_K\coloneqq \bigcup \{K'\in \mathcal{T}_h\mid \partial K\cap \partial K'\neq \emptyset\}$ is the \emph{patch (surrounding $K$)}. 
	\end{assumption}
	
	\begin{assumption}[Projection operator $\Pi_h^\bfV$]\label{ass:proj-div}
		\hspace{-0.2em}We assume that $\mathbb{P}^1_c(\mathcal{T}_h)  \subseteq  \widehat{\bfV}_h$~and~that there
		\hspace{-0.15mm}exists \hspace{-0.15mm}a \hspace{-0.15mm}linear \hspace{-0.15mm}projection \hspace{-0.15mm}operator \hspace{-0.15mm}$\Pi_h^\bfV \colon \bfW^{1,1}(\Omega)  \to \widehat{\bfV}_h$ \hspace{-0.15mm}with \hspace{-0.15mm}the \hspace{-0.15mm}following~\hspace{-0.15mm}\mbox{properties}:
		\begin{itemize}[noitemsep,topsep=2pt,leftmargin=!,labelwidth=\widthof{(iii)}]
			\item[(i)] \emph{Preservation of divergence in $Q_h^*$:} For every $\bfxi \in  \bfW^{1,1}(\Omega)$~and~$\eta_h  \in Q_h$,~there~holds
			\begin{align}
				\label{eq:div_preserving}
				(\eta_h,\mathrm{div}\,\bfxi)_{\Omega} = (\eta_h,\mathrm{div}\,\Pi_h^\bfV
				\bfxi)_{\Omega} \,;
			\end{align}
			\item[(ii)] \emph{Preservation of homogeneous normal Dirichlet boundary values:} $\Pi_h^{\bfV}(\bfV) \subseteq \bfV_h\cap \bfV$;
			\item[(iii)] \emph{Local $\bfL^1$-$\bfW^{1,1}$-stability:} For every $\bfxi \in \bfW^{1,1}(\Omega)$ and $K\in \mathcal{T}_h$, there holds
			\begin{align}
				\label{eq:Pidivcont}
				\|\Pi_h^{\bfV}\bfxi\|_{1,K} &\lesssim \|\bfxi\|_{1,\omega_K} + \textup{diam}(K) \|\nabla \bfxi\|_{1,\omega_K}\,.
			\end{align}
		\end{itemize}
	\end{assumption}

    Assumption \ref{ass:proj-div} implies the  discrete inf-sup stability for the couple $(\bfV_h\cap \bfV,Q_h)$.

    \begin{lemma}[Discrete inf-sup stability for $(\bfV_h\cap \bfV,Q_h)$]\label{lem:discrete_infsup_I}
        Let Assumption \ref{ass:proj-div} be satisfied and $r\in (1,+\infty)$. Then, for every $\eta_h\in Q_h$, we have that
        \begin{align*}
            \|\eta_h\|_{r',\Omega}\lesssim \sup_{\bfxi_h\in (\bfV_h\cap \bfV)\setminus\{\bfzero_d\}}{\bigg\{\frac{(\eta_h,\mathrm{div}\,\bfxi_h)_{\Omega}}{\|\nabla \bfxi_h\|_{r,\Omega}}\bigg\}}\,,
        \end{align*}
        where the implicit constant in $\lesssim$ depends on $r$, $\Omega$, and the discrete~spaces~\eqref{eq:fem_choice}.
    \end{lemma}

    \begin{proof} 
        We follow along the lines of the proof  of \cite[Lem.\ 4.1]{BBDR12}, replacing \mbox{\cite[Ass. 2.9]{BBDR12}} by Assumption \ref{ass:proj-div} in doing so. 
    \end{proof}

    \begin{remark} 
   \hspace{-0.15mm}For \hspace{-0.15mm}a \hspace{-0.15mm}list \hspace{-0.15mm}of \hspace{-0.15mm}discrete \hspace{-0.15mm}spaces \hspace{-0.15mm}\eqref{eq:fem_choice} \hspace{-0.15mm}that
         \hspace{-0.15mm}meet \hspace{-0.15mm}Assumptions~\hspace{-0.15mm}\ref{ass:PiQ}~\hspace{-0.15mm}and~\hspace{-0.15mm}\ref{ass:proj-div}, we~refer~to~\cite{ET} (see also \cite[p.~23]{GazcaGmeinederMaringovaTscherpel2025}). 
    \end{remark}

    Later, we will introduce a discrete formulation that mimics the weak formulation (\textit{cf}. Definition~\ref{def:weak_form}) in seeking the velocity vector field, kinematic pressure, and normal stress~component~separately. Thus, the discrete inf-sup stability of~the~couple~$(\bfV_h\cap \bfV,Q_h)$~is~not~enough;~instead, we need the discrete inf-sup stability of the triple $(\smash{\widehat{\bfV}}_h,Q_h,\widehat{Z}_h)$.
    
      \begin{assumption}[Discrete \hspace{-0.15mm}inf-sup \hspace{-0.15mm}stability \hspace{-0.15mm}for \hspace{-0.15mm}$(\smash{\widehat{\bfV}}_h,Q_h,\smash{\widehat{Z}}_h)$]\label{ass:discrete_inf_sup_II} 
      \hspace{-0.375em}For~\hspace{-0.15mm}every~\hspace{-0.15mm}${r\hspace{-0.175em}\in\hspace{-0.175em} (1,\hspace{-0.15mm}+\infty)}$, we assume that
        for every $(\eta_h,\mu_h)\in Q_h\times  \widehat{Z}_h$, there holds
        \begin{align*}
            \|\eta_h\|_{r',\Omega}+\|\mu_h\|_{-\frac{1}{r'},r',\Gamma}
            \lesssim \sup_{\bfxi_h\in \widehat{\bfV}_h\setminus\{\bfzero_d\}}{\bigg\{\frac{(\eta_h, \mathrm{div}\,\bfxi_h)_{\Omega}-(\mu_h, \bfxi_h\cdot\bfn)_{\Gamma}}{\|\bfxi_h\|_{r,\Omega}+\|\nabla \bfxi_h\|_{r,\Omega}}\bigg\}}\,.
        \end{align*}
    \end{assumption} 

    Similar to Assumption \ref{ass:korn} (\textit{cf}.\ Remark \ref{rem:korn}), Assumption \ref{ass:discrete_inf_sup_II} is met in generic cases.\vspace{-0.5mm}
    
    \begin{remark}
    \label{rem:discrete_inf_sup_II} 
    If $(\smash{\widehat{\bfV}}_h,\smash{\widehat{Q}}_h)$ are such that Assumption \ref{ass:PiQ} and Assumption \ref{ass:proj-div} are satisfied, then Assumption \ref{ass:discrete_inf_sup_II} is satisfied if either of the following cases is satisfied:
        \begin{itemize}[noitemsep,topsep=2pt,leftmargin=!,labelwidth=\widthof{(iii)}]
            \item[(i)] \hypertarget{rem:discrete_inf_sup_II.i}{} $\mathbb{B}_{\mathscr{F}}^{\Gamma}(\mathcal{T}_h)/\mathcal{R}(\Omega)\subseteq \bfV_h 
            $ (\textit{cf}.\ \cite[Prop.\ 4.3, for the case $r=2$]{Verfeurth1986}), where $\mathbb{B}_{\mathscr{F}}^{\Gamma}(\mathcal{T}_h)$ is the boundary facet bubble function space (\textit{cf}.\ \cite[(4.4)\&(4.5)]{Verfeurth1986});
            \item[(ii)] \hypertarget{rem:discrete_inf_sup_II.ii}{} $\bfV_h\coloneqq \widehat{\bfV}_h\cap \bfV$  and $\smash{\widehat{Z}}_h\coloneqq \{0\}$ (\textit{cf}.\ Lemma \ref{lem:discrete_infsup_I}). 
        \end{itemize}
    \end{remark}

	\subsection{Temporal discretization}\label{subsec:time_discretization}\enlargethispage{2mm}\vspace{-0.5mm}
	
	\hspace{5mm}
    For a number of time steps $M\hspace{-0.1em}\in\hspace{-0.1em}\mathbb{N}$, time 
    step size $\tau\hspace{-0.1em}\coloneqq \hspace{-0.1em}\frac{T}{M}$, time steps $t_m\hspace{-0.1em}\coloneqq\hspace{-0.1em} \tau\,m$,~${I_m\hspace{-0.1em}\coloneqq\hspace{-0.1em} \left(t_{m-1},t_m\right]}$, $m =1,\ldots,M$,~$\mathcal{I}_\tau \coloneqq \{I_m\}_{m=1,\ldots,M}$, and $\mathcal{I}_\tau^0 \coloneqq \mathcal{I}_\tau\cup\{I_0\}$,~where~$I_0\coloneqq(t_{-1},t_0]\coloneqq (-\tau,0]$.

    Then, given a (real) Banach space $X$, we denote by 
	\begin{align*}
		\mathbb{P}^0(\mathcal{I}_\tau;X)&\coloneqq \big\{f\colon I\to X\mid f(s)
		=f(t)\text{ in }X\text{ for all }t,s\in I_m\,,\;m=1,\ldots,M\big\}\,,\\
        \mathbb{P}^0(\mathcal{I}_\tau^0;X)&\coloneqq \big\{f\colon I\to X\mid f(s)
		=f(t)\text{ in }X\text{ for all }t,s\in I_m\,,\;m=0,\ldots,M\big\}\,,
	\end{align*}
    \textit{the \hspace{-0.15mm}spaces \hspace{-0.15mm}of \hspace{-0.15mm}$X$-valued \hspace{-0.15mm}functions \hspace{-0.15mm}temporally \hspace{-0.15mm}piece-wise \hspace{-0.15mm}constant \hspace{-0.15mm}(with \hspace{-0.15mm}respect~\hspace{-0.15mm}to~\hspace{-0.15mm}$\mathcal{I}_\tau$~\hspace{-0.15mm}and~\hspace{-0.15mm}$\mathcal{I}_\tau^0$, respectively) functions}. For every $f^\tau\in 	\mathbb{P}^0(\mathcal{I}_\tau^0;X)$, 
		 the \textit{backward difference quotient}~${\mathrm{d}_\tau f^\tau\in \mathbb{P}^0(\mathcal{I}_\tau;X)}$ is defined by\vspace{-0.5mm}
		\begin{align*}
		\smash{\mathrm{d}_\tau f^\tau|_{I_m}\coloneqq \tfrac{1}{\tau}(f^\tau(t_m)-f^\tau(t_{m-1}))\quad\text{ in }X \quad\text{ for all }m=1,\ldots,M\,.}
		\end{align*} 
The \textit{temporal (local) $L^2$-projection operator} $\Pi^0_{\tau}\colon L^1(I;X)\to \mathbb{P}^0(\mathcal{I}_\tau;X)$, for~every~$f\in L^1(I;X)$, is defined by\vspace{-0.5mm}
	\begin{align}\label{def:Pit} \smash{\Pi^0_{\tau}f|_{I_m}\coloneqq \tfrac{1}{\tau}(f,1)_{I_m}\quad\textup{ in }X\quad \text{ for all }m=1,\ldots,M\,.}
	\end{align}
    The \textit{temporal (nodal) interpolation operator} $\mathrm{I}^0_{\tau}\colon C^0(\overline{I};X)\to \mathbb{P}^0(\mathcal{I}_\tau^0;X)$, for every $f\in C^0(\overline{I};X)$, is defined by\vspace{-0.5mm}
	\begin{align}\label{def:Pit}
		\smash{\mathrm{I}^0_{\tau}f|_{I_m}\coloneqq  f(t_m)\quad\textup{ in }X\quad \text{ for all }m=0,\ldots,M\,.}
	\end{align} 

    \subsection{Discrete weak formulation}\vspace{-0.5mm}

\hspace{5mm}In this subsection, we~introduce~the~discrete counterpart to the 
weak formulation (in the sense of Definition \ref{def:weak_form}):\vspace{-0.5mm}

\begin{definition}[Discrete formulation]\label{def:discrete_form}
    Let $\bff^\tau \coloneqq  \Pi_\tau^0\bff\in \mathbb{P}^0(\mathcal{I}_\tau;\bfL^{p'}(\Omega))$ and $\bfv_h^0\coloneqq \mathcal{P}_h \bfv\in \bfV_{h,\textup{div}}$. 
    Then, a triple $(\bfv_h^\tau,q_h^\tau,\lambda_h^\tau)\hspace{-0.15em}\in\hspace{-0.15em} \mathbb{P}^0(\mathcal{I}_\tau^0;\widehat{\bfV}_h)\hspace{-0.15em}\times \hspace{-0.15em}\mathbb{P}^0(\mathcal{I}_\tau;Q_h)\hspace{-0.15em}\times\hspace{-0.15em} \mathbb{P}^0(\mathcal{I}_\tau;\widehat{Z}_h)$ 
    is called \emph{discrete~\mbox{solution}~of~\eqref{eq:p-SE}--\eqref{eq:bc.2}} if   $\bfv_h^\tau(0)=\bfv_0^h$ in $\bfV_{h,\mathrm{div}}$ and  for every  ${(\bfxi_h^{\tau},\eta_h^\tau,\mu_h^\tau)\in \mathbb{P}^0(\mathcal{I}_\tau;\widehat{\bfV}_h)\times \mathbb{P}^0(\mathcal{I}_\tau;\widehat{Q}_h)\times \mathbb{P}^0(\mathcal{I}_\tau;\widehat{Z}_h)}$, there holds\vspace{-0.5mm}
	\begin{align*}
		(\mathrm{d}_\tau\bfv_h^\tau,\bfxi_h^{\tau} )_{\Omega_T}+
		(\bfS(\bfvarepsilon(\bfv_h^\tau))-q_h^\tau\mathbb{I}_{d\times d},\bfvarepsilon(\bfxi_h^{\tau}))_{\Omega_T}+(\lambda_h^\tau,\bfxi_h^\tau\cdot \bfn)_{\Gamma_T}&=(\bff^\tau,\bfxi_h^{\tau})_{\Omega_T}\,,\\
       (\eta_h^\tau,\mathrm{div}\,\bfv_h^{\tau})_{\Omega_T}&=0\,, \\
       (\mu_h^\tau,\bfv_h^{\tau}\cdot\bfn)_{\Gamma_T}&=0\,.
	\end{align*}
\end{definition} 

As in the continuous case, where the well-posedness of the weak formulation is equivalent to the inf-sup stability of~$(\widehat{\bfV},Q,\widehat{Z})$ (\textit{cf}. Lemma~\ref{lem:inf-sup_continuous}), the well-posedness of the discrete formulation is equivalent to the discrete inf-sup stability of~$(\smash{\widehat{\bfV}}_h,Q_h,\smash{\widehat{Z}}_h)$ (\textit{cf}. Assumption~\ref{ass:discrete_inf_sup_II}).\vspace{-0.5mm}

 \begin{remark}[Equivalent discrete formulation]\label{rem:equiv_discrete_form}
   A triple
     $(\bfv_h^\tau,q_h^\tau,\lambda_h^\tau)\in \mathbb{P}^0(\mathcal{I}_\tau^0;\widehat{\bfV}_h)$ $ \times \mathbb{P}^0(\mathcal{I}_\tau;Q_h) \times \mathbb{P}^0(\mathcal{I}_\tau;\widehat{Z}_h)$ is a discrete solution of \eqref{eq:p-SE}--\eqref{eq:bc.2} (in the sense~of~Definition~\ref{def:discrete_form}) if~and~only~if   $\bfv_h^\tau(0)=\bfv_0^h$ in $\bfV_{h,\mathrm{div}}$ and
    for every $(\bfxi_h,\eta_h,\mu_h) \in \widehat{\bfV}_h\times \widehat{Q}_h\times \widehat{Z}_h$ and a.e.\ ${t\in  I}$,~there~holds 
        \begin{align*}
            (\mathrm{d}_\tau\bfv_h^\tau(t),\bfxi_h)_{\Omega}+(\bfS(\bfvarepsilon(\bfv_h^\tau)(t))-q_h^\tau(t)\mathbb{I}_{d\times d},\bfvarepsilon(\bfxi_h))_{\Omega}+(\lambda_h^\tau(t),\bfxi_h\cdot \bfn)_{\Gamma}&= (\bff^\tau(t),\bfxi_h)_{\Omega}\,,\\
 (\eta_h,\mathrm{div}\,\bfv_h^{\tau}(t))_{\Omega}&=0\,, \\
       (\mu_h,\bfv_h^{\tau}(t)\cdot\bfn)_{\Gamma}&=0\,.
        \end{align*}  
    \end{remark}

\section{(Discrete) Leray projection} \label{sec:Leray}
\hspace{5mm}Even though the analytic tools for the analysis of fluid flow equations, such as the Helmholtz decomposition and the Leray projection, have been well-studied, much less is known about the tools for the discretized equations. In particular, the discrete Leray projection plays a pivotal role within the error analysis of the kinematic pressure, but lacks systematic investigation.  

In this section, we derive an explicit representation for the discrete Leray projection; we discuss how its (possible) Lebesgue-stability directly implies its~\mbox{Sobolev-stability}; and we show its quantified convergence to the continuous Leray projection. In addition, we recall some classical results on the continuous Leray projection. 

\subsection{Leray projection on $\bfL^2$-integrable vector fields}
\hspace{5mm}We~start~by~\mbox{defining} the (continuous) Leray projection and Helmholtz decomposition~in~the context of $\bfL^2(\Omega)$. Subsequently, we introduce the discrete Leray projection in an analogous~\mbox{fashion}.\enlargethispage{2.5mm}

\subsubsection{Continuous case}
\hspace{5mm}To begin with, we briefly recall some classical results; further details can be found, \textit{e.g.}, in \cite[Chap. 2, Sect. 3]{MR1855030}.

Let $\mathcal{P}\colon \bfL^2(\Omega)\to \bfH$ be the \emph{(continuous) Leray projection}, \textit{i.e.},~the~orthogonal~projection onto incompressible vector fields with vanishing normal trace, defined by
\begin{align}\label{eq:l2-ort-proj}
    \forall \bfxi \in \bfH\colon \quad ( \bfu - \mathcal{P}\bfu, \bfxi)_{\Omega} = 0\,, \qquad \bfu \in \bfL^2(\Omega)\,.
\end{align}
Then, the \emph{complementary Leray projection} is defined by
$\mathcal{P}^\perp  \coloneqq \identity - \mathcal{P}\colon  \bfL^2(\Omega) \to \bfH^\perp$, where $ \bfH^\perp \coloneqq \{ \bfu \in \bfL^2(\Omega)| \, \forall \bfxi \in \bfH\colon\,( \bfu, \bfxi )_{\Omega} =0 \}$.

Next, let $\Delta_N^{-1}\mathrm{div}\colon \bfL^2(\Omega)\to W^{1,2}(\Omega)\cap L^2_0(\Omega)$ denote the solution operator of the \emph{Neumann--Laplace \hspace{-0.1mm}problem} \hspace{-0.1mm}with \hspace{-0.1mm}right-hand \hspace{-0.1mm}side \hspace{-0.1mm}in \hspace{-0.1mm}divergence \hspace{-0.1mm}form, \hspace{-0.1mm}\textit{i.e.},~\hspace{-0.1mm}given~\hspace{-0.1mm}${\bfg\hspace{-0.175em} \in \hspace{-0.175em}\bfL^2(\Omega)}$,~if~${u\hspace{-0.175em}\in\hspace{-0.175em} W^{1,2}(\Omega)\hspace{-0.175em}\cap \hspace{-0.175em}L^2_0(\Omega)}$  denotes the unique solution of
\begin{align}\label{eq:Neumann-Laplace}
    \forall v\in W^{1,2}(\Omega)\cap L^2_0(\Omega)\colon\qquad (\nabla u,\nabla v)_{\Omega}=(\bfg,\nabla v)_{\Omega}\,,
\end{align}
we define $\Delta_N^{-1}\mathrm{div}\,\bfg \coloneqq u\in W^{1,2}(\Omega)\cap L^2_0(\Omega)$.

By means of the \emph{inverse Neumann--Laplacian}  $\Delta_N^{-1}\mathrm{div}$,  representation formulas for~the~Leray~projection $\mathcal{P}$ and the complementary Leray projection $\mathcal{P}^\perp$ can be derived. In fact, it~is~\mbox{readily}~seen~that\footnote{For normed vector spaces $X$, $Y$, by $\mathcal{L}(X;Y)$ we denote the \emph{space of linear and bounded operators}, equipped with the \emph{operator norm} $\norm{A}_{\mathcal{L}(X;Y)} \coloneqq\sup_{x \in X \backslash \{0\}}{\big\{\tfrac{\|Ax\|_Y}{\|x\|_X}\big\}}$ for $A\in \mathcal{L}(X;Y)$.}\vspace{-4.5mm}
\begin{subequations} \label{eq:leray}
\begin{alignat}{2} \label{eq:leray.1}
    \mathcal{P} &= \identity - \nabla \Delta_{N}^{-1} \mathrm{div}&&\quad \text{ in }\mathcal{L}(\bfL^2(\Omega);\bfH) \,, \\
    \mathcal{P}^\perp &= \nabla \Delta_{N}^{-1} \mathrm{div}&&\quad\text{ in }\mathcal{L}(\bfL^2(\Omega);\bfH^\perp)\,.\label{eq:leray.2}
\end{alignat}
\end{subequations}  
Moreover, by observing that the right-hand side of \eqref{eq:leray.2}~returns~\mbox{gradients},~the~following~ortho-gonal decomposition --called \textit{Helmholtz decomposition}--~is~a~direct~consequence: 
\begin{align} \label{eq:Helmholtz}
 \bfL^2(\Omega) = \bfW^{2}_0(\textup{div}^0;\Omega) \,\oplus \, \nabla W^{1,2}(\Omega)\,.
\end{align}


\subsubsection{Discrete case}\enlargethispage{7.5mm}

\hspace{5mm}Let $\mathcal{P}_h\colon \bfL^2(\Omega) \to \bfV_{h,\mathrm{div}}$ be the \emph{discrete Leray projection}, \textit{i.e.}, the orthogonal projection onto $\bfV_{h,\mathrm{div}}$, defined by
\begin{align}\label{eq:dis-ort-proj}
    \forall \bfxi_h \in \bfV_{h,\mathrm{div}}\colon \quad ( \bfu - \mathcal{P}_h\bfu, \bfxi_h )_{\Omega} = 0\,, \qquad \bfu \in \bfL^2(\Omega)\,.
\end{align}
Then, the \textit{complementary discrete Leray projection} is defined by $\mathcal{P}^\perp_h\hspace{-0.15em}\coloneqq \hspace{-0.15em}\identity - \mathcal{P}_h\colon \hspace{-0.15em}\bfL^2(\Omega)\hspace{-0.15em} \to ( \bfV_{h,\mathrm{div}})^\perp$. Moreover, let $\mathcal{P}_{\bfV_h}\colon \bfL^2(\Omega)\to \bfV_h$ be the orthogonal projection on $\bfV_h$.

By analogy with  \eqref{eq:leray}, our aim is to represent the discrete Leray projection~$\mathcal{P}_h$ and complementary discrete Leray projecion $\mathcal{P}_h^\perp$ in terms of discrete differential and solution operators.\newpage 

Let the \emph{discrete gradient} $\nabla^h\colon L^2(\Omega) \to \bfV_h$, \emph{discrete divergence} $\smash{\mathrm{div}}^h\colon \bfL^2(\Omega) \to Q_h$, and \emph{discrete inverse Neumann--Laplacian} $(\Delta_{N}^h)^{-1}\colon L^2(\Omega) \to Q_h$ be defined by
\begin{subequations}\label{def:discrete_operators}
\begin{alignat}{3}
    &\forall \bfv_h \in \bfV_h\colon \hspace{5em} \,\, (\nabla^h q, \bfv_h)_{\Omega} &&= - (q, \mathrm{div}\,\bfv_h)_{\Omega}\,, &&\qquad q \in L^2(\Omega)\,,\label{def:discrete_gradient} \\[-0.25mm]
    &\forall q_h \in Q_h\colon \hspace{5em}  (\smash{\mathrm{div}}^h \bfu,q_h )_{\Omega} &&= -( \bfu,  \nabla^h q_h)_{\Omega}\,, &&\qquad \bfu \in \bfL^2(\Omega)\,, \label{def:discrete_divergence}\\[-0.25mm]
    &\forall q_h \in Q_h\colon \quad (\nabla^h (\Delta_{N}^h)^{-1}q, \nabla^h q_h)_{\Omega} &&= -(q, q_h)_{\Omega}\,, &&\qquad q \in L^2(\Omega)\,,\label{def:discrete_inverse_neumann}
\end{alignat}
\end{subequations}
respectively. \hspace{-0.1mm}First \hspace{-0.1mm}of \hspace{-0.1mm}all, 
\hspace{-0.1mm}the \hspace{-0.1mm}well-posedness \hspace{-0.1mm}of \hspace{-0.1mm}$\nabla^h$ \hspace{-0.1mm}and \hspace{-0.1mm}$\smash{\mathrm{div}}^h$ \hspace{-0.1mm}is \hspace{-0.1mm}evident. \hspace{-0.1mm}The \hspace{-0.1mm}discrete~\hspace{-0.1mm}\mbox{inf-sup}~\hspace{-0.1mm}stability of the couple $(\bfV_h\cap \bfV,Q_h)$ (\textit{cf}.\ Lemma \ref{lem:discrete_infsup_I})~implies~that~$\nabla^h|_{Q_h}$~is~\mbox{injective}. In fact,~for~${r\in (1,+\infty)}$, from Lemma \ref{lem:discrete_infsup_I}, \eqref{def:discrete_gradient}, and Lemma \ref{lem:poincare}, it readily follows that
\begin{align}\label{eq:discrete_poincare_pressure}
   \smash{\forall q_h\in Q_h\colon\qquad \|q_h\|_{r',\Omega}\lesssim \|\nabla^h q_h\|_{r',\Omega}\,.}
\end{align}
As a result, 
$(\nabla^h\, \cdot\,, \nabla^h \,\cdot\,)_{\Omega}$ is an inner product on~$Q_h$, which ensures the~well-posedness~of~
$\smash{(\Delta_{N}^h)^{-1}}$.

With the help of the discrete gradient $\nabla^h$~(\textit{cf}.~\eqref{def:discrete_gradient}), we can parametrize the discrete~ortho-gonal complement of $\bfV_{h,\mathrm{div}}$, which is the content of the following lemma.

\begin{lemma} \label{lem:para-ortho-comp}
The discrete gradient $\nabla^h$ is a bijection from $Q_h$ to $( \bfV_{h,\mathrm{div}})^\perp \cap \bfV_h$.
\end{lemma}
\begin{proof}
We need to show that\vspace{-0.5mm}
\begin{itemize}[noitemsep,topsep=2pt,leftmargin=!,labelwidth=\widthof{a)},font=\itshape]
    \item[a)]\hypertarget{lem:para-ortho-comp.a}{} $\nabla^h\colon Q_h\to \nabla^h Q_h$ is injective and, thus, bijective;
    \item[b)]\hypertarget{lem:para-ortho-comp.b}{} $\nabla^h Q_h \subseteq ( \bfV_{h,\mathrm{div}})^\perp \cap \bfV_h$;
    \item[c)]\hypertarget{lem:para-ortho-comp.c}{} $\nabla^h Q_h \supseteq ( \bfV_{h,\mathrm{div}})^\perp \cap \bfV_h$.
\end{itemize} 

\textit{ad \hyperlink{lem:para-ortho-comp.a}{a})}
From the discrete Poincar\'e type inequality \eqref{eq:discrete_poincare_pressure}, it follows that $\nabla^h$~is~injective~and,~thus,  is bijective onto its range $\nabla^h Q_h$.

\textit{ad \hyperlink{lem:para-ortho-comp.b}{b})}  Let $q_h \in Q_h$ be arbitrary. We need to show that $\nabla^h q_h \in \bfV_h$ and $\nabla^h q_h \in ( \bfV_{h,\mathrm{div}})^\perp$,
which is each a direct consequence from the definition~of~$\nabla^h$~(\textit{cf}.~\eqref{def:discrete_gradient}).

\textit{ad \hyperlink{lem:para-ortho-comp.c}{c})}  Instead of showing $\nabla^h Q_h \supseteq (\bfV_{h,\mathrm{div}})^\perp \cap \bfV_h$ directly, we show that $(\nabla^h Q_h)^\perp \cap \bfV_h \subseteq \bfV_{h,\mathrm{div}}$. Once we have verified the latter inclusion, the former follows from the following identities: $ \bfV_{h,\mathrm{div}} \, \oplus \, (( \bfV_{h,\mathrm{div}})^\perp \cap \bfV_h) =  \bfV_h = ((\nabla^h Q_h)^\perp \cap \bfV_h) \, \oplus \, \nabla^h Q_h$. 

Let $\bfv_h \hspace{-0.1em}\in  \hspace{-0.1em}(\nabla^h Q_h)^\perp \cap \bfV_h$ be fixed, but arbitrary. Then, we have that $(\bfv_h, \nabla^h q_h)_{\Omega} \hspace{-0.1em}= \hspace{-0.1em}0$~for~all~$q_h \hspace{-0.1em}\in \hspace{-0.1em} Q_h$. Applying the definition of~$\nabla^h$ (\textit{cf}.\ \eqref{def:discrete_gradient}), for which~we~use~that~${\bfv_h \hspace{-0.1em}\in\hspace{-0.1em} \bfV_h}$, shows that $\bfv_h \in \bfV_{h,\mathrm{div}}$. This completes~the~proof.
\end{proof}

A by-product of Lemma~\ref{lem:para-ortho-comp} is the following \textit{discrete Helmholtz~\mbox{decomposition}:}
\begin{align}\label{eq:discrete_Helmholtz}
    \smash{\bfV_h = \bfV_{h,\mathrm{div}} \, \oplus \, \nabla^h Q_h\,.}
\end{align}

Using the discrete differential and solution operators defined in \eqref{def:discrete_operators}, we can
derive representation formulas similar to \eqref{eq:leray.1} for  the orthogonal projections~$\mathcal{P}_h$ and~$\mathcal{P}^\perp_h$.\vspace{-0.5mm}\enlargethispage{0.5mm}
\begin{lemma}[Representation of orthogonal projections] \label{lem:rep-disc-proj} There hold\vspace{-0.5mm}
\begin{subequations}\label{lem:rep-disc-proj.0} 
\begin{alignat}{2}  
    \mathcal{P}_h  &= \mathcal{P}_{\bfV_h} - \nabla^h (\Delta_{N}^h)^{-1}\smash{\mathrm{div}}^h&& \quad\text{ in }\mathcal{L}(\bfL^2(\Omega);\bfV_{h,\textup{div}}) \,,\label{lem:rep-disc-proj.0.1} \\[-0.5mm]
    \mathcal{P}^\perp_h & = \mathcal{P}_{\bfV_h}^\perp  + \nabla^h (\Delta_{N}^h)^{-1} \smash{\mathrm{div}}^h &&\quad\text{ in }\mathcal{L}(\bfL^2(\Omega);(\bfV_{h,\textup{div}})^\perp) \,.\label{lem:rep-disc-proj.0.2} 
\end{alignat}
\end{subequations}
\end{lemma}

Before \hspace{-0.1mm}we \hspace{-0.1mm}prove \hspace{-0.1mm}the \hspace{-0.1mm}representations \hspace{-0.1mm}\eqref{lem:rep-disc-proj.0}, \hspace{-0.1mm}we \hspace{-0.1mm}briefly \hspace{-0.1mm}comment~\hspace{-0.1mm}on~\hspace{-0.1mm}their~\hspace{-0.1mm}\mbox{consequences}.
\begin{remark}\label{rem:rep-disc-proj}
The representations \eqref{lem:rep-disc-proj.0}  enable the transfer of stability (\textit{e.g.}, in Lebesgue~or~Sobolev spaces) of the unconstrained orthogonal projection~$\mathcal{P}_{\bfV_h}$, which is known to hold on quasi-uniform triangulations and some graded triangulations (see, \textit{e.g.},~\cite{Boman2006,Diening2021}) to the constrained orthogonal projection~$\mathcal{P}_h$, provided that the discrete differential~and solution operators defined in \eqref{def:discrete_operators}~are~stable. Thus, verifying the stability of the latter is an alternative approach for the stability derivation of the constrained projections. 

Moreover, note that restricted to  $\bfV_h$, the representation formulas \eqref{lem:rep-disc-proj.0}  reduce to\vspace{-0.5mm}
\begin{subequations}\label{rem:rep-disc-proj.0} 
\begin{alignat}{2}  
    \mathcal{P}_h  &= \identity - \nabla^h (\Delta_{N}^h)^{-1}\smash{\mathrm{div}}^h&&\quad\text{ in }\mathcal{L}(\bfV_h;\bfV_{h,\textup{div}}) \,,\label{rem:rep-disc-proj.0.1} \\[-0.5mm]\mathcal{P}_h^\perp & = \nabla^h (\Delta_{N}^h)^{-1}\smash{\mathrm{div}}^h &&\quad\text{ in }\mathcal{L}(\bfV_h;(\bfV_{h,\textup{div}})^\perp\cap \bfV_h)\,.\label{rem:rep-disc-proj.0.2} 
\end{alignat}
\end{subequations}
\end{remark}

\begin{proof}[Proof (of Lemma~\ref{lem:rep-disc-proj})]\let\qed\relax
Note that it is sufficient to verify either~\eqref{rem:rep-disc-proj.0.1}~or~\eqref{rem:rep-disc-proj.0.2}, as the other would follow from $\identity = \mathcal{P}_h  + \mathcal{P}^\perp_h$. For this reason, we only verify \eqref{rem:rep-disc-proj.0.1}. To~this~end,~we~introduce the operator $\mathcal{J}_h\colon \bfL^2(\Omega)\to \bfV_{h,\textup{div}}$, 
 defined by
\begin{align}\label{def:Jh}
   \smash{\forall \bfu \in \bfL^2(\Omega)\colon\quad \mathcal{J}_h \bfu \coloneqq \mathcal{P}_{\bfV_h} \bfu - \nabla^h (\Delta_{N}^h)^{-1}\smash{\mathrm{div}}^h \bfu\quad\text{ in }\bfV_{h,\textup{div}}\,.}
\end{align} 

We need to show that\enlargethispage{3mm}
\begin{itemize}[noitemsep,topsep=2pt,leftmargin=!,labelwidth=\widthof{a)},font=\itshape]
    \item[a)]\hypertarget{lem:rep-disc-proj.a}{} $\mathcal{J}_h\colon \bfL^2(\Omega)\to \bfV_{h,\textup{div}}$ is a projection on $\bfV_{h,\mathrm{div}}$;
    \item[b)]\hypertarget{lem:rep-disc-proj.b}{} $\mathcal{J}_h\colon \bfL^2(\Omega)\to \bfV_{h,\textup{div}}$ satisfies~\eqref{eq:dis-ort-proj}.
\end{itemize} 
If \hyperlink{lem:rep-disc-proj.a}{a}) and \hyperlink{lem:rep-disc-proj.b}{b}) are verified, they will imply that $\mathcal{J}_h$ is an orthogonal projection~on~$\bfV_{h,\mathrm{div}}$, which, due to the uniqueness of the orthogonal projection, will guarantee that $\mathcal{J}_h = \mathcal{P}_h$. 

Before we start the verification of \hyperlink{lem:rep-disc-proj.a}{a}) and \hyperlink{lem:rep-disc-proj.b}{b}), for every $\bfu\in \bfL^2(\Omega)$, 
we~note~that
\begin{align} \label{eq:disc-identity}
\smash{\forall q_h \in Q_h\colon \quad     (q_h, \mathrm{div}\, \nabla^h (\Delta_{N}^h)^{-1} \smash{\mathrm{div}}^h \bfu)_{\Omega} = -(\nabla^h q_h,   \bfu)_{\Omega}}\,, 
\end{align}
which follows from the definitions of $\nabla^h$, $\mathrm{div}^h$,  and $(\Delta_{N}^h)^{-1}$~(\textit{cf}.~\eqref{def:discrete_gradient}--\eqref{def:discrete_inverse_neumann}).

\emph{ad \hyperlink{lem:rep-disc-proj.a}{a})}   
We need to show that 
\begin{itemize}[noitemsep,topsep=2pt,leftmargin=!,labelwidth=\widthof{a.2)},font=\itshape]
    \item[a.1)]\hypertarget{lem:rep-disc-proj.a.1}{} $\mathcal{J}_h(\bfL^2(\Omega))\subseteq \bfV_{h,\mathrm{div}}$;
    \item[a.2)]\hypertarget{lem:rep-disc-proj.a.2}{} $\mathcal{J}_h = \identity$~in~$\bfV_{h,\mathrm{div}}$.
\end{itemize} 

\emph{ad \hyperlink{lem:rep-disc-proj.a.1}{a.1})}    Let $\bfu\in \bfL^2(\Omega)$ be fixed, but arbitrary. Invoking the definitions of $\nabla^h$~(\textit{cf}.\ \eqref{def:discrete_gradient}) and $\smash{\mathrm{div}}^h$ (\textit{cf}.\ \eqref{def:discrete_divergence}) together with~\eqref{eq:disc-identity}, 
we observe that
\begin{align*}
    \smash{\forall q_h \in Q_h\colon \quad (q_h, \mathrm{div} \,\mathcal{J}_h\bfu)_{\Omega} = (q_h, \mathrm{div} \,\mathcal{P}_{\bfV_h} \bfu)_{\Omega} - (q_h, \mathrm{div}\, \nabla^h (\Delta_{N}^h)^{-1}\smash{\mathrm{div}}^h \bfu )_{\Omega} = 0\,.}
\end{align*}

\emph{ad \hyperlink{lem:rep-disc-proj.a.2}{a.2})}  
Since, \hspace{-0.175mm}by \hspace{-0.175mm}Lemma \hspace{-0.175mm}\ref{lem:para-ortho-comp},
\hspace{-0.175mm}$(\nabla^h(\Delta_N^h)^{-1}\textup{div}^h)(\bfV_{h,\textup{div}})\hspace{-0.2em}\subseteq  \hspace{-0.2em}(\bfV_{h,\textup{div}})^{\perp}\hspace{-0.05em}\cap\hspace{-0.05em}\bfV_h$~\hspace{-0.175mm}and,~\hspace{-0.175mm}by~\hspace{-0.175mm}\hyperlink{lem:rep-disc-proj.a.1}{a.1}), $(\nabla^h(\Delta_N^h)^{-1}\textup{div}^h)(\bfV_{h,\textup{div}})\hspace{-0.2em}=\hspace{-0.2em}(\mathcal{P}_{\bfV_h}\hspace{-0.15em}-\mathcal{J}_h)(\bfV_{h,\textup{div}})\hspace{-0.2em}=\hspace{-0.2em}(\textup{Id}-\mathcal{J}_h)(\bfV_{h,\textup{div}})\hspace{-0.2em}\subseteq \hspace{-0.2em}\bfV_{h,\textup{div}}$,~we~have~that $(\nabla^h(\Delta_N^h)^{-1}\textup{div}^h)(\bfV_{h,\textup{div}})\subseteq  ((\bfV_{h,\textup{div}})^{\perp}\bfV_h)\cap \bfV_{h,\textup{div}}=\{\bfzero_d\}$, which implies~claim~\hyperlink{lem:rep-disc-proj.a.2}{a.2}).

\emph{ad \hyperlink{lem:rep-disc-proj.b}{b})} Let $\bfu \in \bfL^2(\Omega)$ be fixed, but arbitrary. Using the definition of $\mathcal{J}_h$ (\textit{cf}.\ \eqref{def:Jh}),  $ (\bfV_{h})^\perp \subseteq (\bfV_{h,\mathrm{div}})^\perp $, and Lemma~\ref{lem:para-ortho-comp} (which yields that ${\nabla^h (\Delta_{N}^h)^{-1} \mathrm{div}^h \bfu \in  (\bfV_{h,\mathrm{div}})^\perp}$), 
we find that\vspace{-0.5mm}
\begin{align*}
     \smash{\forall\bfxi_h \in \bfV_{h,\mathrm{div}}\colon \quad ( \bfu - \mathcal{J}_h \bfu, \bfxi_h )_{\Omega}  = ( \mathcal{P}_{\bfV_h}^\perp \bfu +  \nabla^h (\Delta_{N}^h)^{-1} \smash{\mathrm{div}}^h \bfu, \bfxi_h )_{\Omega} = 0\,.}\tag*{$\qedsymbol$}
\end{align*} 
\end{proof}


\subsection{Leray projection on $\bfL^r$-integrable vector fields}\vspace{-0.5mm}
\hspace{5mm}In the previous section, many  arguments used the fact that $\bfL^2(\Omega)$ is a Hilbert space.~\mbox{However}, the \hspace{-0.1mm}canonical \hspace{-0.1mm}function \hspace{-0.1mm}space \hspace{-0.1mm}to \hspace{-0.1mm}deal \hspace{-0.1mm}with \hspace{-0.1mm}the \hspace{-0.1mm}extra-stress \hspace{-0.1mm}tensor  \hspace{-0.1mm}is \hspace{-0.1mm}$\mathbb{L}^p(\Omega)$. \hspace{-0.1mm}This \hspace{-0.1mm}requires~\hspace{-0.1mm}to~\hspace{-0.1mm}\mbox{generalize} the \hspace{-0.1mm}(discrete) \hspace{-0.1mm}Leray \hspace{-0.1mm}projection \hspace{-0.1mm}and \hspace{-0.1mm}(discrete) \hspace{-0.1mm}Helmholtz \hspace{-0.1mm}decompositions \hspace{-0.1mm}to \hspace{-0.1mm}$\bfL^r$-integrable~\hspace{-0.1mm}vector \hspace{-0.1mm}fields, where $r \in (1,+\infty)$ denotes an arbitrary integrability~\mbox{index}.\vspace{-0.5mm}

\subsubsection{\hspace{-0.5mm}Continuous \hspace{-0.1mm}case}\vspace{-0.5mm}
\hspace{5mm}Since \hspace{-0.1mm}without \hspace{-0.1mm}a \hspace{-0.1mm}Hilbert \hspace{-0.1mm}structure \hspace{-0.1mm}we \hspace{-0.1mm}can~\hspace{-0.1mm}no~\hspace{-0.1mm}longer~\hspace{-0.1mm}\mbox{define} an orthogonal projection, we define $\mathcal{P}$ and $\mathcal{P}^\perp$ by means of \eqref{eq:leray} instead. Then,~stability of $\mathcal{P}$ and $\mathcal{P}^\perp$ on $\bfL^r(\Omega)$ and $\mathbf{W}^{1,r}_{\bfn}(\Omega)$ is inherited from the respective  
stability~of~$\Delta_N^{-1}\mathrm{div}$, which itself depends on 
$\Omega$ and $r$.\vspace{-0.5mm}

\begin{lemma}[$\bfL^r(\Omega)$-stability of $\mathcal{P}$]\label{lem:lr-stab}
Let $\Omega\subseteq \mathbb{R}^d$, $d\ge 2$, be a bounded domain such that
\hspace{-0.1mm}the \hspace{-0.1mm}homogeneous \hspace{-0.1mm}Neumann--Laplace \hspace{-0.1mm}problem~\hspace{-0.1mm}with~\hspace{-0.1mm}right-hand~\hspace{-0.1mm}side~\hspace{-0.1mm}in~\hspace{-0.1mm}\mbox{divergence}~\hspace{-0.1mm}form is \emph{$W^{1,r}$-regular}, \textit{i.e.}, for every $\bfg\in \bfL^r(\Omega)$, there exists a unique 
$u\in W^{1,r}(\Omega)\cap L^r_0(\Omega)$ such that 
\begin{align}\label{eq:laplace-neumann-weak.1}
    \smash{\forall v\in W^{1,r'}(\Omega)\cap L^{r'}_0(\Omega)\colon\quad(\nabla u,\nabla v)_{\Omega}=( \bfg,\nabla v)_{\Omega}\,,}
\end{align}
and 
\begin{align}\label{eq:laplace-neumann-weak.2}
    \smash{\|\nabla u\|_{r,\Omega}\lesssim \|\bfg\|_{r,\Omega}\,.}
\end{align}
Then, there holds
\begin{align} \label{eq:leray-stab-on-Lr}
  \smash{\forall \bfu \in \bfL^r(\Omega)\colon \quad  \|\mathcal{P} \bfu\|_{r,\Omega} + \|\mathcal{P}^\perp \bfu\|_{r,\Omega} \lesssim  \| \bfu\|_{r,\Omega}\,.}
\end{align}
In particular, there holds $\smash{\bfL^r(\Omega) = \bfW^r_0(\mathrm{div}^0;\Omega) \oplus \nabla W^{1,r}(\Omega)}$. 
\end{lemma}\newpage

\begin{remark}\label{rem:lr-stab}
   The homogeneous Neumann--Laplace problem with right-hand side in divergence form is $W^{1,r}$-regular if either of the following sufficient cases is satisfied: 
\begin{itemize}[noitemsep,topsep=2pt,leftmargin=!,labelwidth=\widthof{(Case 3)}]
    \item[(Case 1)]\hypertarget{Case 1}{} $\partial \Omega$ is smooth (\textit{cf}.\ \cite[Thm.~2]{zbMATH03600473});
    \item[(Case 2)]\hypertarget{Case 2}{} $\Omega$ is convex (\textit{cf}.\ \cite[Thms.~1.2,~1.3]{Geng2010}).
\end{itemize}
\end{remark}

\begin{proof}[Proof (of Lemma \ref{lem:lr-stab})]
 By the representations in \eqref{eq:leray}, both including $\Delta_N^{-1}\mathrm{div}$ (\textit{cf}.\ \eqref{eq:Neumann-Laplace}), the assertion is a consequence of the assumed $W^{1,r}$-regularity (\textit{cf}.\ Remark~\ref{rem:lr-stab}). 
\end{proof}

\begin{lemma}[$\mathbf{W}^{1,r}_{\bfn}(\Omega)$-stability of $\mathcal{P}$] \label{lem:Wr-leray}
Let $\Omega\subseteq \mathbb{R}^d$, $d\ge 2$, be a bounded domain such that the homogeneous Neumann--Laplace problem  right-hand side in divergence form (\textit{i.e.}, \eqref{eq:laplace-neumann} with $g=0$ and $f=\mathrm{div}\,\bfg$ for some $\bfg\in \bfW^{1,r}_\bfn(\Omega)$) is $W^{2,r}$-regular (\textit{cf}.\ Lemma \ref{lem:inf-sup_continuous}).  
Then, there holds 
\begin{align} \label{eq:leray-stab-on-Wr}
  \forall \bfu \in \mathbf{W}^{1,r}_{\bfn}(\Omega)\colon \quad  \|\nabla \mathcal{P} \bfu\|_{r,\Omega} + \|\nabla \mathcal{P}^\perp \bfu\|_{r,\Omega} \lesssim \|\nabla \bfu\|_{r,\Omega}\,.
\end{align}
In particular, there holds $\mathbf{W}^{1,r}_{\bfn}(\Omega) = (\mathbf{W}^{1,r}_{\bfn}(\Omega)\cap \mathbf{W}^r_0(\textup{div}^0;\Omega)) \oplus (\mathbf{W}^{1,r}_{\bfn}(\Omega) \cap \nabla W^{2,r}(\Omega))$. 
\end{lemma}
\begin{proof} 
By the representations in  \eqref{eq:leray}, 
both including $\Delta_N^{-1}\mathrm{div}$ (\textit{cf}.\ \eqref{eq:Neumann-Laplace}),~the~asser\-tion is a  consequence of the assumed $W^{2,r}$-regularity (\textit{cf}.\ Remark \ref{rem:laplace-neumann}). 
\end{proof}

\begin{remark}
The derivation of a Helmholtz decomposition for Sobolev spaces~with\-out specified boundary conditions is straightforward. However, an extension to no-slip  
boundary conditions seems to be impossible, since
correcting the divergence while preserving tangential boundary traces is unfeasible. In this context,~the~decomposition ${\mathbf{W}^{1,r}_0(\Omega) \hspace{-0.15em}=\hspace{-0.15em}\{ \bfu\hspace{-0.15em} \in\hspace{-0.15em} \mathbf{W}^{1,r}_0(\Omega)| \Delta \mathrm{div}\,\bfu \hspace{-0.15em}=\hspace{-0.15em} 0\} \hspace{-0.15em}\oplus\hspace{-0.15em}  \nabla W^{2,r}_0(\Omega)}$ might be of use in the error analysis of numerical schemes. More details on this decomposition can be found, \textit{e.g.}, in \cite{BORCHERS1990}.
\end{remark}

\subsubsection{Discrete case}

\hspace{5mm}The forthcoming error analysis  for the kinematic pressure builds strongly on the $\mathbf{L}^r(\Omega)$- and $\mathbf{W}^{1,r}(\Omega)$-stability of the discrete Leray projection~$\mathcal{P}_h$. However, --and in contrast to the continuous case-- this stability has not been~proved~yet. Merely, the $\mathbf{W}^{1,2}(\Omega)$-stability of the discrete Leray projection \hspace{-0.5mm}restricted \hspace{-0.5mm}to \hspace{-0.5mm}incompressible \hspace{-0.5mm}vector \hspace{-0.5mm}fields \hspace{-0.5mm}with \hspace{-0.5mm}vanishing~\hspace{-0.5mm}trace~\hspace{-0.5mm}is~\hspace{-0.5mm}known~\hspace{-0.5mm}(\textit{cf}.~\hspace{-0.5mm}\mbox{\cite[Lem.~\hspace{-0.5mm}3.1]{Bernardi1985}}). Even for 
not necessarily incompressible 
vector fields the $\mathbf{W}^{1,2}(\Omega)$-stability~is~\mbox{unknown}.

Next, we demonstrate how to extend $\mathbf{L}^r(\Omega)$- 
to $\mathbf{W}^{1,r}(\Omega)$-stability for not necessarily incompressible vector fields. To this end, we make the following essential assumption.

\begin{assumption}[$\bfL^r(\Omega)$-stability of $\mathcal{P}_h$] \label{ass:Lr-stab-leray-dis}
We assume that 
\begin{align} \label{eq:Lr-stab-leray-dis}
\forall \bfu \in \bfL^r(\Omega) \colon \quad \|\mathcal{P}_h \bfu\|_{r,\Omega}
\lesssim \|\bfu\|_{r,\Omega} \,.
\end{align}
\end{assumption}

\begin{remark}[stability of $\mathcal{P}_h$ $\Leftrightarrow$ stability of $\mathcal{P}_h^\perp$]\label{rem:Lr-stab-leray-dis.0}
Due to the identity $\identity = \mathcal{P}_h  + \mathcal{P}^\perp_h$, the $\bfL^r(\Omega)/\bfW^{1,r}(\Omega)$-stability of $\mathcal{P}_h$ is equivalent to the $\bfL^r(\Omega)/\bfW^{1,r}(\Omega)$-stability~of~$\mathcal{P}_h^\perp$.
\end{remark}

\begin{remark}[$\bfL^r(\Omega)$-stability $\Leftrightarrow$ $\bfL^{r'}(\Omega)$-stability]\label{rem:Lr-stab-leray-dis}
    Since $\mathcal{P}_h$ is $\bfL^2(\Omega)$-self-adjoint, its $\bfL^r(\Omega)$-stability also implies its $\bfL^{r'}(\Omega)$-stability and vice versa.
\end{remark}

\begin{lemma}[$\mathbf{W}^{1,r}_{\bfn}(\Omega)$-stability \hspace{-0.1mm}of \hspace{-0.1mm}$\mathcal{P}_h$] \label{lem:unconstrained-stab}
Let \hspace{-0.1mm}the \hspace{-0.1mm}assumptions \hspace{-0.1mm}of \hspace{-0.1mm}Lemma~\hspace{-0.1mm}\ref{lem:Wr-leray}~\hspace{-0.1mm}be~\hspace{-0.1mm}satis\-fied. Moreover,  let Assumptions~\ref{ass:PiQ},~\ref{ass:proj-div},~and~\ref{ass:Lr-stab-leray-dis} be true and assume that $\mathbb{P}^1_c(\mathcal{T}_h) \subseteq \widehat{Q}_h$.
Then, there holds
\begin{align} \label{eq:Wr-stable-disc-leray}
    \forall \bfu \in \bfW^{1,r}_{\bfn}(\Omega)\colon \qquad \|\nabla \mathcal{P}_h \bfu\|_{r,\Omega}  + \|\nabla \mathcal{P}_h^\perp \bfu\|_{r,\Omega}  \lesssim \|\nabla \bfu\|_{r,\Omega}\,.
\end{align}
\end{lemma}
\begin{proof}
Let $\bfu \hspace{-0.1em}\in\hspace{-0.1em} \bfW^{1,r}_{\bfn}(\Omega)$ be fixed, but arbitrary. Due to Remark \ref{rem:Lr-stab-leray-dis.0},~it~suffices to verify that $ \|\nabla \mathcal{P}_h \bfu\|_{r,\Omega} \lesssim \|\nabla \bfu\|_{r,\Omega}$. According to  
Lemma~\ref{lem:Wr-leray}, there exist  $\bfxi \in \mathbf{W}^{1,r}_{\bfn}(\Omega)\cap \mathbf{W}^r_0(\textup{div}^0;\Omega)$ and $\nabla g \in \mathbf{W}^{1,r}_{\bfn}(\Omega) \cap \nabla W^{2,r}(\Omega)$ such that $\bfu = \bfxi + \nabla g$~a.e.~in~$\Omega$~and
\begin{align}\label{lem:unconstrained-stab.1}
    \|\nabla \bfxi\|_{r,\Omega} + \|\nabla^2 g\|_{r,\Omega} \lesssim \|\nabla \bfu\|_{r,\Omega}\,.
\end{align}
This decomposition enables to split the stability verification into~two~separate parts --the stability for incompressible vector fields and gradients-- which~we~address~individually;  that is,~we~will~show\newpage \noindent the following stability estimates: 
\begin{subequations}
\begin{alignat}{2} \label{eq:stab-first}
    &\forall \bfxi \in \mathbf{W}^{1,r}_{\bfn}(\Omega)\cap \mathbf{W}^r_0(\textup{div}^0;\Omega)\colon \qquad  \|\nabla \mathcal{P}_h \bfxi\|_{r,\Omega}&&\lesssim \|\nabla \bfxi\|_{r,\Omega}\,, \\ \label{eq:stab-second}
    &\forall \nabla g \in \mathbf{W}^{1,r}_{\bfn}(\Omega) \cap \nabla W^{2,r}(\Omega)\colon \quad\,  \|\nabla \mathcal{P}_h \nabla g\|_{r,\Omega} &&\lesssim  \|\nabla^2 g \|_{r,\Omega}\,.
\end{alignat}    
\end{subequations}
Once \eqref{eq:stab-first} and \eqref{eq:stab-second} have been verified, from \eqref{lem:unconstrained-stab.1}, it follows that \eqref{eq:Wr-stable-disc-leray} applies. 

\textit{$\bullet$ Stability for incompressible vector fields.} Here, the stability follows similarly to~\cite[Lem.~3.1]{Bernardi1985}, where the special case $r=2$ is discussed. For the  sake of completeness, we present the arguments adapted to the general case.\enlargethispage{2.5mm}

We intend to correct the discrete Leray projection $\mathcal{P}_h$~by~$\Pi_h^\bfV$. To this end, we first note that 
$\Pi_h^{\bfV}\big(\mathbf{W}^{1,r}_{\bfn}(\Omega)\cap \mathbf{W}^r_0(\textup{div}^0;\Omega) \big) \subseteq \bfV_{h,\mathrm{div}}$, which implies that $\mathcal{P}_h \Pi_h^\bfV \xi = \Pi_h^\bfV \xi$. 
Invoking an inverse estimate (\textit{cf}.\ \cite[Lem.~12.1]{EG21}) and Assumption~\ref{ass:Lr-stab-leray-dis}, we find that
\begin{align}\label{lem:unconstrained-stab.2}
    \begin{aligned}
    \|\nabla \mathcal{P}_h \bfxi \|_{r,\Omega} &\leq \|\nabla \mathcal{P}_h( \bfxi - \Pi_h^\bfV \bfxi)  \|_{r,\Omega} + \|\nabla  \Pi_h^\bfV \bfxi  \|_{r,\Omega}\\& \lesssim h^{-1} \|\bfxi - \Pi_h^\bfV \bfxi \|_{r,\Omega} + \|\nabla  \Pi_h^\bfV \bfxi  \|_{r,\Omega}\,. 
    \end{aligned}
\end{align}
The $\bfW^{1,r}(\Omega)$-stability and $\bfL^r(\Omega)$-approximability of $\Pi_h^{\bfV}$ (\textit{cf}.~\cite[Thm.~3.2]{BBDR12})~yield~that
\begin{align}\label{lem:unconstrained-stab.3}
    \begin{aligned}
     h^{-1} \|\bfxi - \Pi_h^\bfV \bfxi\|_{r,\Omega} + \|\nabla  \Pi_h^\bfV \bfxi \|_{r,\Omega} \lesssim \|\nabla  \bfxi \|_{r,\Omega}\,.
     \end{aligned}
\end{align}
Using \eqref{lem:unconstrained-stab.3} in \eqref{lem:unconstrained-stab.2}, we conclude that the stability estimate 
\eqref{eq:stab-first} applies.

\textit{$\bullet$ Stability for gradients.} A key role in verifying the stability for gradients is the inclusion $\nabla^hQ_h\hspace{-0.15em}\subseteq\hspace{-0.15em}\textup{ker}(\mathcal{P}_h )$  (\textit{i.e.}, $\mathcal{P}_h \nabla^h q_h\hspace{-0.15em} =\hspace{-0.15em} 0$ for all $q_h\hspace{-0.15em} \in\hspace{-0.15em} Q_h$, \textit{cf}.\ Lemma~\ref{lem:para-ortho-comp}),~which~\mbox{enables} to artificially correct analytic gradients. To this end, let $q_h \in  Q_h$ be fixed,~but~\mbox{arbitrary}.
Then,  
an inverse estimate (\textit{cf}.\ \cite[Lem.~12.1]{EG21}) and $\nabla^hQ_h\subseteq\textup{ker}(\mathcal{P}_h )$~show~that 
\begin{align}\label{lem:unconstrained-stab.4}
   \|\nabla \mathcal{P}_h \nabla g\|_{r,\Omega} \lesssim h^{-1} \|\mathcal{P}_h (\nabla g - \nabla^h q_h)\|_{r,\Omega}\,.
\end{align}
by the $\bfL^2(\Omega)$-self-adjointness and $\bfL^{r'}(\Omega)$-stability (\textit{cf}.\ Remark \ref{rem:Lr-stab-leray-dis})~of~$\mathcal{P}_h$,~we~find~that
\begin{align}\label{lem:unconstrained-stab.5}
\begin{aligned}
   \|\mathcal{P}_h (\nabla g - \nabla^h q_h)\|_{r,\Omega}&= \sup_{\bfxi \in \bfL^{r'}(\Omega)\setminus\{\bfzero_d\} } \bigg\{\frac{(\nabla g - \nabla^h q_h, \mathcal{P}_h\bfxi )_{\Omega}}{\|\bfxi\|_{r',\Omega}}\bigg\}
   \\&\lesssim \sup_{\bfxi \in \bfL^{r'}(\Omega)\setminus\{\bfzero_d\} } \bigg\{\frac{(\nabla g - \nabla^h q_h, \mathcal{P}_h\bfxi )_{\Omega}}{\|\mathcal{P}_h\bfxi\|_{r',\Omega}}\bigg\}
   \\&\leq  \sup_{\bfxi_h \in \bfV_h\setminus\{\bfzero_d\} } \bigg\{\frac{(\nabla g - \nabla^h q_h, \bfxi_h )_{\Omega}}{\|\bfxi_h\|_{r',\Omega}}\bigg\}\,.
   \end{aligned}
\end{align}
Integration-by-parts, the definition of $\nabla^h$ (\textit{cf}.\ \eqref{def:discrete_gradient}), H\"older's inequality,~and~an~inverse estimate (\textit{cf}.\ \cite[Lem.~12.1]{EG21}) ensure that
\begin{align}\label{lem:unconstrained-stab.6}
    \begin{aligned}
    \sup_{\bfxi_h \in \bfV_h\setminus\{\bfzero_d\} }\bigg\{ \frac{(\nabla g - \nabla^h q_h, \bfxi_h )_{\Omega}}{\|\bfxi_h\|_{r',\Omega}}\bigg\} &= \sup_{\bfxi_h \in \bfV_h\setminus\{\bfzero_d\} } \bigg\{\frac{(g -  q_h, \mathrm{div}\,\bfxi_h )_{\Omega}}{\|\bfxi_h\|_{r',\Omega}}\bigg\} \\&\lesssim h^{-1} \|g - q_h\|_{r,\Omega}\,.
    \end{aligned}
\end{align}
By \eqref{lem:unconstrained-stab.5} and \eqref{lem:unconstrained-stab.6}, taking the infimum with respect to $q_h\in Q_h$ in \eqref{lem:unconstrained-stab.6},~we~arrive~at
\begin{align} \label{eq:best-approx-gradients}
    \|\mathcal{P}_h \nabla g\|_{r,\Omega}+h\,\|\nabla \mathcal{P}_h \nabla g\|_{r,\Omega} \lesssim h^{-1} \inf_{q_h \in Q_h}\big\{\|g - q_h\|_{r,\Omega}\big\}\,.
\end{align}
The next step is to verify that the approximation space for pressure $Q_h$ is rich enough to approximate twice weakly differentiable functions with second-order accuracy~in~$L^r(\Omega)$.~In~fact, since we assumed that $\mathbb{P}^1_c(\mathcal{T}_h) \subseteq  \widehat{Q}_h$, the pressure space is capable~of~this; it \hspace{-0.1mm}can \hspace{-0.1mm}be~\hspace{-0.1mm}seen~\hspace{-0.1mm}by,~\hspace{-0.1mm}\textit{e.g.}, \hspace{-0.1mm}choosing \hspace{-0.1mm}the \hspace{-0.1mm}Clément \hspace{-0.1mm}interpolant \hspace{-0.1mm}(\textit{cf}.\ \hspace{-0.1mm}\cite{Clement1975})~\hspace{-0.1mm}of~\hspace{-0.1mm}$g\hspace{-0.15em}\in\hspace{-0.15em} W^{2,r}(\Omega)$~\hspace{-0.1mm}and~\hspace{-0.1mm}\mbox{correcting} its integral mean that~$\inf_{q_h \in Q_h}\{\|g - q_h\|_{r,\Omega}\}\hspace{-0.1em} \lesssim\hspace{-0.1em} h^2 \|\nabla^2 g\|_{r,\Omega}$.~Hence,~from~\eqref{eq:best-approx-gradients},~we~get
\begin{align}\label{lem:unconstrained-stab.7}
    \|\mathcal{P}_h \nabla g\|_{r,\Omega} +h\, \|\nabla \mathcal{P}_h \nabla g\|_{r,\Omega}  \lesssim h\, \|\nabla^2 g\|_{r,\Omega}\,,
\end{align}
which includes the claimed stability  estimate 
\eqref{eq:stab-second}.
\end{proof}

\begin{remark}
The case $r=2$ is especially important not only for Newtonian fluids, where $\bfL^2(\Omega)$-integrable vector field are canonical, but also for non-Newtonian~fluids~and, in particular, our error analysis of the kinematic pressure. In this case,~\mbox{Assumption}~\ref{ass:Lr-stab-leray-dis} is trivially satisfied, 
so that Lemma~\ref{lem:unconstrained-stab}  implies
the $\bfW^{1,2}_{\bfn}(\Omega)$-stability of $\mathcal{P}_h$.

In the general case $r \neq 2$, however, the verification of Assumption~\ref{ass:Lr-stab-leray-dis} is non-trivial; it has to be verified on a case-by-case basis. 

A related result,  not relying on this assumption, was derived~in~\mbox{\cite[Lem.~2.32]{tscherpel2018finite}}. Therein, the author uses higher-order regularity to artificially utilize the $\bfL^2(\Omega)$-stability also in the  case $r\neq 2$. However, the result is limited to incompressible~vector~fields.
\end{remark}

\subsection{Operator convergence}
\hspace{5mm}As its name suggests, the discrete Leray projection should approximate the continuous~Leray projection asymptotically. The next lemma 
proves this convergence with a~linear~error~\mbox{decay}~rate.\enlargethispage{5mm}

\begin{lemma}[$\bfL^r(\Omega)$-$\bfW^{1,r}_{\bfn}(\Omega)$-approximability] \label{lem:operator-conv}
Under the assumptions of Lemma~\ref{lem:unconstrained-stab},~there holds
\begin{align*}
    \smash{\|\mathcal{P} - \mathcal{P}_h\|_{\smash{\mathcal{L}(\bfW^{1,r}_{\bfn}(\Omega); \bfL^r(\Omega))}} + \|\mathcal{P}^\perp - \mathcal{P}_h^\perp\|_{\smash{\mathcal{L}(\bfW^{1,r}_{\bfn}(\Omega); \bfL^r(\Omega))}} \lesssim h\,.}
\end{align*} 
\end{lemma}

\begin{proof}
Let $\bfu \in \bfW^{1,r}_{\bfn}(\Omega)$ be fixed, but arbitrary. 
Due to $\mathcal{P}^\perp\bfu - \mathcal{P}_h^\perp\bfu = \mathcal{P}\bfu - \mathcal{P}_h\bfu$,
it~is~sufficient~to show that $\|\mathcal{P}\bfu - \mathcal{P}_h\bfu\|_{r,\Omega} \lesssim h\, \|\nabla \bfu\|_{r,\Omega}$,
for which we resort~to~arguments similar to the proof of Lemma~\ref{lem:unconstrained-stab}:

 Lemma~\ref{lem:Wr-leray} yields $\bfxi \in \mathbf{W}^{1,r}_{\bfn}(\Omega)\cap \mathbf{W}^r_0(\textup{div}^0;\Omega)$ and ${\nabla g \in \mathbf{W}^{1,r}_{\bfn}(\Omega) \cap \nabla W^{2,r}(\Omega)}$~such that $\bfu = \bfxi + \nabla g$ a.e.\ in $\Omega$ and $\|\nabla \bfxi\|_{r,\Omega} + \|\nabla^2 g\|_{r,\Omega} \lesssim \|\nabla \bfu\|_{r,\Omega}$. Then, due to $\mathcal{P}\bfxi=\bfxi$ a.e.\ in $\Omega$ and $\mathcal{P}\nabla g =\bfzero$ a.e.\ in $\Omega$, we find that
\begin{align}\label{lem:operator-conv.1}
    \smash{\|\mathcal{P}\bfu - \mathcal{P}_h\bfu\|_{r,\Omega}  \leq  \|\bfxi - \mathcal{P}_h \bfxi\|_{r,\Omega} +  \|\mathcal{P}_h  \nabla g\|_{r,\Omega}\,.}
\end{align}
The two terms on the right-hand side of \eqref{lem:operator-conv.1} correspond to the approximability of incompressible vector fields and gradients, respectively: 

\textit{$\bullet$ Approximability for incompressible vector fields.}
Adding and subtracting $\Pi_h^{\bfV}$, using  that \hspace{-0.1mm}$\mathcal{P}_h \Pi_h^\bfV \xi \hspace{-0.15em}=\hspace{-0.15em} \Pi_h^\bfV \xi$, \hspace{-0.1mm}invoking \hspace{-0.1mm}the \hspace{-0.1mm}assumed \hspace{-0.1mm}$\bfL^r(\Omega)$-\hspace{-0.1mm}stability \hspace{-0.1mm}of \hspace{-0.1mm}$\mathcal{P}_h$ \hspace{-0.1mm}(\textit{cf}.~\hspace{-0.1mm}\mbox{Assumption}~\hspace{-0.1mm}\ref{ass:Lr-stab-leray-dis}), and utilising the $\bfL^r(\Omega)$-approximability of $\Pi_h^{\bfV}$ (\cite[Thm.~3.2]{BBDR12}) yield that
\begin{align*}
    \begin{aligned} 
    \|\bfxi - \mathcal{P}_h \bfxi\|_{r,\Omega} &\leq \|\bfxi - \Pi_h^{\bfV} \bfxi \|_{r,\Omega} + \|\mathcal{P}_h( \bfxi - \Pi_h^{\bfV} \bfxi)\|_{r,\Omega}\\&\lesssim \|\bfxi - \Pi_h^{\bfV} \bfxi\|_{r,\Omega} 
    \\&\lesssim h \,\|\nabla \bfxi\|_{r,\Omega}\,.
    \end{aligned}
\end{align*}

\textit{$\bullet$ Approximability for gradients.} According to \eqref{lem:unconstrained-stab.7}, we have that $\|\mathcal{P}_h \nabla g\|_{r,\Omega}\lesssim h\, \|\nabla^2 g\|_{r,\Omega}$.
\end{proof}
 
From the $\bfL^r(\Omega)$-$\bfW^{1,r}_{\bfn}(\Omega)$-approximability 
and $\bfL^2(\Omega)$-self-adjointness of $\mathcal{P}_h$~and~$\mathcal{P}$, 
it follows their $(\bfW^{1,r}_{\bfn}(\Omega))^*$-$\bfL^{r'}(\Omega)$-approximability.

\begin{lemma}[$(\bfW^{1,r}_{\bfn}(\Omega))^*$-$\bfL^{r'}(\Omega)$-approximability] \label{lem:operator-conv_dual}
Under the assumptions~of Lemma \ref{lem:unconstrained-stab},  there holds
\begin{align*}
    \smash{\|\mathcal{P} - \mathcal{P}_h\|_{\smash{\mathcal{L}(\bfL^{r'}(\Omega);(\bfW^{1,r}_{\bfn}(\Omega))^*)}} + \|\mathcal{P}^\perp - \mathcal{P}_h^\perp\|_{\smash{\mathcal{L}(\bfL^{r'}(\Omega);(\bfW^{1,r}_{\bfn}(\Omega))^*)}} \lesssim h\,.}
\end{align*}
\end{lemma} 

\begin{proof}
    Let $\bfu\in \bfL^{r'}(\Omega)$ be fixed, but arbitrary. 
Due to the $\bfL^2(\Omega)$-self-adjointness of $\mathcal{P}$ and $\mathcal{P}_h$ and the $\bfL^r(\Omega)$-$\bfW^{1,r}_{\bfn}(\Omega)$-approximability of $\mathcal{P}_h$ (\textit{cf}.\ Lemma~\ref{lem:operator-conv}),~we~find~that
\begin{align}\label{lem:operator-conv_dual.0}
    \begin{aligned}
    \|\mathcal{P}\bfu - \mathcal{P}_h\bfu\|_{\smash{(\bfW^{1,r}_{\bfn}(\Omega))^*}}&=
    \sup_{\bfxi\in \bfW^{1,r}_{\bfn}(\Omega)\setminus\{\bfzero_d\}}{\bigg\{\frac{(\bfu,\mathcal{P}\bfxi- \mathcal{P}_h\bfxi)_{\Omega}}{\|\nabla \bfxi\|_{r,\Omega}}\bigg\}}
    \lesssim h\,\|\bfu\|_{r',\Omega}\,,
    \end{aligned}
\end{align}
\textit{i.e.}, the claimed $(\bfW^{1,r}_{\bfn}(\Omega))^*$-$\bfL^{r'}(\Omega)$-approximability of $\mathcal{P}_h$. 
Since ${\mathcal{P}^\perp\bfu -\mathcal{P}_h^\perp\bfu =\mathcal{P}\bfu-\mathcal{P}_h\bfu}$, from \eqref{lem:operator-conv_dual.0}, we also conclude the claimed $(\bfW^{1,r}_{\bfn}(\Omega))^*$-$\bfL^{r'}(\Omega)$-approximability~of~$\mathcal{P}_h^\perp$.
\end{proof}

\section{\emph{A priori} error analyses}\label{sec:error-estimate}

\hspace{5mm}In this section, we carry out \emph{a priori} error analyses for the velocity vector field, the kinematic pressure, and the acceleration vector field.

\subsection{\emph{A priori} error analysis for the velocity vector field}

\hspace{5mm}In this subsection, we derive a (quasi-)best-approximation result as well as explicit error decay rates for the velocity vector field.

\begin{theorem}[(quasi-)best-approximation]\label{thm:main.velocity} Let the Assumptions \ref{ass:PiQ}, \ref{ass:proj-div}, and \ref{ass:discrete_inf_sup_II} be satisfied~and~\mbox{either} $\mathbb{P}^1(\mathcal{S}_h^{\Gamma})\subseteq  \widehat{Z}_h$ or $\bfV_h\coloneqq \widehat{\bfV}_h\cap \bfV$ together with $\widehat{Z}_h=\{0\}$. Moreover, assume that $\lambda\in  L^{p'}(I;L^{p'}(\Gamma))$ and $\bfF(\bfvarepsilon(\bfv))\in C^0(\overline{I};\mathbb{L}^2(\Omega))$. Then, there holds 
\begin{align*}
    \|\bfv_h^\tau-\mathrm{I}_\tau^0 \bfv \|_{L^\infty(I;\bfL^2(\Omega))}^2&+\|\bfF(\bfvarepsilon(\bfv_h^\tau))-\bfF(\mathrm{I}_\tau^0\bfvarepsilon(\bfv))\|_{2,\Omega_T}^2\\&\lesssim \|\bfF(\bfvarepsilon(\bfv))-\bfF(\mathrm{I}_\tau^0\bfvarepsilon(\bfv))\|_{2,\smash{\Omega_T}}^2
        +\|\bfF(\bfvarepsilon(\bfv))-\bfF(\Pi_h^0\bfvarepsilon(\bfv))\|_{2,\smash{\Omega_T}}^2\\&\quad+\inf_{\bfxi_h^\tau\in \mathbb{P}^0(\mathcal{I}_\tau;\bfV_{h,\textup{div}})}{        \left\{\begin{aligned}
            &\|\bfF(\mathrm{I}_\tau^0\bfvarepsilon(\bfv))-\bfF(\bfvarepsilon(\bfxi_h^{\tau}))\|_{2,\smash{\Omega_T}}^2\\&\quad+\tfrac{1}{\tau}\|\mathrm{I}_\tau^0\bfv-
			\bfxi_h^\tau\|_{2,\smash{\Omega_T}}^2
        \end{aligned}\right\}}
        \\&\quad+ \inf_{\bfxi_h\in \bfV_{h,\textup{div}}}{\big\{\|\bfv_0-\bfxi_h\|_{2,\Omega}^2\big\}}
            \\&\quad + \inf_{\eta_h^\tau\in \mathbb{P}^0(\mathcal{I}_\tau;\widehat{Q}_h)}\big\{\rho_{(\varphi_{\smash{\vert\bfvarepsilon(\bfv)\vert}})^*,\smash{\Omega_T}}(\Pi_\tau^0\Pi_h^{\ell_{\bfv}-1}q-\eta_h^\tau)\big\}\\&\quad+\inf_{\mu_h^\tau \in \mathbb{P}^0(\mathcal{I}_\tau;\widehat{Z}_h)}{\big\{h\,\rho_{(\varphi_{\smash{\vert \Pi_h^0\bfvarepsilon(\bfv)\vert}})^*,\smash{\Gamma_T}}(\Pi_\tau^0\pi_h^{\ell_{\bfv}}\lambda-\mu_h^\tau)\big\}}
        \,,
\end{align*}
where, for each $\ell\in \mathbb{N}_0$,  $\Pi_h^{\ell}\colon  \hspace{-0.1em}L^1(\Omega)\hspace{-0.1em}\to\hspace{-0.1em}\mathbb{P}^{\ell}(\mathcal{T}_h)$ and $\pi_h^{\ell}\colon\hspace{-0.1em} L^1(\Gamma)\hspace{-0.1em}\to\hspace{-0.1em} \mathbb{P}^{\ell}(\mathcal{S}_h^{\Gamma})$~denote~the~\mbox{(local)} $L^2$-projections. If  $\bfV_h\coloneqq \widehat{\bfV}_h\cap\bfV$, the last infimum on the right-hand~side can~be~omitted.
\end{theorem}

As a direct consequence of the (quasi-)best-approximation result in Theorem~\ref{thm:main.velocity}, we obtain  explicit error decay rates given appropriate regularity~assumptions~on~solutions.

\begin{corollary}[error decay rates]\label{cor:main.velocity} 
Let the assumptions of Theorem \ref{thm:main.velocity}~be~satisfied. Moreover,  assume that $h^2\lesssim\tau$,  $\bfv_0 \in \bfW^{1,2}_\bfn(\Omega)$, and
    \begin{subequations}\label{eq:regularity}
    \begin{alignat}{2}
        \bfF(\bfvarepsilon(\bfv))&\in N^{\alpha_t,2}(I;\mathbb{L}^2(\Omega))\cap L^2(I;(N^{\alpha_x,2}(\Omega))^{d\times d})\,,\; \alpha_t\in (\tfrac{1}{2},1]\,,\;\alpha_x\in (0,1]\,,\label{eq:regularity.1}\\ 
        \bfv&\in L^\infty(I;(N^{\alpha_x+1,2}(\Omega))^d)\,,\label{eq:regularity.2}\\ 
        q &\in L^{p'}(I;C^{\beta_x,p'}(\Omega))\,,\; \beta_x\in (0,1]\,,\label{eq:regularity.3}\\ 
        \lambda &\in L^{p'}(I;C^{\gamma_x-\smash{\frac{1}{p'}},p'}(\Gamma))\,,\; \gamma_x\in (\tfrac{1}{p'},1]\,, \label{eq:regularity.4}
    \end{alignat}
    \end{subequations}
    where, for $s\in (0,1]$ and $r\in (1,+\infty)$, we denote by $N^{s,r}$ the \emph{(Bochner--)Nikolskii~space} and by $C^{s,r}$ the \emph{Calder\'on space} (\textit{cf}.\ \cite[Subsecs.\ 2.1, 2.3]{BerselliKaltenbachKo2025}). 
    If $p<2$ and \eqref{eq:bc.1} is weakly imposed, we additionally assume that $\alpha_x>\frac{1}{2}$. Then,
     there holds
    \begin{align}\label{eq:rates.velocity} 
 \|\bfv_h^\tau\hspace{-0.1em}-\hspace{-0.1em}\mathrm{I}_\tau^0 \bfv \|_{L^\infty(I;\bfL^2(\Omega))}+\|\bfF(\bfvarepsilon(\bfv_h^\tau))\hspace{-0.1em}-\hspace{-0.1em}\bfF(\mathrm{I}_\tau^0\bfvarepsilon(\bfv))\|_{2,\Omega_T}&\lesssim \left\{\begin{aligned} 
    &\smash{\tau^{\alpha_t}+h^{\alpha_x}}\\[-0.5mm]&+\smash{h^{\smash{\min\{1,\frac{p'}{2}\}\min\{\beta_x,\gamma_x\}}}}\,.\end{aligned}\right.
    \end{align}    
    If $\delta>0$ and $p\ge 2$, the  assumption \eqref{eq:regularity.2} can be reduced to $\bfv\in \smash{L^\infty(I;(N^{\alpha_x,2}(\Omega))^d)}$.
\end{corollary}  

\begin{remark} 
    Note \hspace{-0.1mm}that \hspace{-0.1mm}\eqref{eq:regularity.1} \hspace{-0.1mm}yields \hspace{-0.1mm}that \hspace{-0.1mm}$ \bfF(\bfvarepsilon(\bfv))\hspace{-0.15em}\in\hspace{-0.15em} \smash{C^0(\overline{I};\mathbb{L}^2(\Omega))}$ \hspace{-0.1mm}(\textit{cf}.\ \hspace{-0.1mm}\cite[Lem.~\hspace{-0.1mm}3.4]{BerselliKaltenbachKo2025}). 
\end{remark} 

\begin{proof}[Proof (of Theorem \ref{thm:main.velocity})]
    To begin with, let $\bfxi_h^\tau\in \mathbb{P}^0(\mathcal{I}_\tau;\bfV_{h,\textup{div}})$,  $\eta_h^\tau\in \mathbb{P}^0(\mathcal{I}_\tau;\widehat{Q}_h)$, $\mu_h^\tau\in \mathbb{P}^0(\mathcal{I}_\tau;\smash{\widehat{Z}}_h)$,  and $m=1,\ldots,M$ be fixed, but arbitrary and abbreviate $\Omega_T^m\coloneqq (0,t_m)\times \Omega$ and $\Gamma_T^m\coloneqq (0,t_m)\times \Gamma$. Then, using \eqref{eq:hammera},~we~find~that
	\begin{align}\label{thm:main.velocity.1}
		\begin{aligned}
			 \|\bfF(\bfvarepsilon(\bfv^{\tau}_h)) - \bfF(\mathrm{I}_\tau^0\bfvarepsilon(\bfv))\|_{2,\smash{\Omega_T^m}}^2
            &\lesssim (\bfS(\bfvarepsilon(\bfv_h^\tau)) - \bfS(\mathrm{I}_\tau^0\bfvarepsilon(\bfv)),\bfvarepsilon(\bfv_h^\tau)-\mathrm{I}_\tau^0\bfvarepsilon(\bfv))_{\smash{\Omega_T^m}}
			\\&= (\bfS(\bfvarepsilon(\bfv_h^\tau)) - \bfS(\bfvarepsilon(\bfv)),\bfvarepsilon(\bfv_h^\tau)-\bfvarepsilon(\bfxi_h^\tau))_{\smash{\Omega_T^m}} 
			\\&\quad+ (\bfS(\bfvarepsilon(\bfv_h^\tau)) - \bfS(\bfvarepsilon(\bfv)),\bfvarepsilon(\bfxi_h^\tau)-\mathrm{I}_\tau^0\bfvarepsilon(\bfv))_{\smash{\Omega_T^m}}
			\\&\quad+(\bfS(\bfvarepsilon(\bfv))-\bfS(\mathrm{I}_\tau^0\bfvarepsilon(\bfv)),\bfvarepsilon(\bfv_h^\tau)-\mathrm{I}_\tau^0\bfvarepsilon(\bfv))_{\smash{\Omega_T^m}} 
			\\&\eqqcolon \smash{I_{m}^1+ I_{m}^2 +I_{m}^3}\,. 
		\end{aligned}
	\end{align}

Let us next estimate the terms $\smash{I_{m}^{i}}$, $i=1,2,3$:  
		
	\textit{ad $I_{m}^i$, $i=1,2$.} Using the $\varepsilon$-Young inequality \eqref{ineq:young.1} (with $a=\vert\mathrm{I}_\tau^0\bfvarepsilon(\bfv)\vert$)~and~\eqref{eq:hammera}, we obtain
	\begin{align}\label{thm:main.velocity.2}
		\begin{aligned} 
		  I_{m}^2 
       &\leq c_\varepsilon\,\|\bfF(\bfvarepsilon(\bfxi_h^\tau))-\bfF(\mathrm{I}_\tau^0\bfvarepsilon(\bfv))\|_{2,\smash{\Omega_T^m}}^2
            \\& \quad +\varepsilon\, \big\{\|\bfF(\bfvarepsilon(\bfv))-\bfF(\mathrm{I}_\tau^0\bfvarepsilon(\bfv))\|_{2,\smash{\Omega_T^m}}^2+\|\bfF(\bfvarepsilon(\bfv_h^{\tau}))-\bfF(\mathrm{I}_\tau^0\bfvarepsilon(\bfv))\|_{2,\smash{\Omega_T^m}}^2\big\} \,;\\
            I_{m}^3
            & \leq c_\varepsilon\,\|\bfF(\bfvarepsilon(\bfv))-\bfF(\mathrm{I}_\tau^0\bfvarepsilon(\bfv))\|_{2,\smash{\Omega_T^m}}^2
            +\varepsilon\,\|\bfF(\bfvarepsilon(\bfv_h^\tau))-\bfF(\mathrm{I}_\tau^0\bfvarepsilon(\bfv))\|_{2,\smash{\Omega_T^m}}^2
			\,.
	   \end{aligned}
	\end{align}

    \textit{ad $I_{m}^1$.}  Testing \hspace{-0.1mm}the \hspace{-0.1mm}first \hspace{-0.1mm}lines \hspace{-0.1mm}of \hspace{-0.1mm}the \hspace{-0.1mm}weak \hspace{-0.1mm}formulation \hspace{-0.1mm}(\textit{cf}.\ \hspace{-0.1mm}Definition \hspace{-0.1mm}\ref{def:weak_form})~\hspace{-0.1mm}and~\hspace{-0.1mm}the~\hspace{-0.1mm}discrete formulation (\textit{cf}.\ Definition \ref{def:discrete_form}) with  ${\boldsymbol{\varphi}\hspace{-0.05em}=\hspace{-0.05em}\boldsymbol{\varphi}_h^{\tau}\hspace{-0.05em}=\hspace{-0.05em}(\bfv_h^\tau-\bfxi_h^\tau)\chi_{(0,t_m)}\hspace{-0.05em}\in  \hspace{-0.05em}\mathbb{P}^0(\mathcal{I}_\tau;\bfV_{h,\textup{div}})}$, subtracting the resulting equations, and using that, by Definition \ref{def:discrete_form}, we have that\enlargethispage{5mm}
    \begin{align*}
        (q_h^\tau,\textup{div}  (\bfv_h^\tau-\bfxi_h^\tau))_{\smash{\Omega_T^m}} &=0 =(\eta_h^\tau,\textup{div}(  \bfv_h^\tau-\bfxi_h^\tau))_{\smash{\Omega_T^m}}\,,\\
         (\lambda_h^\tau,(\bfv_h^\tau-\bfxi_h^\tau)\cdot \bfn)_{\smash{\Gamma_T^m}} &=0 =(\mu_h^\tau,(  \bfv_h^\tau-\bfxi_h^\tau)\cdot\bfn)_{\smash{\Gamma_T^m}}\,,
    \end{align*}
   as well as that $\textup{div}(  \bfv_h^\tau-\bfxi_h^\tau)\in \smash{\mathbb{P}^0(\mathcal{I}_\tau;\mathbb{P}^{\ell_{\bfv}-1}(\mathcal{T}_h))}$ and $(  \bfv_h^\tau-\bfxi_h^\tau)\cdot\bfn\in \smash{\mathbb{P}^0(\mathcal{I}_\tau;\mathbb{P}^{\ell_{\bfv}}(\mathcal{S}_h^{\Gamma}))}$,~we~arrive~at
	\begin{align}\label{thm:main.velocity.3}
		\begin{aligned} 
			I_{m}^1&=((\eta_h^\tau- \smash{\Pi_\tau^0 \Pi_h^{\ell_{\bfv}-1}q})\mathbb{I}_{d\times d},\bfvarepsilon (\bfv_h^\tau-\bfxi_h^\tau))_{\smash{\Omega_T^m}}
            +(\mu_h^\tau-\Pi_\tau^0 \pi_h^{\ell_{\bfv}}\lambda,(\bfv_h^\tau-\bfxi_h^\tau)\cdot \bfn)_{\smash{\Gamma_T^m}}\\&\quad+(\mathrm{d}_\tau\{ \bfv_h^\tau-\mathrm{I}_\tau^0\bfv\}, \bfv_h^\tau-\bfxi_h^\tau)_{\smash{\Omega_T^m}} 
            \eqqcolon  I_{m}^{1,1}+ I_{m}^{1,2}+I_{m}^{1,3}\,.
		\end{aligned}
	\end{align} 
	
    Let us next estimate $I_{m}^{1,i}$, $i=1,2,3$:
		
	\textit{ad $I_{m}^{1,1}$.} \hspace{-0.5mm}Using  \hspace{-0.15mm}the \hspace{-0.15mm}$\varepsilon$-Young \hspace{-0.15mm}inequality~\hspace{-0.15mm}\eqref{ineq:young.1} \hspace{-0.15mm}(with \hspace{-0.15mm}$a\hspace{-0.15em}=\hspace{-0.15em}\vert\mathrm{I}_\tau^0\bfvarepsilon(\bfv)\vert$) \hspace{-0.15mm}and \hspace{-0.15mm}the \hspace{-0.15mm}shift~\hspace{-0.15mm}change~\hspace{-0.15mm}\eqref{lem:shift-change.3},~\hspace{-0.15mm}we~\hspace{-0.15mm}\mbox{get}
    \begin{align}\label{thm:main.velocity.4}
		\begin{aligned}
			I_{m}^{1,1}
            &\leq c_\varepsilon\,\rho_{(\varphi_{\smash{\vert\mathrm{I}_\tau^0\bfvarepsilon(\bfv)\vert}})^*,\smash{\Omega_T^m}}(\smash{\Pi_\tau^0 \Pi_h^{\ell_{\bfv}-1}q}-\eta_h^\tau)\\&\quad+
			\varepsilon\,\big\{\|\bfF(\bfvarepsilon(\bfv_h^{\tau}))-\bfF(\mathrm{I}_\tau^0\bfvarepsilon(\bfv))\|_{2,\smash{\Omega_T^m}}^2+\|\bfF(\mathrm{I}_\tau^0\bfvarepsilon(\bfv))-\bfF(\bfvarepsilon(\bfxi_h^\tau))\|_{2,\smash{\Omega_T^m}}^2\big\}
            \\&\lesssim c_\varepsilon\,\big\{\rho_{(\varphi_{\smash{\vert\bfvarepsilon(\bfv)\vert}})^*,\smash{\Omega_T^m}}(\Pi_\tau^0 \Pi_h^{\ell_{\bfv}-1}q-\eta_h^\tau)
           + \|\bfF(\mathrm{I}_\tau^0\bfvarepsilon(\bfv))-\bfF(\bfvarepsilon(\bfv))\|_{2,\smash{\Omega_T^m}}^2\big\}
           \\&\quad+
			\varepsilon\,\big\{\|\bfF(\bfvarepsilon(\bfv_h^{\tau}))-\bfF(\mathrm{I}_\tau^0\bfvarepsilon(\bfv))\|_{2,\smash{\Omega_T^m}}^2+\|\bfF(\mathrm{I}_\tau^0\bfvarepsilon(\bfv))-\bfF(\bfvarepsilon(\bfxi_h^\tau))\|_{2,\smash{\Omega_T^m}}^2\big\}
            \,.
		\end{aligned}
	\end{align} 
    
    \textit{ad $I_{m}^{1,2}$.} In the case $\bfV_h\coloneqq \smash{\widehat{\bfV}_h}\cap \bfV$  and $\smash{\widehat{Z}_h}=\{0\}$, we have that  $I_{m}^{1,2}=0$.~Otherwise, using   the $\varepsilon$-Young inequality~\eqref{ineq:young.1} (with $a=\vert\Pi_h^0\bfvarepsilon(\bfv)\vert$), 
    we obtain
    \begin{align}\label{thm:main.velocity.5}
		\begin{aligned}
			I_{m}^{1,2}
            &\lesssim c_\varepsilon\,h\,\rho_{(\varphi_{\smash{\vert \Pi_h^0\bfvarepsilon(\bfv)\vert}})^*,\smash{\Gamma_T^m}}(\smash{\Pi_\tau^0 \pi_h^{\ell_{\bfv}}\lambda}-\mu_h^\tau)
            +\varepsilon\,h\,\rho_{\varphi_{\smash{\vert\Pi_h^0\bfvarepsilon(\bfv)\vert}},\smash{\Gamma_T^m}}(h^{-1}(\bfv_h^{\tau}-\bfxi_h^{\tau})\cdot\bfn)\,.
		\end{aligned}
	\end{align}
    where, due to Lemma \ref{lem:normal_trace_estimate} (as $\smash{\mathbb{P}^1(\mathcal{S}_h^{\Gamma})\subseteq \widehat{Z}_h}$) and the shift change \eqref{lem:shift-change.3}, we have that
    \begin{align}\label{thm:main.velocity.6}
        \begin{aligned} h\,\rho_{\varphi_{\smash{\vert\Pi_h^0\bfvarepsilon(\bfv)\vert}},\smash{\Gamma_T^m}}(h^{-1}(\bfv_h^{\tau}-\bfxi_h^{\tau})\cdot\bfn)&\lesssim \rho_{\varphi_{\smash{\vert\Pi_h^0\bfvarepsilon(\bfv)\vert}},\smash{\Omega_T^m}}(\bfvarepsilon(\bfv_h^{\tau})-\bfvarepsilon(\bfxi_h^{\tau}))
        \\&\lesssim \|\bfF(\bfvarepsilon(\bfv_h^{\tau}))-\bfF(\bfvarepsilon(\bfxi_h^{\tau}))\|_{2,\smash{\Omega_T^m}}^2\\&\quad+\|\bfF(\bfvarepsilon(\bfv_h^{\tau}))-\bfF(\Pi_h^0\bfvarepsilon(\bfv))\|_{2,\smash{\Omega_T^m}}^2\,.
        \end{aligned}
    \end{align} 
		
	\textit{ad $I_{m}^{1,3}$.} \hspace{-0.1mm}Using \hspace{-0.1mm}the \hspace{-0.1mm}discrete \hspace{-0.1mm}integration-by-parts \hspace{-0.1mm}formula \hspace{-0.1mm}for \hspace{-0.1mm}$\mathrm{d}_\tau$~\hspace{-0.1mm}and~\hspace{-0.1mm}classical~\hspace{-0.1mm}weighted Young inequality, every $m = 1,\ldots,M$,  we observe that 
	\begin{align}\label{thm:main.velocity.7}
		\begin{aligned} 
			I_{m}^{1,3}
            &=(\mathrm{d}_\tau \{\bfv_h^\tau-\mathrm{I}_\tau^0 \bfv\} , \bfv_h^\tau-\mathrm{I}_\tau^0\bfv )_{\smash{\Omega_T^m}}+(\mathrm{d}_\tau\{\bfv_h^\tau-\mathrm{I}_\tau^0 \bfv\} , \mathrm{I}_\tau^0\bfv-
			\bfxi_h^\tau)_{\smash{\Omega_T^m}}
			\\&\ge 
             \tfrac{1}{2}\|\bfv_h^\tau(t_m)-\bfv(t_m)\|_{2,\Omega}^2-\tfrac{1}{2}\|\mathcal{P}_h\bfv_0-\bfv_0\|_{2,\Omega}^2 - \tfrac{1}{2\tau}\|\mathrm{I}_\tau^0\bfv-
			\bfxi_h^\tau\|_{2,\smash{\Omega_T^m}}^2\,.
		\end{aligned}
	\end{align} 
    Using \eqref{thm:main.velocity.2}--\eqref{thm:main.velocity.7} in \eqref{thm:main.velocity.1}, forming the maximum~with~\mbox{respect} to $m = 1,\ldots,M$ and the infimum with respect to $\bfxi_h^\tau\in \mathbb{P}^0(\mathcal{I}_\tau;\bfV_{h,\textup{div}})$, $\eta_h^\tau\in \mathbb{P}^0(\mathcal{I}_\tau;\widehat{Q}_h)$,  and ${\mu_h^\tau\in  \mathbb{P}^0(\mathcal{I}_\tau;\widehat{Z}_h)}$, as well as using that $\|\bfv_0-\mathcal{P}_h\bfv_0\|_{2,\Omega}
    \leq \inf_{\bfxi_h\in \bfV_{h,\textup{div}}}{\{\|\bfv_0-\bfxi_h\|_{2,\Omega}\}}$,~we~conclude~that the claimed (quasi-)best-approximation result for the velocity vector field applies.
\end{proof}

\begin{proof}[Proof (of Corollary \ref{cor:main.velocity})]
    
    Denote the terms on the right-hand side in Theorem~\ref{thm:main.velocity} (in the order displayed) by $I_{\tau,h}^i$, $i=1,\ldots,6$, respectively, which are estimated as follows:\enlargethispage{2.5mm}
    
    \emph{ad $I_{\tau,h}^{i}$, $i=1,2$.} Due to the regularity assumption \eqref{eq:regularity.1} and \cite[Lem.\ 5.1(5.3)]{BerselliKaltenbachKo2025},  we have that $ I_{\tau,h}^1\lesssim \tau^{\smash{2\alpha_t}}[\bfF(\bfvarepsilon(\bfv))]_{N^{\alpha_t,2}(I;\mathbb{L}^2(\Omega))}^2$ and $I_{\tau,h}^2\lesssim h^{\smash{2\alpha_x}}[\bfF(\bfvarepsilon(\bfv))]_{L^2(I;(N^{\alpha_x,2}(\Omega))^{d\times d})}^2$.
    
    \emph{ad $I_{\tau,h}^3$.} 
    For $\bfxi_h^\tau\hspace{-0.1em}\coloneqq\hspace{-0.1em} \mathrm{I}_\tau^0\Pi_h^{\bfV}\bfv\hspace{-0.1em}\in \hspace{-0.1em}\mathbb{P}^0(\mathcal{I}_\tau;\bfV_{h,\textup{div}})$, similar to \cite[Lem.\ 5.15(5.17)]{BerselliKaltenbachKo2025}~as~well~as using that $h^2\hspace{-0.1em}\lesssim\hspace{-0.1em} \tau$ and the approximation properties of $\Pi_h^{\bfV}$ (\textit{cf}.\ \cite[Prop.\ 2.9]{BerselliRuzicka2021}),~we~\mbox{obtain} 
    \begin{subequations}\label{cor:main.velocity.0-1}
    \begin{align}\label{cor:main.velocity.0-1.1}
    \left\{\quad\begin{aligned}
        \|\bfF(\mathrm{I}_\tau^0\bfvarepsilon(\bfv))-\bfF(\bfvarepsilon(\mathrm{I}_\tau^0\Pi_h^{\bfV}\bfv))\|_{2,\smash{\Omega_T}}^2&\lesssim \tau^{\smash{2\alpha_t}}[\bfF(\bfvarepsilon(\bfv))]_{N^{\alpha_t,2}(I;\mathbb{L}^2(\Omega))}^2\\[-0.25mm]&\quad+h^{\smash{2\alpha_x}}[\bfF(\bfvarepsilon(\bfv))]_{L^2(I;(N^{\alpha_x,2}(\Omega))^{d\times d})}^2\,;
    \end{aligned}\right.\\[-0.25mm]\label{cor:main.velocity.0-1.2}
    \left\{\quad\begin{aligned}
        \tfrac{1}{\tau}\|\mathrm{I}_\tau^0\bfv-
			\mathrm{I}_\tau^0\Pi_h^{\bfV}\bfv\|_{2,\smash{\Omega_T}}^2&\lesssim \tfrac{h^2}{\tau}\|\boldsymbol{\varepsilon}(\mathrm{I}_\tau^0\bfv)-
			\boldsymbol{\varepsilon}(\mathrm{I}_\tau^0\Pi_h^{\bfV}\bfv)\|_{2,\smash{\Omega_T}}^2
            \\[-0.25mm] &\lesssim h^{\smash{2\alpha_x}}[\bfv]_{L^\infty(I;(N^{1+\alpha_x,2}(\Omega))^d}^2\,.\hspace*{1.3cm}
        \end{aligned}\right.
    \end{align}
    \end{subequations}
    If $\delta\hspace{-0.1em}>\hspace{-0.1em}0$ and $p\hspace{-0.1em}\ge\hspace{-0.1em} 2$, we have that $\delta^{p-2}\vert \bfA-\bfB\vert^2\hspace{-0.1em}\lesssim\hspace{-0.1em} \vert \bfF(\bfA)-\bfF(\bfB)\vert^2$ for all $\bfA,\bfB\hspace{-0.1em}\in \hspace{-0.1em}\mathbb{R}^{d\times d}$;~thus,~we~can~use 
    $\|\boldsymbol{\varepsilon}(\mathrm{I}_\tau^0\bfv)-
			\boldsymbol{\varepsilon}(\mathrm{I}_\tau^0\Pi_h^{\bfV}\bfv)\|_{2,\smash{\Omega_T}}^2\lesssim \delta^{p-2}\|\bfF(\bfvarepsilon(\mathrm{I}_\tau^0\bfv))-\bfF(\bfvarepsilon(\mathrm{I}_\tau^0\Pi_h^{\bfV}\bfv))\|_{2,\Omega_T}^2$ together~with~\eqref{cor:main.velocity.0-1.1}~in~\eqref{cor:main.velocity.0-1.2}. 

    \emph{ad $I_{\tau,h}^4$.}  For $\bfxi_h \coloneqq \Pi_h^{\bfV} \bfv_0\in  \bfV_{h,\textup{div}}$, using the approximation properties~of~$\Pi_h^{\bfV}$ (\textit{cf}.\ \cite[Prop.\ 2.9]{BerselliRuzicka2021}), we find that $I_{\tau,h}^4\lesssim h^2\|\nabla\bfv_0\|_{2,\Omega}^2$.
    
    \emph{ad $I_{\tau,h}^5$.} For $\eta_h^\tau\coloneqq\Pi_\tau^0\Pi_h^{Q} q\in \mathbb{P}^0(\mathcal{I}_\tau;\widehat{Q}_h)$, similar to \cite[Lems.\ 5.18(5.20), B.5(B.9)]{BerselliKaltenbachKo2025}~and~using that $(\varphi_a)^*(hr)\lesssim h^{\smash{\min\{2,p'\}}}(\varphi_a)^*(r)$ for all $a,r\ge 0$ and $h\in (0,1]$,~we~find~that
    \begin{align*}
        I_{\tau,h}^4&\lesssim 
       h^{\smash{\min\{2,p'\}\beta_x}}\rho_{(\varphi_{\vert \bfvarepsilon(\bfv)\vert})^*,\Omega_T}(\vert \nabla^{\beta_x}q\vert )
        \\[-0.25mm]&\quad +\tau^{\smash{2\alpha_t}}[\bfF(\bfvarepsilon(\bfv))]_{N^{\alpha_t,2}(I;\mathbb{L}^2(\Omega))}^2 + h^{\smash{2\alpha_x}}[\bfF(\bfvarepsilon(\bfv))]_{L^2(I;(N^{\alpha_x,2}(\Omega))^{d\times d})}^2 \,,
    \end{align*}  
    where $\vert \nabla^{\beta_x} q(t)\vert$ 
    denotes 
    for a.e.\ $t\in I$ the \textit{
    upper 
    Calder\'on gradient} (\textit{cf}.\ \cite[Subsec.~2.1]{BerselliKaltenbachKo2025}).

    \emph{ad $I_{\tau,h}^6$.} In the case $\bfV_h\coloneqq  \widehat{\bfV}_h\cap \bfV$  and $\widehat{Z}_h\coloneqq\{0\}$, we have that $I_{\tau,h}^6=0$. Otherwise, for $\mu_h^\tau\coloneqq\Pi_\tau^0\smash{\pi_h^{\ell_\lambda}}\lambda\in  \mathbb{P}^0(\mathcal{I}_\tau;\smash{\widehat{Z}_h})$, using the shift change \eqref{lem:shift-change.3}, we find that
    \begin{align*}
        I_{\tau,h}^6&\lesssim h\,\rho_{(\varphi_{\smash{\vert \Pi_\tau^0\Pi_h^0\bfvarepsilon(\bfv)\vert}})^*,\smash{\Gamma_T}}(\Pi_\tau^0(\pi_h^{\ell_{\bfv}}\lambda-\pi_h^{\ell_{\lambda}}\lambda))+h\,\|\bfF(\Pi_h^0\bfvarepsilon(\bfv))-\bfF(\Pi_\tau^0\Pi_h^0\bfvarepsilon(\bfv))\|_{2,\Gamma_T}^2
\,,
   \end{align*}
    where, by the discrete trace inequality in space (\textit{cf}.\ \cite[Lem.\ 12.8]{EG21}) as well as \eqref{eq:hammera}, Jensen's inequality in space, the shift change \eqref{lem:shift-change.1}, and, again, \eqref{eq:hammera}, we have that 
    \begin{align*}
        h\,\|\bfF(\Pi_h^0\bfvarepsilon(\bfv))-\bfF(\Pi_\tau^0\Pi_h^0\bfvarepsilon(\bfv))\|_{2,\Gamma_T}^2&\lesssim \|\bfF(\Pi_h^0\bfvarepsilon(\bfv))-\bfF(\Pi_\tau^0\Pi_h^0\bfvarepsilon(\bfv))\|_{2,\Omega_T}^2
        \\[-0.25mm]&\lesssim \smash{I_{\tau,h}^1}+\smash{I_{\tau,h}^2}\,.
    \end{align*}
    By \hspace{-0.1mm}$\vert \pi_h^{\ell_{\bfv}}\lambda-\pi_h^{\ell_{\lambda}}\lambda\vert\hspace{-0.15em}\lesssim\hspace{-0.15em} h^{\smash{\gamma_x-\frac{1}{p'}}}(\vert \nabla^{\gamma_x}\hspace{-0.1em}\lambda\vert+\pi_h^0\vert \nabla^{\gamma_x}\hspace{-0.1em}\lambda\vert)$~\hspace{-0.1mm}a.e.~\hspace{-0.1mm}on~\hspace{-0.1mm}$\Gamma_T$, \hspace{-0.1mm}Jensen's~\hspace{-0.1mm}\mbox{inequality}~\hspace{-0.1mm}in~\hspace{-0.1mm}time~and space,  
    $(\varphi_a)^*(hr)\lesssim h^{\smash{\min\{2,p'\}}}(\varphi_a)^*(r)$ for all $a,r\ge  0$ and $h\in  (0,1]$,~and~${\min\{2,p'\}\frac{1}{p'}\ge 1}$, we have that\vspace{-0.5mm}
    \begin{align*}
        h\,\rho_{(\varphi_{\smash{\vert \Pi_\tau^0\Pi_h^0\bfvarepsilon(\bfv)\vert}})^*,\smash{\Gamma_T}}(\Pi_\tau^0(\pi_h^{\ell_{\bfv}}\lambda\hspace{-0.1em}-\hspace{-0.1em}\pi_h^{\ell_{\lambda}}\lambda))
        \lesssim h^{\smash{\min\{2,p'\}\gamma_x}}\,\rho_{(\varphi_{\smash{\vert \Pi_\tau^0\Pi_h^0\bfvarepsilon(\bfv)\vert}})^*,\smash{\Gamma_T}}(\vert \nabla^{\gamma_x}\lambda\vert)\,.
    \end{align*}
    
    Next, we distinguish the cases $p\ge 2$ and $p< 2$:

    \emph{$\bullet$ Case $p\ge 2$.} In this case, we have that $(\varphi_a)^*(r)\lesssim  \varphi^*(r)$ for all $a,r\ge 0$~and,~thus,\vspace{-0.5mm}
    \begin{align*}
        \rho_{(\varphi_{\smash{\vert \Pi_\tau^0\Pi_h^0\bfvarepsilon(\bfv)\vert}})^*,\smash{\Gamma_T}}(\vert \nabla^{\gamma_x}\lambda\vert)\lesssim  \rho_{\varphi^*,\smash{\Gamma_T}}(\vert \nabla^{\gamma_x}\lambda\vert)\,.
    \end{align*}

    \emph{$\bullet$ Case $p< 2$.} In this case, we have that $\bfF(\bfvarepsilon(\bfv))\in L^2(I;\mathbb{L}^2(\Gamma))$ (equivalent~to $\bfvarepsilon(\bfv)\in L^p(I;\mathbb{L}^p(\Gamma))$) and, consequently, 
    by the shift change \eqref{lem:shift-change.3}~and~a~fractional~trace inequality in space (\textit{cf}.\ \cite[Rem.\ 12.19]{EG21}), we have that
    \begin{align*}
        \rho_{(\varphi_{\smash{\vert \Pi_\tau^0\Pi_h^0\bfvarepsilon(\bfv)\vert}})^*,\smash{\Gamma_T}}(\vert \nabla^{\gamma_x}\lambda\vert)&\lesssim  \rho_{(\varphi_{\smash{\vert \bfvarepsilon(\bfv)\vert}})^*,\smash{\Gamma_T}}(\vert \nabla^{\gamma_x}\lambda\vert)+\|\bfF(\bfvarepsilon(\bfv))-\bfF(\Pi_\tau^0\Pi_h^0\bfvarepsilon(\bfv))\|_{2,\Gamma_T}^2\\&\lesssim  \rho_{(\varphi_{\smash{\vert \bfvarepsilon(\bfv)\vert}})^*,\smash{\Gamma_T}}(\vert \nabla^{\gamma_x}\lambda\vert)+\|\bfF(\bfvarepsilon(\bfv))-\bfF(\Pi_\tau^0\Pi_h^0\bfvarepsilon(\bfv))\|_{2,\Omega_T}^2
        \\& \quad +h^{\smash{2\widetilde{\alpha}_x-1}}[\bfF(\bfvarepsilon(\bfv))]_{L^2(I;\mathbb{W}^{\widetilde{\alpha}_x,2}(\Omega))}^2\,,
    \end{align*}
    where $\widetilde{\alpha}_x\in \smash{(\frac{1}{2},\alpha_x)}$, so that $\smash{(N^{\alpha_x,2}(\Omega))^{d\times d}\hookrightarrow \mathbb{W}^{\widetilde{\alpha}_x,2}(\Omega)}$.

    Putting it all together, we conclude that the claimed \emph{a priori} error estimate \eqref{eq:rates.velocity} for the velocity vector field applies.
\end{proof}
\newpage

\subsection{\emph{A priori} error analysis for the kinematic pressure and acceleration vector field}\vspace{-0.5mm}

\hspace{5mm}In this subsection, we derive a   (quasi-)best-approximation result as well~as~explicit~error decay rates for the kinematic pressure and the acceleration~vector~field.\vspace{-0.5mm}

\begin{theorem}[(quasi-)best-approximation]\label{thm:main}  
Under the assumptions of Theorem~\ref{thm:main.velocity}, if
    for some $r\in (1,+\infty)$, Assumption \ref{ass:Lr-stab-leray-dis} is satisfied and\vspace{-0.5mm} 
    \begin{align}\label{ass:integrability}
        \smash{q(t), \vert \partial_t\bfv(t)\vert,\vert \bfS(\bfvarepsilon(\bfv))(t)\vert\in L^{r'}(\Omega)\qquad \text{ for a.e.\ }t\in I\,,}
    \end{align} 
    then, for a.e. $t\in I$, there holds
    \begin{align*}
            \|q(t)-q_h^{\tau}(t)\|_{r',\Omega} &+ \|\partial_t \bfv(t)-\mathrm{d}_\tau \bfv_h^{\tau}(t)\|_{\smash{(\bfW^{1,r}_{\bfn}(\Omega))^*}} \\ 
            & \lesssim \inf_{\eta_h(t)\in Q_h}{\big\{\|q(t)-\eta_h(t)\|_{r',\Omega}\big\}} +\inf_{\mu_h(t)\in \smash{\widehat{Z}_h}}{\big\{\|\lambda(t)-\mu_h(t)\|_{r',\Gamma}\big\}}\\ 
            &\quad
+\|\bfS(\bfvarepsilon(\bfv_h^{\tau})(t)) - \bfS(\bfvarepsilon(\bfv)(t))\|_{r',\Omega}
\\&\quad +\|\bff^\tau(t)-\bff(t)\|_{(\bfV_h)^*}
+h\, \|\partial_t\bfv(t)\|_{r',\Omega}\,,
        \end{align*}
    where $(\bfV_h)^*$ is the (topological) dual space of $\bfV_h$ equipped with the $\bfW^{1,r}(\Omega)$-norm~and $\lesssim$ depends on $r$, $\Omega$, and the choice of finite element spaces \eqref{eq:fem_choice}.  If  $\bfV_h\coloneqq \widehat{\bfV}_h\cap\bfV$, the second infimum on the right-hand~side can~be~omitted.
\end{theorem}

As \hspace{-0.1mm}for \hspace{-0.1mm}the \hspace{-0.1mm}velocity \hspace{-0.1mm}vector \hspace{-0.1mm}field, \hspace{-0.1mm}the \hspace{-0.1mm}(quasi-)best-\hspace{-0.1mm}approximation \hspace{-0.1mm}of \hspace{-0.1mm}the \hspace{-0.1mm}kinematic \hspace{-0.1mm}\mbox{pressure} and the acceleration vector field immediately implies  explicit error decay rates, provided that the solution (\textit{i.e.}, velocity vector field and~kinematic~\mbox{pressure}) is sufficiently regular. For this, it is important to identify particular choices~for~${r\in (1,+\infty)}$, such that not only  \eqref{ass:integrability} is satisfied, but, in addition, the kinematic pressure, normal stress \hspace{-0.1mm}component, \hspace{-0.1mm}and \hspace{-0.1mm}extra-stress \hspace{-0.1mm}tensor \hspace{-0.1mm}have \hspace{-0.1mm}increased \hspace{-0.1mm}temporal~\hspace{-0.1mm}and~\hspace{-0.1mm}\mbox{spatial}~\hspace{-0.1mm}\mbox{regularity}.\vspace{-0.5mm} 

\begin{corollary}[error decay rates]\label{cor:pressure_rates}
    Let the assumptions of Corollary~\ref{cor:main.velocity} 
    be satisfied. Moreover, assume that $\bff\in \smash{N^{\alpha_t,p'}(I;\widehat{\bfV}^*)}$. Then, the following statements~apply:

    \begin{itemize}[noitemsep,topsep=2pt,leftmargin=!,labelwidth=\widthof{(ii)}]
        \item[(i)] \hypertarget{cor:pressure_rates.i}{} If the assumptions of Theorem~\ref{thm:main} (with $r = p$) are satisfied, then\vspace{-0.5mm}
    \begin{align*}      
    \|q-q_h^{\tau}\|_{p',\Omega_T}+\|\partial_t \bfv-\mathrm{d}_\tau \bfv_h^{\tau}\|_{L^{p'}(I;\bfV^*)}\lesssim \left\{
         \begin{aligned}
             &\smash{\tau^{\smash{\alpha_t}}+h^{\smash{\alpha_x}}}\\[-0.5mm]&\!\!+\smash{h^{\smash{\min\{1,\frac{p'}{2}\}\min\{\beta_x,\gamma_x\}}}}
         \end{aligned}
        \right\}^{\smash{\min\{1,\frac{2}{p'}\}}}\,,
    \end{align*}
    where $\alpha_t\in \smash{(\frac{1}{2},1]}$ and $\alpha_x,\beta_x,\gamma_x\in (0,1]$ are defined as in Corollary~\ref{cor:main.velocity};
    \item[(ii)]\hypertarget{cor:pressure_rates.ii}{} If $\delta>0$, $p\leq2$, and the assumptions of Theorem~\ref{thm:main} (with $r = 2$) are satisfied,~then 
    \begin{align*}
        \smash{\|q-q_h^{\tau}\|_{2,\Omega_T}+\|\partial_t \bfv-\mathrm{d}_\tau \bfv_h^{\tau}\|_{L^2(I;(\bfW^{1,2}_{\bfn}(\Omega))^*)}\lesssim
        \smash{\tau^{\alpha_t}+h^{\min\{\alpha_x,\beta_x,\gamma_x\}}}}\,.
    \end{align*}
    \end{itemize} 
\end{corollary}

\begin{remark}
An important ingredient in the derivation of Theorem~\ref{thm:main} and Corollary~\ref{cor:pressure_rates} is the assumed integrability of the kinematic pressure, acceleration~vector~field, and extra-stress tensor (\textit{cf}.~\eqref{ass:integrability}). To apply it in particular situations, we need to verify that this~condition~is~satisfied. However, since these quantities are generally unknown, the condition cannot be checked directly; instead, regularity theory is needed to provide verifiable conditions on the data; \textit{e.g.}, Proposition~\ref{prop:regularity1} and Proposition~\ref{prop:regularity2}~ensure~that:
\begin{itemize}[noitemsep,topsep=0pt,leftmargin=!,labelwidth=\widthof{\quad\;$\bullet$}]
    \item Condition~\eqref{ass:integrability} (with $r=p$) is satisfied for $p \ge \frac{-1+4d+\sqrt{9-4d+4d^2}}{3d+2}$, 
        that is $p\ge \frac{1}{8}(7+\sqrt{17})\approx 1.39$ if $d=2$ and $p\ge \frac{1}{11}(11+\sqrt{33}) \approx 1.52$ if $d=3$;
    \item Condition~\eqref{ass:integrability} (with $r=2$) is satisfied for $p \in (\tfrac{2d}{d+2},2]$.
\end{itemize}

Both  choices: $r = p$ and $r = 2$, are natural due to the following: The choice $r = p$ is related to the  growth behaviour of the non-linearity~\eqref{def:A}, whereas the  choice $r = 2$ is motivated \hspace{-0.1mm}by \hspace{-0.1mm}recently-derived \hspace{-0.1mm}regularity \hspace{-0.1mm}results \hspace{-0.1mm}for \hspace{-0.1mm}unsteady~\hspace{-0.1mm}\mbox{$p$-Laplace}~\hspace{-0.1mm}\mbox{systems}~\hspace{-0.1mm}(\textit{cf}.~\hspace{-0.1mm}\cite{Cianchi2019}); \textit{i.e.}, it was shown that $\nabla \bfS(\nabla \bfv) \in \mathbb{L}^2(\Omega)$, highlighting the importance~of~the~\mbox{$2$-scale}. However, we want to stress that the \emph{a priori} regularity analysis of the 
extra-stress~\mbox{tensor} 
lacks systematic investigation~for~the~unsteady~\mbox{$p$-Stokes}~\mbox{equations}~\eqref{eq:p-SE}, since~further difficulties arise from the strain-rate tensor 
and incompressibility~constraint~(\textit{cf}.~\mbox{\cite{BerselliRuzicka2022,BehnDiening2024}}). 
\end{remark}

\begin{proof}[Proof (of Theorem \ref{thm:main})]
We treat the kinematic pressure and the acceleration vector field one after the other:

\emph{1. (Quasi-)best-approximation result for the kinematic pressure.} 
For a.e.\ $t\in I$, let $\eta_h(t)\in Q_h$ and  $\mu_h(t)\in \smash{\widehat{Z}_h}$ be fixed, but arbitrary.
 Due~to~Lemma~\ref{lem:discrete_infsup_I},
        we have that\enlargethispage{5mm}
        \begin{align}\label{eq:main.1}
           \begin{aligned} 
                \|q(t)-q_h^{\tau}(t)\|_{r',\Omega}&\lesssim \|q(t)-\eta_h(t)\|_{r',\Omega}+\|\eta_h(t)-q_h^{\tau}\|_{r',\Omega}\\[-0.5mm]&\lesssim \|q(t)-\eta_h(t)\|_{r',\Omega}+\sup_{\bfxi_h\in (\bfV_h\cap \bfV)\setminus \{\bfzero_d\}}{\bigg\{\frac{( \eta_h(t)-q_h^{\tau}(t),\textup{div} \bfxi_h)_{\Omega}}{\|\nabla\bfxi_h\|_{r,\Omega}}\bigg\}}
                \\[-0.5mm]&\lesssim  \|q(t)-\eta_h(t)\|_{r',\Omega}
                +\sup_{\bfxi_h\in (\bfV_h\cap \bfV)\setminus \{\bfzero_d\}}{\bigg\{\frac{(q(t)-q_h^{\tau}(t),\textup{div} \bfxi_h)_{\Omega}}{\|\nabla\bfxi_h\|_{r,\Omega}}\bigg\}}\,.
            \end{aligned}
        \end{align}
		Next, let $\bfxi_h\in \bfV_h \cap \bfV$ be fixed, but arbitrary. Then, 
        from~the~equivalent~weak and discrete formulations (\textit{cf}.\ Remarks \ref{rem:equiv_form} and \ref{rem:equiv_discrete_form}, respectively), 
        for a.e.\ $t\in I$,~we~find~that
		\begin{align}\label{eq:main.2}
			\begin{aligned}
				(q_h^{\tau}(t)-q(t),\textup{div}\,\bfxi_h)_{\Omega}
				&=(\bfS(\bfvarepsilon(\bfv_h^{\tau})(t)) - \bfS(\bfvarepsilon(\bfv)(t)),\bfvarepsilon(\bfxi_h))_{\Omega}
                \\[-0.5mm]&\quad+(\bff(t)-\bff^\tau(t), \bfxi_h)_{\Omega}\\&\quad+(\mathrm{d}_\tau \bfv_h^\tau(t)-\partial_t \bfv(t), \bfxi_h)_{\Omega} 
				\\[-0.5mm]&\eqqcolon \smash{I_{\tau,h}^1(t)+I_{\tau,h}^2(t)+I_{\tau,h}^3(t)}\,.
			\end{aligned}
		\end{align}
Thus, we need to estimate $\smash{I_{\tau,h}^m(t)}$, $m=1,2$, for all $\bfxi_h\in (\bfV_h \cap \bfV)\setminus\{\bfzero_d\}$~and~a.e.~$t\in I$:

        \textit{ad $I_{\tau,h}^i(t)$, $i=1,2$.}  
        By Hölder's inequality, for a.e.\ $t\in I$, we have that\vspace{-0.5mm}
        \begin{align}\label{eq:main.3.1}
            \vert I_{\tau,h}^1(t)\vert&\lesssim \|\bfS(\bfvarepsilon(\bfv_h^{\tau})(t)) - \bfS(\bfvarepsilon(\bfv)(t))\|_{r',\Omega}\|\nabla\bfxi_h\|_{r,\Omega}\,,\\[-0.25mm]
            \vert I_{\tau,h}^2(t)\vert&\lesssim \|\bff^\tau(t)- \bff(t)\|_{(\bfV_h)^*}\|\nabla\bfxi_h\|_{r,\Omega}\,.\label{eq:main.3.2}
        \end{align} 
        
        \textit{ad $I_{\tau,h}^3(t)$.} We add $\pm\mathcal{P}_h\partial_t \bfv(t)$ and use  $\partial_t \bfv(t)- \mathcal{P}_h\partial_t \bfv(t)\perp_{\bfL^2} \mathcal{P}_h \bfxi_h$ for a.e.\ $t\in I$, 
        to arrive at\vspace{-0.5mm}
        \begin{align}\label{eq:main.4}
            \begin{aligned} 
                I_{\tau,h}^3(t)&=( \mathcal{P}_h\partial_t \bfv(t)-\partial_t \bfv(t), \mathcal{P}_h^{\perp} \bfxi_h)_{\Omega}
                +(\mathrm{d}_\tau \bfv_h^\tau(t)-\mathcal{P}_h\partial_t \bfv(t), \bfxi_h)_{\Omega}
                \\[-0.5mm]& \eqqcolon \smash{I_{\tau,h}^{2,1}(t)+ I_{\tau,h}^{2,2}(t)}\,,
            \end{aligned}
        \end{align}
        so that it is left to estimate $\smash{I_{\tau,h}^{3,i}(t)}$, $i=1,2$,~for~a.e.~$t\in I$: 

        \textit{ad $I_{\tau,h}^{3,1}(t)$.}
        Using that $\mathcal{P}_h\partial_t \bfv(t)\perp_{\bfL^2} \mathcal{P}_h^{\perp} \bfxi_h$ and $\partial_t \bfv(t)\perp_{\bfL^2} \mathcal{P}^{\perp} \bfxi_h$ for a.e.\ $t\in I$, Hölder's inequality, and Lemma \ref{lem:operator-conv}, we find that
        \begin{align}\label{eq:main.5}
            \begin{aligned} 
                \smash{\vert I_{\tau,h}^{3,1}(t)\vert}
                &=\smash{\vert(\partial_t \bfv(t), \smash{\mathcal{P}_h^{\perp}} \bfxi_h-\mathcal{P}^{\perp}\bfxi_h)_{\Omega}\vert}
                \\[-0.5mm]&\leq \|\partial_t \bfv(t)\|_{r',\Omega}\|\smash{\mathcal{P}_h^{\perp}} \bfxi_h-\smash{\mathcal{P}^{\perp}}\bfxi_h\|_{r,\Omega}
                \\[-0.5mm]& \lesssim h\,\|\partial_t \bfv(t)\|_{r',\Omega}\|\nabla\bfxi_h\|_{r,\Omega}\,.
            \end{aligned}
        \end{align}
        
        \textit{ad $I_{\tau,h}^{3,2}(t)$.} Using the equivalent weak and discrete formulations (\textit{cf}.\ Remarks \ref{rem:equiv_form} and \ref{rem:equiv_discrete_form}, respectively) together with
        \begin{align*}
            \text{for a.e.\ } t\in I\colon  \quad(q_h^\tau(t), \textup{div}\,\mathcal{P}_h\bfxi_h)_{\Omega}&=0= (\eta_h(t), \textup{div}\,\mathcal{P}_h\bfxi_h)_{\Omega}\,,\\[-0.5mm]
            \text{for a.e.\ } t\in I\colon \hspace{0.5mm}\quad(\lambda_h^\tau(t), \mathcal{P}_h\bfxi_h\cdot\bfn )_{\Gamma}&=0= (\mu_h(t), \mathcal{P}_h\bfxi_h\cdot \bfn)_{\Gamma}\,,
        \end{align*}
        Hölder's inequality, Lemma \ref{lem:normal_trace_estimate}\eqref{lem:normal_trace_estimate.4} (if $\mathbb{P}^1(\mathcal{S}_h^{\Gamma})\subseteq \smash{\widehat{Z}_h}$), and Lemma \ref{lem:unconstrained-stab}, we obtain
        \begin{align}\label{eq:main.6}
            \hspace{-5mm}\begin{aligned} 
            I_{\tau,h}^{3,2}&=(\mathrm{d}_\tau \bfv_h^\tau(t)-\partial_t \bfv(t), \mathcal{P}_h\bfxi_h)_{\Omega}
            \\[-0.25mm]&=((\eta_h(t)-q(t))\mathbb{I}_{d\times d}, \bfvarepsilon(\mathcal{P}_h\bfxi_h))_{\Omega}
             -(\mu_h(t)-\lambda(t), \mathcal{P}_h\bfxi_h\cdot\bfn)_{\Gamma}
            \\[-0.25mm]&\quad+( \bfS(\bfvarepsilon(\bfv)(t))-\bfS(\bfvarepsilon(\bfv_h^{\tau})(t)),\bfvarepsilon(\mathcal{P}_h\bfxi_h))_{\Omega}+ (\bff^\tau(t)-\bff(t),\mathcal{P}_h\bfxi_h)_{\Omega}
           \\[-0.25mm]&\lesssim\left\{\begin{aligned}
                & \|q(t)- \eta_h(t)\|_{r',\Omega}+\|\lambda(t)- \mu_h(t)\|_{r',\Gamma}
                \\[-0.25mm]&+ \|\bfS(\bfvarepsilon(\bfv_h^{\tau})(t)) - \bfS(\bfvarepsilon(\bfv)(t))\|_{r',\Omega}+\|\bff^\tau(t)- \bff(t)\|_{(\bfV_h)^*}
            \end{aligned}\right\}
           \times \|\nabla\mathcal{P}_h\bfxi_h\|_{r,\Omega}
            \\[-0.25mm]&\lesssim \left\{\begin{aligned}
                & \| \eta_h(t)-q(t)\|_{r',\Omega}+\| \mu_h(t)-\lambda(t)\|_{r',\Gamma}
                \\[-0.25mm]&+ \|\bfS(\bfvarepsilon(\bfv_h^{\tau})(t)) - \bfS(\bfvarepsilon(\bfv)(t))\|_{r',\Omega}+\|\bff^\tau(t)- \bff(t)\|_{(\bfV_h)^*}
            \end{aligned}\right\}
            \times \|\nabla\bfxi_h\|_{r,\Omega}\,.
            \end{aligned}
        \end{align}
        Then, since for a.e.\ $t\in I$, $\eta_h(t)\in Q_h$ and  $\mu_h(t)\in \smash{\widehat{Z}_h}$ were  arbitrary, 
        from \eqref{eq:main.1}--\eqref{eq:main.6}, we conclude the assertion for the kinematic pressure.\newpage

\emph{2. (Quasi-)best-approximation result for the acceleration vector field.} For~a.e.~$t\in I$, let $\eta_h(t)\in Q_h$ and  $\mu_h(t)\in \smash{\widehat{Z}_h}$ be fixed, but arbitrary. Artificially introducing~the~projected~(by~applying~$\mathcal{P}_h$) analytic acceleration~vector~field~and~estimating the resulting projection error by means of Lemma \ref{lem:operator-conv_dual}, we find that
     \begin{align}\label{cor:main.1_}
            \begin{aligned}\|\partial_t \bfv(t)-\mathrm{d}_\tau \bfv_h^{\tau}(t)\|_{\smash{(\bfW^{1,r}_{\bfn}(\Omega))^*}}
            &\lesssim \|\mathcal{P}_h\partial_t \bfv(t)-\mathrm{d}_\tau \bfv_h^{\tau}(t)\|_{\smash{(\bfW^{1,r}_{\bfn}(\Omega))^*}} 
            + h\,\|\partial_t \bfv(t)\|_{r',\Omega}\,.
            \end{aligned}
     \end{align}
     Then, due to the $\bfL^2(\Omega)$-self-adjointness of $\mathcal{P}_h$, we have that
     \begin{align}\label{cor:main.1}
        \hspace{-1mm}\|\mathcal{P}_h\partial_t \bfv(t)-\mathrm{d}_\tau \bfv_h^{\tau}(t)\|_{\smash{(\bfW^{1,r}_{\bfn}(\Omega))^*}}&=  
            \sup_{\bfxi\in \bfW^{1,r}_{\bfn}(\Omega)\setminus\{\bfzero_d\}}{\bigg\{\frac{(\partial_t \bfv(t)-\mathrm{d}_\tau \bfv_h^{\tau}(t),\mathcal{P}_h\bfxi)_{\Omega}}{\|\nabla\bfxi\|_{r,\Omega}}\bigg\}}\,,
     \end{align}  
     where, 
     due to Remark \ref{rem:equiv_form}, Remark \ref{rem:equiv_discrete_form}, Lemma \ref{lem:unconstrained-stab}, Hölder's inequality, and Lemma~\ref{lem:normal_trace_estimate}\eqref{lem:normal_trace_estimate.4} (if $\mathbb{P}^1(\mathcal{S}_h^{\Gamma})\subseteq \widehat{Z}_h$),  for every $\bfxi\in \bfW^{1,r}_{\bfn}(\Omega)$ and a.e.\ $t\in I$, we find that
     \begin{align}\label{cor:main.2}
     \begin{aligned} 
        (\partial_t \bfv(t)-\mathrm{d}_\tau \bfv_h^{\tau}(t),\mathcal{P}_h\bfxi)_{\Omega}&=(q(t)-\eta_h(t),\textup{div}\,\mathcal{P}_h\bfxi)_{\Omega}
        \\&\quad -(\lambda(t)-\mu_h(t),\mathcal{P}_h\bfxi\cdot\bfn)_{\Gamma}
        \\&\quad + (\bfS(\bfvarepsilon(\bfv_h^\tau)(t))-\bfS(\bfvarepsilon(\bfv)(t)),\bfvarepsilon(\mathcal{P}_h\bfxi))_{\Omega}
        \\&\quad +(\bff(t)-\bff^\tau(t),\mathcal{P}_h\bfxi)_{\Omega}
        \\&\lesssim \left\{\begin{aligned}
                &\|q(t)- \eta_h(t)\|_{r',\Omega}
                \\&+\|\lambda(t)- \mu_h(t)\|_{r',\Gamma}
                \\&+\|\bfS(\bfvarepsilon(\bfv_h^{\tau})(t)) - \bfS(\bfvarepsilon(\bfv)(t))\|_{r',\Omega}
                \\&+\|\bff(t)-\bff^\tau(t)\|_{(\bfV_h)^*}
            \end{aligned}\right\}
            \times \|\nabla\bfxi\|_{r,\Omega}\,.
            \end{aligned}
     \end{align}
     From \eqref{cor:main.1_}--\eqref{cor:main.2}, we conclude the assertion for the acceleration vector field.  
\end{proof} 

\begin{proof}[Proof (of Corollary~\ref{cor:pressure_rates})]

    \emph{ad (\hyperlink{cor:pressure_rates.i}{i}).} 
   For $\eta_h(t)\hspace{-0.15em}\coloneqq\hspace{-0.15em} \Pi_h^Q q(t)-(\Pi_h^Q q(t),1)_{\Omega}\hspace{-0.15em}\in\hspace{-0.15em} Q_h$~and~${\mu_h(t)\hspace{-0.15em}\coloneqq \hspace{-0.15em}\pi_h^{\ell_\lambda} \lambda(t)\hspace{-0.15em}\in\hspace{-0.15em} \widehat{Z}_h}$ \hspace{-0.1mm}for \hspace{-0.1mm}a.e.\ \hspace{-0.1mm}$t\hspace{-0.05em}\in\hspace{-0.05em} I$ \hspace{-0.1mm}in \hspace{-0.1mm}Theorem \hspace{-0.1mm}\ref{thm:main} \hspace{-0.1mm}with \hspace{-0.1mm}$r\hspace{-0.05em}=\hspace{-0.05em}p$, integrating~\hspace{-0.1mm}with~\hspace{-0.1mm}\mbox{respect}~to~\hspace{-0.1mm}${t\hspace{-0.05em}\in\hspace{-0.05em} I}$ and using the approximation properties of $\Pi_h^Q$ (\textit{cf}.\ \cite[Thm.\ 18.16]{EG21}), $\pi_h^{\ell_{\lambda}}$ (\textit{cf}.\ \cite[Rem.\ 18.17]{EG21}), and $\Pi_\tau^0$ (together with $\|\cdot\|_{(\bfV_h)^*}\leq \|\cdot\|_{\smash{\widehat{\bfV}^*}}$), we find that
    \begin{align}\label{cor:pressure_rates.3}
        \hspace*{-2.5mm}\begin{aligned} 
        \|q-q_h^{\tau}\|_{p',\Omega_T}+\|\partial_t \bfv-\mathrm{d}_\tau \bfv_h^{\tau}\|_{\smash{L^{p'}(I;\bfV^*)}}
        &\lesssim h^{\beta_x}\|\vert \nabla^{\beta_x}q\vert \|_{p',\Omega_T}+h^{\gamma_x}\|\vert \nabla^{\gamma_x}\lambda\vert \|_{p',\Gamma_T}\\&
        \quad +\|\bfS(\bfvarepsilon(\bfv_h^{\tau})) - \bfS(\bfvarepsilon(\bfv))\|_{p',\Omega_T}\\&\quad+\tau^{\alpha_t}[\bff]_{\smash{N^{\alpha_t,p'}(I;\smash{\widehat{\bfV}^*})}}+h\,\| \partial_t \bfv\|_{p',\Omega_T}
        \,,
        \end{aligned}
    \end{align}
    where, by \cite[Lem.\ 4.6]{BBDR12}, we have that 
    \begin{align}\label{cor:pressure_rates.4}
        \|\bfS(\bfvarepsilon(\bfv_h^{\tau})) - \bfS(\bfvarepsilon(\bfv))\|_{p',\Omega_T}\lesssim 
        \begin{cases}
            \smash{\|\bfF(\bfvarepsilon(\bfv_h^{\tau})) - \bfF(\bfvarepsilon(\bfv))\|_{2,\Omega_T}^{\frac{2}{p'}}}&\text{ if }p\leq 2\,,\\[1mm]
           \left\{\begin{aligned}
               &\|\bfF(\bfvarepsilon(\bfv_h^{\tau})) - \bfF(\bfvarepsilon(\bfv))\|_{2,\Omega_T}\\[1mm]&\!\!\times
           \rho_{\varphi,\Omega_T}(\vert \bfvarepsilon(\bfv)\vert+\vert\bfvarepsilon(\bfv_h^\tau)\vert)^{\smash{\frac{2-p'}{2p'}}}\end{aligned}\right\}&\text{ if }p>2\,,
        \end{cases}
    \end{align}
    which together with Corollary \ref{cor:main.velocity} implies the assertion.

    \emph{ad (\hyperlink{cor:pressure_rates.ii}{ii}).}  If we proceed as for \eqref{cor:pressure_rates.3}, but use Theorem \ref{thm:main} with $r=2$ instead, using that
    $L^{p'}(\Omega)\hookrightarrow L^2(\Omega)$ (as $p'\geq 2$), we find that
    \begin{align*}
        \begin{aligned} 
        \|q-q_h^{\tau}\|_{2,\Omega_T}+\|\partial_t \bfv-\mathrm{d}_\tau \bfv_h^{\tau}\|_{\smash{L^2(I;(\bfW^{1,2}_\bfn(\Omega))^*)}}
        &\lesssim h^{\beta_x}\|\vert \nabla^{\beta_x}q\vert \|_{p',\Omega_T}+h^{\gamma_x}\|\vert \nabla^{\gamma_x}\lambda\vert \|_{p',\Gamma_T}\\&
        \quad +\|\bfS(\bfvarepsilon(\bfv_h^{\tau})) - \bfS(\bfvarepsilon(\bfv))\|_{2,\Omega_T}\\&\quad+\tau^{\alpha_t}[\bff]_{\smash{N^{\alpha_t,p'}(I;\smash{\widehat{\bfV}^*})}}+h\,\| \partial_t \bfv\|_{2,\Omega_T}
        \,,
        \end{aligned}
    \end{align*}
    where, due to  $\delta>0$ and $p\leq 2$, we have that $\vert \bfS(\bfA) - \bfS(\bfB)\vert^2\lesssim \delta^{2-p'}\vert\bfF(\bfA) - \bfF(\bfB)\vert^2$ for all $\bfA,\bfB\in \mathbb{R}^{d\times d}$ and, thus, we can use that
    \begin{align*}
         \smash{\|\bfS(\bfvarepsilon(\bfv_h^{\tau})) - \bfS(\bfvarepsilon(\bfv))\|_{2,\Omega_T}\lesssim \delta^{\smash{\frac{2-p'}{2}}}\|\bfF(\bfvarepsilon(\bfv_h^{\tau})) - \bfF(\bfvarepsilon(\bfv))\|_{2,\Omega_T}\,,}
    \end{align*}
    which together with Corollary \ref{cor:main.velocity} implies the assertion.
\end{proof}

\section{Numerical Experiments}\label{sec:experiments} 

\hspace{5mm}In this section, we review the theoretical findings~of
Section \ref{sec:error-estimate} via numerical experiments.

\subsection{Implementation details} 

\hspace{5mm}All experiments were conducted employing the finite element software \texttt{FEniCS}~(version~2019.1.0, \textit{cf}. \cite{fenics}). In the numerical experiments, we restrict to the case $d=2$ as well as to the (lowest order) Taylor--Hood element (\textit{cf}.\ \cite{TaylorHood1973}) for the approximation of the velocity vector field and the kinematic pressure, \textit{i.e.},~we~employ~the~discrete~spaces $\widehat{\bfV}_h\coloneqq (\mathbb{P}^2_c(\mathcal{T}_h))^2$ and $\widehat{Q}_h\coloneqq \mathbb{P}^1_c(\mathcal{T}_h)$,~so~that the Assumptions~\ref{ass:PiQ}~and~\ref{ass:proj-div}~are~satisfied.

In order to comply with the remaining assumptions in Section \ref{sec:disc-problem}, we distinguish two cases with regard to the imposition of the discrete impermeability condition \eqref{eq:bc.1}:

\begin{itemize}[noitemsep,topsep=2pt,labelwidth=\widthof{$\bullet$}]
    \item[$\bullet$] \emph{Strong imposition of \eqref{eq:bc.1}:}  Let $\bfV_h\coloneqq \smash{\widehat{\bfV}}_h\cap\bfV$ and $\smash{\widehat{Z}}_h\coloneqq\{0\}$,
    so that Assumption~\ref{ass:korn} (\textit{cf}. \mbox{Remark} \ref{rem:korn}(\hyperlink{ass:korn.ii}{ii})) and Assumption~\ref{ass:discrete_inf_sup_II} (\textit{cf}.\ Remark \ref{rem:discrete_inf_sup_II}(\hyperlink{rem:discrete_inf_sup_II.ii}{ii}))~are~met;

    \item[$\bullet$] \emph{Weak imposition of \eqref{eq:bc.1}:}  
   Let $\bfV_h  \coloneqq  \{ \bfxi_h \in  \widehat{\bfV}_h\mid  \forall\eta_h\in\widehat{Q}_h\colon(\textup{div}\,\bfxi_h,\eta_h)_{\Omega}=0\}$~and~$\widehat{Z}_h \coloneqq\mathbb{P}^1(\mathcal{S}_h^{\Gamma})$, so that  Assumption  \ref{ass:korn} (\textit{cf}.\ Remark \ref{rem:korn}(\hyperlink{ass:korn.iii}{iii})) and, due to $\mathbb{B}_{\mathscr{F}}^{\Gamma}(\mathcal{T}_h)\subseteq \widehat{\bfV}_h$,  Assumption~\ref{ass:discrete_inf_sup_II} (\textit{cf}.\  Remark~\ref{rem:discrete_inf_sup_II}(\hyperlink{rem:discrete_inf_sup_II.i}{i})) are met.
    
\end{itemize}

We approximate the iterates that piece-wise in time define the discrete solution 
$(\bfv_h^{\tau},q_h^{\tau},\lambda_h^{\tau})\in \mathbb{P}^0(\mathcal{I}_\tau^0;\bfV_{h,\textup{div}})\times \mathbb{P}^0(\mathcal{I}_\tau;Q_h) \times \mathbb{P}^0(\mathcal{I}_\tau;\widehat{Z}_h)$ (in the sense of~\mbox{Definition}~\ref{def:discrete_form}) employing the default Newton solver from \texttt{PETSc} (version 3.17.3, \textit{cf}.\ \cite{petsc}) with absolute 
tolerance $\texttt{tol}_{\textup{abs}} \coloneqq 1.0\times10^{-10}$ and relative 
tolerance $\texttt{tol}_{\textup{rel}} \coloneqq 1.0\times10^{-8}$. In doing so, for the solution of the linearized system, we apply a sparse direct
solver from \texttt{MUMPS} (version 5.5.0, \textit{cf}. \cite{mumps}).

\subsection{Experimental orders of convergence} 

\hspace{5mm}We consider the unsteady~\mbox{$p$-Stokes} equations \eqref{eq:p-SE} supplemented with impermeability \eqref{eq:bc.1} and perfect Navier~slip \eqref{eq:bc.2} boundary conditions, where $I\coloneqq (0,T)$,  $T\coloneqq 0.1$, $\Omega\coloneqq (0,1)^2$, and the  extra-stress tensor 
is of the form \eqref{def:A} with 
$\nu_0\coloneqq1$, $\delta\coloneqq 1.0\times 10^{-5}$,  and $p\in \{1.5,2.5\}$.

We \hspace{-0.1mm}compute \hspace{-0.1mm}data \hspace{-0.1mm}so \hspace{-0.1mm}that \hspace{-0.1mm}the \hspace{-0.1mm}velocity \hspace{-0.1mm}vector \hspace{-0.1mm}field \hspace{-0.1mm}$\bfv\colon \Omega_T\to\mathbb{R}^2$ \hspace{-0.1mm}and \hspace{-0.1mm}the~\hspace{-0.1mm}\mbox{kinematic}~\hspace{-0.1mm}pres\-sure $q\colon\Omega_T\to\mathbb{R}$ solving 
\eqref{eq:p-SE}--\eqref{eq:bc.2}, for every~${(t,x)=(t,x_1,x_2)\in \Omega_T}$,~are~given~via
\begin{subequations}\label{manufactured_solution}
\begin{align}
    \bfv(t,x)&\coloneqq t \times\vert x\vert^{2\smash{\frac{\alpha-1}{p}+1.0\times 10^{-2}}}(x_2,-x_1)\,,\label{manufactured_solution.1}\\[-0.5mm] q(t,x)&\coloneqq t\times c_q\big\{\vert x\vert^{\alpha - \smash{\frac{2}{p'}}+ 1.0\times 10^{-2}}-\big(\vert \cdot\vert^{\alpha - \smash{\frac{2}{p'}}+1.0\times 10^{-2}},1\big)_{\Omega}\big\}\,,\label{manufactured_solution.2}
\end{align}
\end{subequations}
where $\alpha\in \{0.5,1.0\}$, so that 
the regularity assumptions \eqref{eq:regularity.1}--\eqref{eq:regularity.4} with $\alpha_t=1.0$ and $\alpha=\alpha_x=\beta_x=\gamma_x$ are met, and  $c_q =1.0\times 10^{-3}$ if $p=1.5$ and $c_q =1.0\times 10^{3}$~if~$p=2.5$.

Starting \hspace{-0.1mm}with \hspace{-0.1mm}a \hspace{-0.1mm}triangulation \hspace{-0.1mm}$\mathcal{T}_{h_0}$, \hspace{-0.1mm}where \hspace{-0.1mm}$h_0 \hspace{-0.1em}=\hspace{-0.1em} \sqrt{2}$, 
\hspace{-0.1mm}consisting~\hspace{-0.1mm}of~\hspace{-0.1mm}two~\hspace{-0.1mm}triangles, \hspace{-0.1mm}refined \hspace{-0.1mm}triangulations \hspace{-0.1mm}$\{\mathcal{T}_{\smash{h_i}}\}_{i= 1,\ldots,7}$, \hspace{-0.1mm}where
\hspace{-0.1mm}$h_{i+1}\hspace{-0.1em} =\hspace{-0.1em}\frac{h_i}{2}$~\hspace{-0.1mm}for~\hspace{-0.1mm}all~\hspace{-0.1mm}${i\hspace{-0.1em}= \hspace{-0.1em}0,\ldots,6}$, \hspace{-0.1mm}are~\hspace{-0.1mm}\mbox{generated}~\hspace{-0.1mm}\mbox{using}~uniform \hspace{-0.1mm}mesh-refinement. \hspace{-0.1mm}The \hspace{-0.1mm}partitions \hspace{-0.1mm}$\{\mathcal{I}_{\tau_i}\}_{i=0,\ldots,7}$ \hspace{-0.1mm}and \hspace{-0.1mm}$\{\mathcal{I}_{\tau_i}^0\}_{i=0,\ldots,7}$~\hspace{-0.1mm}of~\hspace{-0.1mm}$I$~\hspace{-0.1mm}and~\hspace{-0.1mm}$(-\tau_i,T)$, $i=0,\ldots,7$, are defined as in Subsection \ref{subsec:time_discretization} with step-sizes $\tau_i\coloneqq T\times 2^{-i-2}$,~${i= 0,\ldots,7}$.

For \hspace{-0.15mm}measuring \hspace{-0.15mm}error \hspace{-0.15mm}decay \hspace{-0.15mm}rates, \hspace{-0.15mm}for \hspace{-0.15mm}errors \hspace{-0.15mm}$\texttt{err}_i\in \{\texttt{err}_{\bfv,i},\texttt{err}_{q,i}^{\scaleto{L^r}{4pt}}\}$, \hspace{-0.1mm}${i=0,\ldots,7}$,~where\vspace{-0.5mm}
\begin{align*}
\left.\begin{aligned} 
    \texttt{err}_{\bfv,i}&\coloneqq \|\bfv_{h_i}^{\tau_i}-\mathrm{I}_{\tau_i}^0\bfv\|_{L^2(I;\bfL^2(\Omega))}+\|\bfF(\bfvarepsilon(\bfv_{h_i}^{\tau_i}))-\bfF(\mathrm{I}_{\tau_i}^0\bfvarepsilon(\bfv))\|_{2,\Omega_T}\,,\\
    \texttt{err}_{q,i}^{\scaleto{L^r}{4pt}}&\coloneqq \|q-q_{h_i}^{\tau_i}\|_{r,\Omega_T}\,,\quad r\in \{2,p'\}\,,
    \end{aligned}\quad\right\}\quad i=0,\ldots,7\,,
\end{align*}
we compute  the experimental orders of convergence 
\begin{align*}
    \texttt{EOC}_i(\texttt{err}_i) \coloneqq \frac{\log(\frac{\texttt{err}_{i+1}}{\texttt{err}_i})}{\log(\frac{\tau_{i+1}+h_{i+1}}{\tau_i+h_i})}\,,\;i=0,\ldots,6\,,
\end{align*}
presented in 
Table \ref{tab:1} (if \eqref{eq:bc.1} is strongly imposed) and Table~\ref{tab:2} (if \eqref{eq:bc.1} is weakly imposed). 

We make the following observations:

\begin{itemize}[noitemsep,topsep=2pt,leftmargin=!,labelwidth=\widthof{$\bullet$}]
    \item[$\bullet$] For the velocity errors, we report the expected experimental orders of convergence of about $\texttt{EOC}(\texttt{err}_{\bfv,i})\approx \alpha\smash{\min\{1,\frac{p'}{2}\}}$, $i=1,\ldots,6$, (\textit{cf}.\ Corollary \ref{cor:main.velocity}), where the experimental orders of convergence are slightly higher than, but asymptotically~approaching, the expected ones if \eqref{eq:bc.1} is weakly imposed (\textit{cf}.\ Table~\ref{tab:2}).

    \item[$\bullet$] For the pressure errors, in the case $p=1.5$, we report the expected experimental orders of convergence of about $\texttt{EOC}(\texttt{err}_{q,i}^{\smash{\scaleto{L^{p'}}{7pt}}})\approx \alpha\frac{2}{p'}$, $i=1,\ldots,6$, and ${\texttt{EOC}(\texttt{err}_{q,i}^{\smash{\scaleto{L^2}{6pt}}})\approx\alpha}$, ${i=1,\ldots,6}$, (\textit{cf}.\ Corollary \ref{cor:pressure_rates}) (if $\alpha=0.5$, 
    we report~the~(possibly~\mbox{pre-asymptotic})~increased experimental orders of convergence of about $\texttt{EOC}(\texttt{err}_{q,i}^{\smash{\scaleto{L^2}{6pt}}})\hspace{-0.1em}\approx \hspace{-0.1em}0.66\overline{6}$,~${i\hspace{-0.1em}=\hspace{-0.1em}1,\ldots,6}$).
    In the case $p=2.5$,~we~report~higher experimental orders of convergence than the ones expected~by~\mbox{Corollary}~\ref{cor:pressure_rates}.
\end{itemize}

\begin{table}[H]
	\setlength\tabcolsep{2.5pt}
	\centering
	\begin{tabular}{c|c|c|c|c|c|c|c|c|c|c|c|c|} \toprule 
		\multicolumn{1}{|c||}{\cellcolor{lightgray}\diagbox[height=0.6\line,width=0.05\dimexpr\linewidth]{\vspace{-1.5mm}\hspace{-2.5mm} $i$}{\\[-6mm] $p$\hspace{-1.5mm}}}
		& \multicolumn{6}{c||}{\cellcolor{lightgray}$1.5$} & \multicolumn{6}{c|}{\cellcolor{lightgray}$2.5$}
        \\ \toprule\toprule 
        \multicolumn{1}{|c||}{\cellcolor{lightgray}$\alpha$}	
		& \multicolumn{3}{c||}{\cellcolor{lightgray}$1.0$}   & \multicolumn{3}{c||}{\cellcolor{lightgray}$0.5$} & \multicolumn{3}{c|}{\cellcolor{lightgray}$1.0$} & \multicolumn{3}{c|}{\cellcolor{lightgray}$0.5$} \\ 
		\toprule\toprule
        \multicolumn{1}{|c||}{\cellcolor{lightgray}$\texttt{err}_i$} & $\texttt{err}_{\bfv,i}$ & $\texttt{err}_{\smash{q,i}}^{\smash{\scaleto{L^{p'}}{6pt}}}$ & \multicolumn{1}{c||}{$\texttt{err}_{\smash{q,i}}^{\smash{\scaleto{L^2}{5.25pt}}}$} & $\texttt{err}_{\bfv,i}$ & $\texttt{err}_{\smash{q,i}}^{\smash{\scaleto{L^{p'}}{6pt}}}$ & \multicolumn{1}{c||}{$\texttt{err}_{\smash{q,i}}^{\smash{\scaleto{L^2}{5.25pt}}}$} & $\texttt{err}_{\bfv,i}$ & $\texttt{err}_{\smash{q,i}}^{\smash{\scaleto{L^{p'}}{6pt}}}$ & \multicolumn{1}{c||}{$\texttt{err}_{\smash{q,i}}^{\smash{\scaleto{L^2}{5.25pt}}}$} & $\texttt{err}_{\bfv,i}$ & $\texttt{err}_{\smash{q,i}}^{\smash{\scaleto{L^{p'}}{6pt}}}$ & $\texttt{err}_{q,i}^{\smash{\scaleto{L^2}{5.5pt}}}$ \\ \toprule\toprule
        \multicolumn{1}{|c||}{\cellcolor{lightgray}$1$} & 1.111 & 0.777 & \multicolumn{1}{c||}{0.989} & 0.731 & 0.457 & \multicolumn{1}{c||}{0.718} & 0.847 & 1.031 & \multicolumn{1}{c||}{0.852} & 0.539 & 0.572 & 0.366 \\ \hline
		\multicolumn{1}{|c||}{\cellcolor{lightgray}$2$} & 1.069 & 0.719 & \multicolumn{1}{c||}{1.020} & 0.649 & 0.379 & \multicolumn{1}{c||}{0.680} & 0.966 & 1.036 & \multicolumn{1}{c||}{0.841} & 0.540 & 0.544 & 0.342 \\ \hline
		\multicolumn{1}{|c||}{\cellcolor{lightgray}$3$} & 1.041 & 0.697 & \multicolumn{1}{c||}{1.022} & 0.593 & 0.357 & \multicolumn{1}{c||}{0.672} & 0.924 & 1.025 & \multicolumn{1}{c||}{0.826} & 0.487 & 0.528 & 0.327 \\ \hline
		\multicolumn{1}{|c||}{\cellcolor{lightgray}$4$} & 1.025 & 0.684 & \multicolumn{1}{c||}{1.016} & 0.558 & 0.348 & \multicolumn{1}{c||}{0.670} & 0.888 & 1.018 & \multicolumn{1}{c||}{0.818} & 0.456 & 0.519 & 0.318 \\ \hline
		\multicolumn{1}{|c||}{\cellcolor{lightgray}$5$} & 1.017 & 0.677 & \multicolumn{1}{c||}{1.011} & 0.537 & 0.343 & \multicolumn{1}{c||}{0.670} & 0.866 & 1.014 & \multicolumn{1}{c||}{0.814} & 0.440 & 0.514 & 0.314 \\ \hline
		\multicolumn{1}{|c||}{\cellcolor{lightgray}$6$} & 1.013 & 0.674 & \multicolumn{1}{c||}{1.008} & 0.525 & 0.341 & \multicolumn{1}{c||}{0.671} & 0.854 & 1.012 & \multicolumn{1}{c||}{0.812} & 0.432 & 0.512 & 0.312 \\  \hline\hline  
		\multicolumn{1}{|c||}{\cellcolor{lightgray}\small theory} & \cellcolor{lightgray!25!white}$1.000$ & \cellcolor{lightgray!25!white}$0.66\overline{6}$ & \multicolumn{1}{c||}{\cellcolor{lightgray!25!white}$1.000$} & \cellcolor{lightgray!25!white}$0.500$ & \cellcolor{lightgray!25!white}$0.33\overline{3}$ & \multicolumn{1}{c||}{\cellcolor{lightgray!25!white}$0.500$} & \cellcolor{lightgray!25!white}$0.83\overline{3}$ & \cellcolor{lightgray!25!white}$0.83\overline{3}$ & \multicolumn{1}{c||}{ \cellcolor{lightgray!25!white}---\, } & \cellcolor{lightgray!25!white}$0.41\overline{6}$ & \cellcolor{lightgray!25!white}$0.41\overline{6}$ & \cellcolor{lightgray!25!white}---\, \\ \toprule
		\end{tabular}\vspace{-2mm}
	\caption{\hspace{-0.15mm}$\texttt{EOC}_i(\texttt{err}_i\hspace{-0.175em}\in\hspace{-0.175em} \{\texttt{err}_{\bfv,i},\texttt{err}_{q,i}^{\scaleto{L^r}{4pt}}\})$, \hspace{-0.15mm}$r\hspace{-0.175em}\in\hspace{-0.175em}  \{2,p,p'\}$, \hspace{-0.15mm}$i\hspace{-0.12em}=\hspace{-0.12em}1,\ldots,6$;~\hspace{-0.15mm}\eqref{eq:bc.1}~\hspace{-0.15mm}strongly~\hspace{-0.15mm}\mbox{imposed}.}
	\label{tab:1} 
	\end{table}\vspace{-5mm}

    \begin{table}[H]
	\setlength\tabcolsep{2.5pt}
	\centering
	\begin{tabular}{c|c|c|c|c|c|c|c|c|c|c|c|c|} \toprule 
		\multicolumn{1}{|c||}{\cellcolor{lightgray}\diagbox[height=0.6\line,width=0.05\dimexpr\linewidth]{\vspace{-1.5mm}\hspace{-2.5mm} $i$}{\\[-6mm] $p$\hspace{-1.5mm}}}
		& \multicolumn{6}{c||}{\cellcolor{lightgray}$1.5$} & \multicolumn{6}{c|}{\cellcolor{lightgray}$2.5$}
        \\ \toprule\toprule 
        \multicolumn{1}{|c||}{\cellcolor{lightgray}$\alpha$}	
		& \multicolumn{3}{c||}{\cellcolor{lightgray}$1.0$}   & \multicolumn{3}{c||}{\cellcolor{lightgray}$0.5$} & \multicolumn{3}{c|}{\cellcolor{lightgray}$1.0$} & \multicolumn{3}{c|}{\cellcolor{lightgray}$0.5$} \\ 
		\toprule\toprule
        \multicolumn{1}{|c||}{\cellcolor{lightgray}$\texttt{err}_i$} & $\texttt{err}_{\bfv,i}$ & $\texttt{err}_{q,i}^{\scaleto{L^{p'}}{6pt}}$ & \multicolumn{1}{c||}{$\texttt{err}_{q,i}^{\scaleto{L^2}{5.5pt}}$} & $\texttt{err}_{\bfv,i}$ & $\texttt{err}_{q,i}^{\scaleto{L^{p'}}{6pt}}$ & \multicolumn{1}{c||}{$\texttt{err}_{q,i}^{\scaleto{L^2}{5.5pt}}$} & $\texttt{err}_{\bfv,i}$ & $\texttt{err}_{q,i}^{\scaleto{L^{p'}}{6pt}}$ & \multicolumn{1}{c||}{$\texttt{err}_{q,i}^{\scaleto{L^2}{5.5pt}}$} & $\texttt{err}_{\bfv,i}$ & $\texttt{err}_{q,i}^{\scaleto{L^{p'}}{6pt}}$ & $\texttt{err}_{q,i}^{\scaleto{L^2}{5.5pt}}$ \\ \toprule\toprule
		 \multicolumn{1}{|c||}{\cellcolor{lightgray}$1$} & 1.051 & 0.945 & \multicolumn{1}{c||}{1.093} & 0.633 & 0.416 & \multicolumn{1}{c||}{0.675} & 1.065 & 1.034 & \multicolumn{1}{c||}{0.855} & 0.739 & 0.577 & 0.370 \\ \hline
		\multicolumn{1}{|c||}{\cellcolor{lightgray}$2$} & 1.046 & 0.721 & \multicolumn{1}{c||}{0.996} & 0.588 & 0.348 & \multicolumn{1}{c||}{0.640} & 1.071 & 1.039 & \multicolumn{1}{c||}{0.842} & 0.680 & 0.550 & 0.344 \\ \hline
		\multicolumn{1}{|c||}{\cellcolor{lightgray}$3$} & 1.028 & 0.680 & \multicolumn{1}{c||}{0.984} & 0.555 & 0.344 & \multicolumn{1}{c||}{0.652} & 1.006 & 1.027 & \multicolumn{1}{c||}{0.827} & 0.608 & 0.532 & 0.327 \\ \hline
		\multicolumn{1}{|c||}{\cellcolor{lightgray}$4$} & 1.017 & 0.668 & \multicolumn{1}{c||}{0.988} & 0.536 & 0.343 & \multicolumn{1}{c||}{0.662} & 0.950 & 1.020 & \multicolumn{1}{c||}{0.818} & 0.552 & 0.522 & 0.319 \\ \hline
		\multicolumn{1}{|c||}{\cellcolor{lightgray}$5$} & 1.012 & 0.666 & \multicolumn{1}{c||}{0.993} & 0.524 & 0.341 & \multicolumn{1}{c||}{0.668} & 0.911 & 1.015 & \multicolumn{1}{c||}{0.814} & 0.512 & 0.517 & 0.314 \\ \hline
		\multicolumn{1}{|c||}{\cellcolor{lightgray}$6$} & 1.009 & 0.667 & \multicolumn{1}{c||}{0.997} & 0.517 & 0.340 & \multicolumn{1}{c||}{0.670} & 0.885 & 1.013 & \multicolumn{1}{c||}{0.812} & 0.483 & 0.514 & 0.312 \\  \hline\hline  
		\multicolumn{1}{|c||}{\cellcolor{lightgray}\small theory} &  \cellcolor{lightgray!25!white}$1.000$ & \cellcolor{lightgray!25!white}$0.66\overline{6}$ & \multicolumn{1}{c||}{\cellcolor{lightgray!25!white}$1.000$} & 
        \cellcolor{lightgray!25!white}$0.500$ & \cellcolor{lightgray!25!white}$0.33\overline{3}$ & \multicolumn{1}{c||}{\cellcolor{lightgray!25!white}$0.500$} & \cellcolor{lightgray!25!white}$0.83\overline{3}$ & \cellcolor{lightgray!25!white}$0.83\overline{3}$ & \multicolumn{1}{c||}{\cellcolor{lightgray!25!white} --- } & \cellcolor{lightgray!25!white}$0.41\overline{6}$ & \cellcolor{lightgray!25!white}$0.41\overline{6}$ & \cellcolor{lightgray!25!white}--- \\ \toprule
		\end{tabular}\vspace{-2mm}
	\caption{\hspace{-0.15mm}$\texttt{EOC}_i(\texttt{err}_i\hspace{-0.175em}\in\hspace{-0.175em} \{\texttt{err}_{\bfv,i},\texttt{err}_{q,i}^{\scaleto{L^r}{4pt}}\})$, \hspace{-0.15mm}$r\hspace{-0.175em}\in\hspace{-0.175em}  \{2,p,p'\}$, \hspace{-0.15mm}$i\hspace{-0.1em}=\hspace{-0.1em}1,\ldots,6$;~\hspace{-0.15mm}\eqref{eq:bc.1}~\hspace{-0.15mm}weakly~\hspace{-0.15mm}\mbox{imposed}.}
	\label{tab:2} 
	\end{table}\vspace{-5mm}

    \subsection{$\bfL^r(\Omega)$-stability test for $\mathcal{P}_h$ and $\mathcal{P}_h^{\perp}$}\label{subsec:stab}

   \hspace{5mm}For triangulations $\{\mathcal{T}_{\smash{h_i}}\}_{i=1,\ldots,39}$, each obtained by partitioning $\Omega\coloneqq (0,1)^2$~into~$i^2$~equal~squares and subdividing~each~square along a diagonal,  
   $r\in \{2,p,p'\}$, and ${\mathcal{J}_{\smash{h_i}}\in  \{\mathcal{P}_{\smash{h_i}},\mathcal{P}_{\smash{h_i}}^{\perp}\}}$, $i=1,\ldots,39$, we compute
    \begin{align*}
        c_{\textup{stab}}^{i}(\mathcal{J}_{\smash{h_i}})&\coloneqq \max_{j=1,\ldots,\textup{dim}(\smash{\widehat{\bfV}}_{\smash{h_i}})}{\left\{\frac{\|\mathcal{J}_{\smash{h_i}}\mathcal{P}_{\bfV_{\smash{h_i}}}\boldsymbol{\phi}_{h_i}^{j}\|_{r,\Omega}}{\|\mathcal{P}_{\bfV_{\smash{h_i}}}\boldsymbol{\phi}_{h_i}^{j}\|_{r,\Omega}}\right\}}\,,\;i=1,\ldots,39\,,
    \end{align*}
    where, for every $i=1,\ldots,39$, the set
     $\smash{\{\boldsymbol{\phi}_{h_i}^{j}\}_{j=1,\ldots,\textup{dim}(\smash{\widehat{\bfV}}_{\smash{h_i}})}}\subseteq \smash{\widehat{\bfV}}_{\smash{h_i}}$ denotes~the~shape~\mbox{basis}~of~$\smash{\widehat{\bfV}}_{\smash{h_i}}$, presented in Figure \ref{fig:stab}. In it, for each $r\in \{2,p,p'\}$, we report that~$c_{\textup{stab}}^{i}(\mathcal{P}_{\smash{h_i}})\approx 1$, ${i=1,\ldots,39}$, and that $c_{\textup{stab}}^{i}(\mathcal{P}_{\smash{h_i}}^{\perp})\leq 1$, $i=1,\ldots,39$, which is an indication for that, in the numerical experiments of the previous subsection,  Assumption \ref{ass:Lr-stab-leray-dis}~was~indeed~satisfied.\vspace{-2mm}

    \begin{figure}[H]
        \centering
        \includegraphics[width=0.5\linewidth]{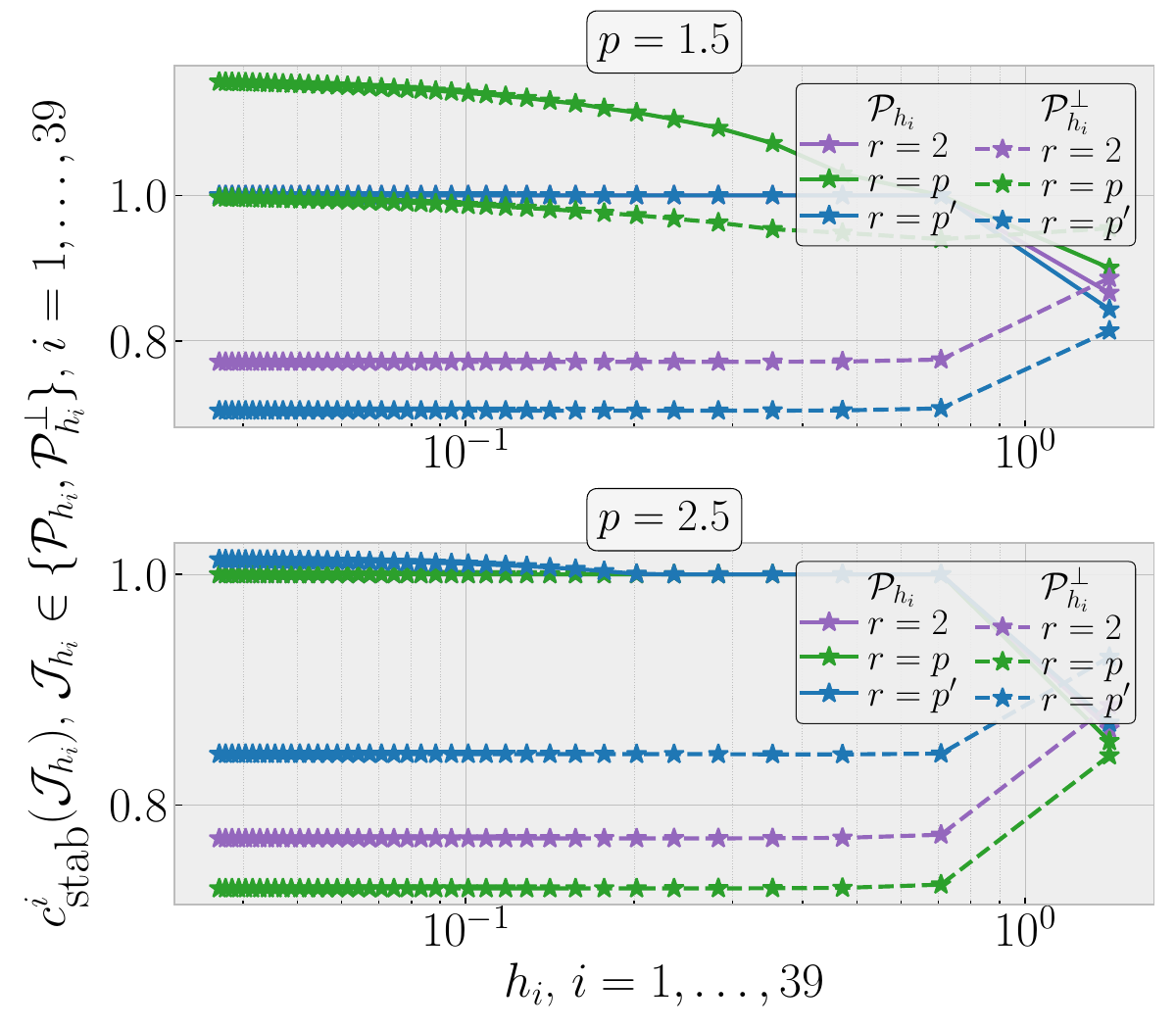}\includegraphics[width=0.5\linewidth]{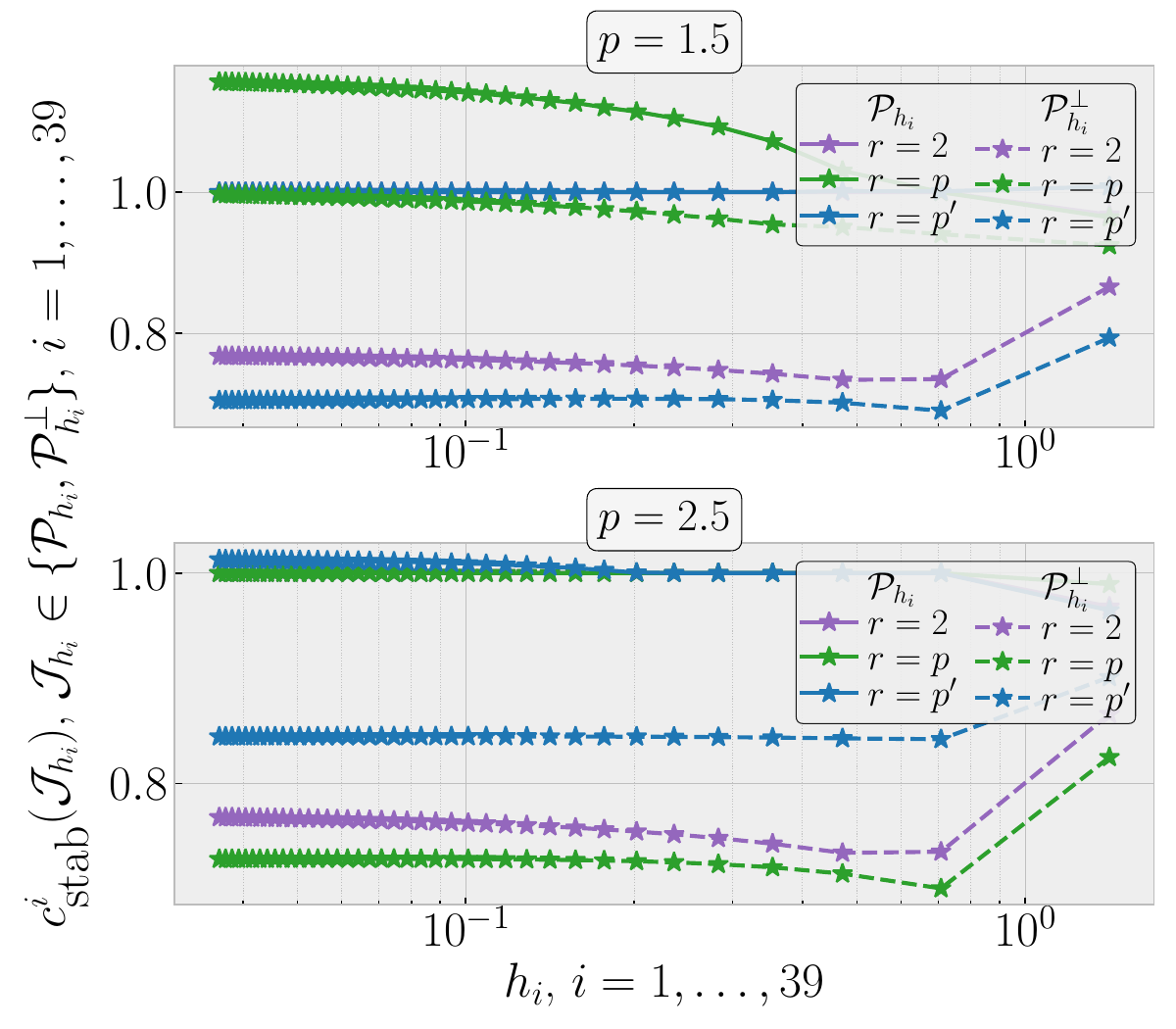}\vspace{-2.5mm}
        \caption{Computed $c_{\textup{stab}}^{i}(\mathcal{J}_{\smash{h_i}})$, $\mathcal{J}_{\smash{h_i}}\in  \{\mathcal{P}_{\smash{h_i}},\mathcal{P}_{\smash{h_i}}^{\perp}\}$, $i=1,\ldots,39$: \textit{left:} \eqref{eq:bc.1} is strongly imposed; \textit{right:} \eqref{eq:bc.1} is weakly imposed.}
        \label{fig:stab}
    \end{figure}

{\setlength{\bibsep}{0pt plus 0.0ex}\small
		
		\bibliographystyle{aomplain}
		\bibliography{references} 
		
}
\end{document}